\documentclass{article}
\usepackage{geometry}
\usepackage{amsmath, amsthm, amsfonts, amssymb, enumitem, esint}
\usepackage[hidelinks]{hyperref}
\usepackage{tikz}
\usepackage[hidelinks]{hyperref}
\usetikzlibrary{backgrounds}
\graphicspath{ {./images/} }

\numberwithin{equation}{section}

\newtheorem{theorem}{Theorem}[section]
\newtheorem{lemma}[theorem]{Lemma}
\newtheorem{proposition}[theorem]{Proposition}
\newtheorem{corollary}[theorem]{Corollary}
\theoremstyle{definition}
\newtheorem{definition}[theorem]{Definition}
\theoremstyle{remark}
\newtheorem{remark}{Remark}[theorem]

\newcommand{\card}{\operatorname{card}}

\newcommand{\cpct}{\operatorname{Cap}}
\newcommand{\csubset}{\subset\joinrel\subset}

\newcommand{\supp}{\operatorname{supp}}
\newcommand{\topint}{\operatorname{int}}

\title{\textbf{Combined effect of homogenization and dimension-reduction in the random Neumann sieve problem}}
\author{\textbf{Mert Baştuğ}}
\date{}

\begin{document}

\maketitle

\begin{abstract}
    We investigate the asymptotic behavior of the solutions to the Neumann sieve problem for the Poisson equation in a thin, randomly perforated domain. The perforations (sieve-holes) are generated by a stationary marked point process. According to the scaling between the domain thickness and the typical hole size, three distinct limiting regimes emerge. We also identify the optimal stochastic integrability condition on the random hole radii that guarantees stochastic homogenization, even in the presence of clustering holes.
\end{abstract}

\section{Introduction}

In this paper, we study Neumann's sieve problem on thin domains with random perforations. For a small parameter $\varepsilon > 0$ and a sequence of positive real numbers $(\delta_\varepsilon)$ with $\lim_{\varepsilon \to 0} \delta_\varepsilon = 0$, we consider the following sequence of boundary value problems in $\mathbb{R}^N$ with $N \ge 3$:
\begin{equation}
    \begin{cases}
        \begin{alignedat}{2} \label{eq:random_bvp}
            -\Delta u_\varepsilon &= f \quad && \text{in } U_\varepsilon, \\
            u_\varepsilon &= 0 \quad && \text{on } \partial U' \times (-\delta_\varepsilon, \delta_\varepsilon), \\
            \nabla u_\varepsilon \cdot \nu &= 0 \quad && \text{on } ((U' \setminus T'_\varepsilon) \times \{0\}) \cup (U' \times \{-\delta_\varepsilon, \delta_\varepsilon\}).
        \end{alignedat}
    \end{cases}
\end{equation}
Here, $U' \subset \mathbb{R}^{N - 1}$ is a bounded and open set with Lipschitz boundary, $f \in L^2(U')$, and $T'_\varepsilon \subset U'$ is an open set. The domain $U_\varepsilon$ is defined by
\[
U_\varepsilon := (U' \times (-\delta_\varepsilon, \delta_\varepsilon)) \setminus ((U' \setminus T'_\varepsilon) \times \{0\}).
\]
Hence, the two parts, $U' \times (0, \delta_\varepsilon)$ and $U' \times (-\delta_\varepsilon, 0)$, are connected through the set $T'_\varepsilon \times \{0\}$ representing the perforations (see Figure \ref{fig:domain}). The goal of the paper is to study the convergence of the rescaled functions $\hat{u}_\varepsilon(x', x_N) := u_\varepsilon(x', \delta_\varepsilon x_N)$ and to obtain a characterization of the limit in terms of an effective boundary value problem. As $\delta_\varepsilon \to 0$, we expect the limit function to depend only on $N - 1$ variables and the effective equation to be set on the domain $U'$.

\begin{figure}[t] \label{fig:domain}
    \centering
    \begin{tikzpicture}
        \node[anchor=south west,inner sep=0] (image) at (0,0) {\includegraphics[scale=0.27]{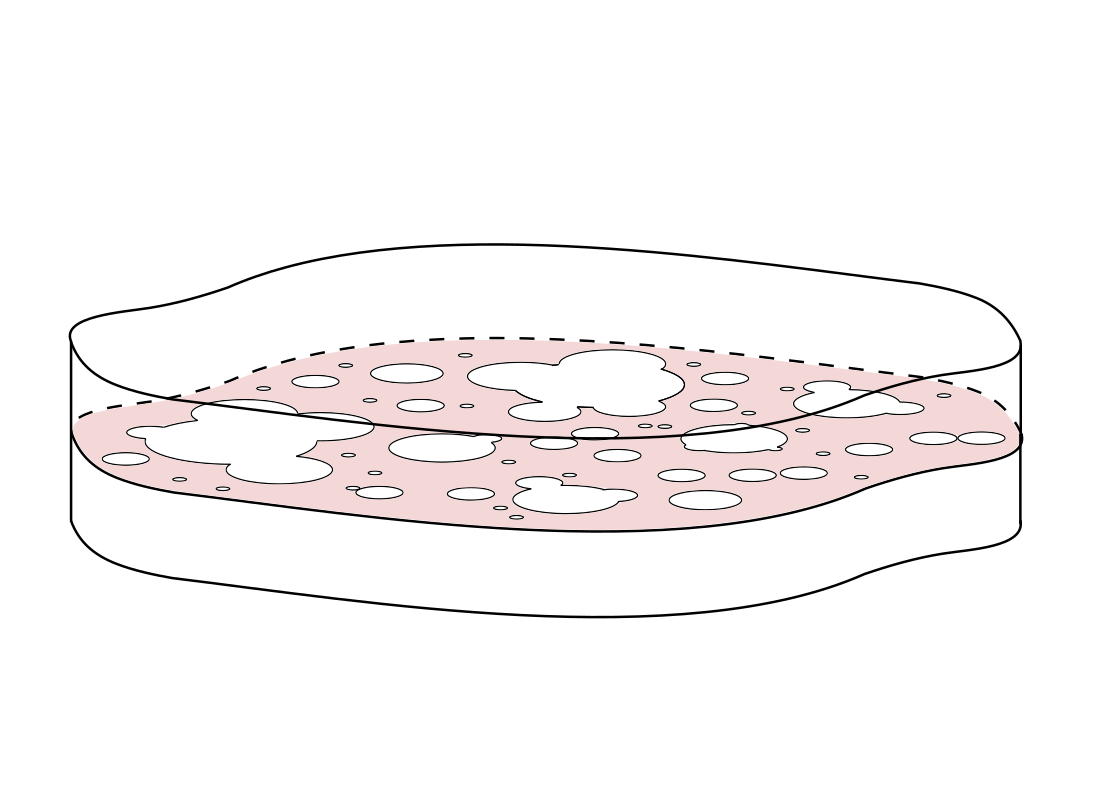}};
            \begin{scope}[x={(image.south east)},y={(image.north west)}]
                \draw (0, 0.5) node {$U_\varepsilon$};
            \end{scope}
    \end{tikzpicture}
    \quad
    \begin{tikzpicture}
        \node[anchor=south west,inner sep=0] (image) at (0,0) {\includegraphics[scale=0.27]{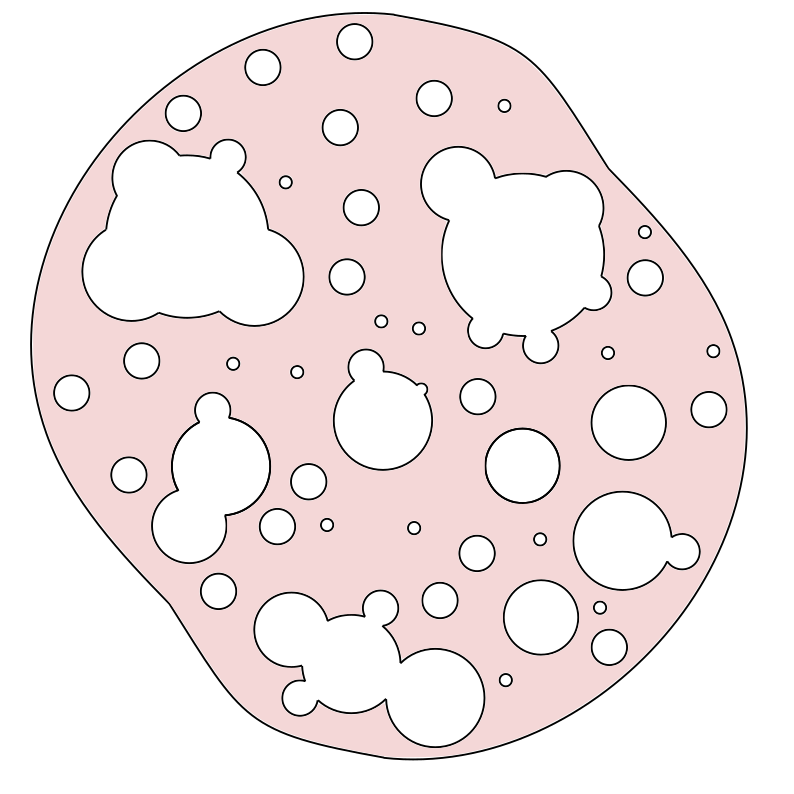}};
            \begin{scope}[x={(image.south east)},y={(image.north west)}]
                \draw (1, 0.8) node {$U' \setminus T'_\varepsilon$};
            \end{scope}
    \end{tikzpicture}
    \caption{Illustrations of the domain $U_\varepsilon$ and the cross-section $U' \setminus T'_\varepsilon$.}
\end{figure}

The asymptotic character of the sequence $(\hat{u}_\varepsilon)$ naturally depends on $T'_\varepsilon$. In this paper, we consider perforations obtained from scaled and translated copies of a model hole. Fix a bounded and open set $T' \subset \mathbb{R}^{N - 1}$. We choose a discrete, countably infinite set of random points $\{(y'_i, \rho_i)\}_{i \in \mathbb{N}} \subset \mathbb{R}^{N - 1} \times (0, \infty)$, and we set
\[
T'_\varepsilon := \bigcup_{\varepsilon y_i' \in U'} (\varepsilon y_i' + a_\varepsilon \rho_i T') \cap U'.
\]
Thus, the position and size of each copy of $T'$ in $T'_\varepsilon$ is random, with $(y'_i)$ determining the centers and $(\rho_i)$ the sizes of the copies. The centers are scaled by $\varepsilon$ and the radii by a positive number $a_\varepsilon$ with $a_\varepsilon \ll \varepsilon$, so that the number of holes in $T'_\varepsilon$ increases while their sizes shrink as $\varepsilon \to 0$. The choice of the random points $\{(y'_i, \rho_i)\}_{i \in \mathbb{N}}$ is made precise with the theory of marked point processes.

The domain $U_\varepsilon$ is a model for two thin films, represented by $U' \times (0, \delta_\varepsilon)$ and $U' \times (-\delta_\varepsilon, 0)$, connected by a randomly generated contact region $T'_\varepsilon \times \{0\}$. In \cite{ABZ}, the authors prove that, for periodically distributed contact regions, the homogenization of the nonlinear elastic energy of a deformation of $U_\varepsilon$ yields an additional interfacial energy term, which results from the penalization of debonding at the interface, in agreement with the phenomenological model introduced in \cite{BFF}. In our paper, we consider a simplification of the problem by replacing the nonlinear elastic energy with the Dirichlet energy, as we are mainly interested in observing the effect of randomizing the contact region. From now on, instead of holes, we shall refer to contact regions or connection regions.

For $\varepsilon$, $\delta_\varepsilon$ and $a_\varepsilon$ satisfying certain asymptotic relationships as $\varepsilon \to 0$ and under mild conditions on the probability law of the random points $\{(y'_i, \rho_i)\}_{i \in \mathbb{N}}$, we show that there exist functions $u^+, u^- \in H_0^1(U')$ and a real-valued random variable $\gamma$, such that, almost surely,
\begin{equation} \label{eq:rescaled_convergence}
    \hat{u}_\varepsilon \rightharpoonup u^+ \text{ in } H^1(U' \times (0, 1)), \quad \hat{u}_\varepsilon \rightharpoonup u^- \text{ in } H^1(U' \times (-1, 0)),
\end{equation}
and the functions $u^+$ and $u^-$ solve the following coupled boundary value problems:
\begin{equation} \label{eq:limit_equations}
    \begin{cases}
        \begin{aligned}
            -\Delta u^+ + \frac{\gamma}{2}(u^+ - u^-) &= f \quad \text{in } U', \\
            u^+ &= 0 \quad \text{on } \partial U',
        \end{aligned}
    \end{cases}
    \quad
    \begin{cases}
        \begin{aligned}
            -\Delta u^- - \frac{\gamma}{2}(u^+ - u^-) &= f \quad \text{in } U', \\
            u^- &= 0 \quad \text{on } \partial U'.
        \end{aligned}
    \end{cases}    
\end{equation}
If $a_\varepsilon$ is either too small or too large relative to the domain height $\delta_\varepsilon$ and the average distance $\varepsilon$ between the centers of neighboring contact regions, then the effective equations are trivial, corresponding to $\gamma = 0$ or $\gamma = \infty$. Hence, in both cases $u^+ = u^-$ and $-\Delta u^+ = -\Delta u^- = f$. Given $(\delta_\varepsilon)$, we show, however, that there exists a critical scaling for $(a_\varepsilon)$ such that the homogenized equations are nontrivial. For the following discussion, we assume $N \ge 4$. We comment on the case $N = 3$ at the end. The critical scaling depends on the growth of a certain capacity function with respect to the relative thickness of the domain, which we now define. For $h > 0$, set
\[
C_h^\infty(\mathbb{R}^{N - 1}) := \left\{v \in C^\infty(\overline{\mathbb{R}^{N - 1} \times (-h, h)}) : \supp(v) \csubset \mathbb{R}^N\right\}.
\]
Hence, if $v \in C_h^\infty(\mathbb{R}^{N - 1})$, then it is smooth up to the boundary of $\mathbb{R}^{N - 1} \times (-h, h)$, and it vanishes at infinity. We note that $v$ does not have to vanish completely on the boundary. For $K' \subset \mathbb{R}^{N - 1}$ bounded, we define
\begin{equation} \label{eq:height_dependent_capacity}
    \cpct_h(K', \mathbb{R}^{N - 1}) := \inf \left\{\int_{\mathbb{R}^{N - 1} \times (-h, h)} |\nabla v|^2 \, dx : v \in C_h^\infty(\mathbb{R}^{N - 1}), \, K' \times \{0\} \subset \topint(\{v = 1\}) \right\},
\end{equation}
where $\topint(\{v = 1\})$ is the interior of the level set $\{v = 1\}$. If $(\delta_\varepsilon)$ and $(a_\varepsilon)$ satisfy
\begin{equation} \label{eq:critical_scaling}
    \lim_{\varepsilon \to 0} \frac{a_\varepsilon^{N - 2}}{\varepsilon^{N - 1} \delta_\varepsilon} \cpct_{\frac{\delta_\varepsilon}{a_\varepsilon}}(T', \mathbb{R}^{N - 1}) \in (0, \infty),
\end{equation}
then the limit equations for \eqref{eq:random_bvp} are given by the system \eqref{eq:limit_equations} with $\gamma \in (0, \infty)$. Assuming the existence of the following limits for the moment, let us set
\begin{equation} \label{eq:a_bunch_of_capacities}
    \cpct(T' \times \{0\}, \mathbb{R}^N) := \lim_{h \to \infty} \cpct_h(T', \mathbb{R}^{N - 1}), \quad \cpct_0(T', \mathbb{R}^{N - 1}) := \lim_{h \to 0} \frac{1}{2h} \cpct_h(T', \mathbb{R}^{N - 1}).
\end{equation}
If all capacities are finite and nonzero, then, depending on the relative thickness $\delta_\varepsilon/a_\varepsilon$, the critical scaling \eqref{eq:critical_scaling} yields three regimes:
\begin{enumerate}[label=(\roman*)]
    \item $\delta_\varepsilon \gg a_\varepsilon$, \quad $\lim_{\varepsilon \to 0} \frac{a_\varepsilon^{N - 2}}{\varepsilon^{N - 1} \delta_\varepsilon} \in (0, \infty)$,
    \item $\delta_\varepsilon \sim a_\varepsilon$, \quad $\lim_{\varepsilon \to 0} \frac{a_\varepsilon^{N - 2}}{\varepsilon^{N - 1} \delta_\varepsilon} \in (0, \infty)$,
    \item  $\delta_\varepsilon \ll a_\varepsilon$, \quad $\lim_{\varepsilon \to 0} \frac{a_\varepsilon^{N - 3}}{\varepsilon^{N - 1}} \in (0, \infty)$.
\end{enumerate}
By an easy computation, we see that a sequence $(a_\varepsilon)$ satisfying (i), (ii) or (iii) exists if and only if $\delta_\varepsilon^{N - 3} \gg \varepsilon^{N - 1}$, $\delta_\varepsilon^{N - 3} \sim \varepsilon^{N - 1}$ or $\delta_\varepsilon^{N - 3} \ll \varepsilon^{N - 1}$, respectively. When $\{y'_i\}_{i \in \mathbb{N}} = \mathbb{Z}^{N - 1}$ and $\rho_i = \rho$ for some $\rho > 0$ and all $i \in \mathbb{N}$, that is, when the connection regions have a uniform size and are distributed periodically, the criticality conditions agree with those studied in \cite{N, ABZ}.

We now give a heuristic derivation of the limit equations \eqref{eq:limit_equations} and show how the criticality condition \eqref{eq:critical_scaling}, equivalently conditions (i), (ii) and (iii), naturally arises. We suppose that \eqref{eq:rescaled_convergence} holds for some $u^+$ and $u^-$. As $u_\varepsilon$ is approximated closely by $u^+$ and $u^-$, we may write the scaled Dirichlet energy of $u_\varepsilon$ as an approximate sum of two terms:
\[
\frac{1}{\delta_\varepsilon} \int_{U_\varepsilon} |\nabla u_\varepsilon|^2 \, dx \approx \left(\int_{U'} |\nabla u^+|^2 + |\nabla u^-|^2 \, dx' \right) + \frac{1}{\delta_\varepsilon}\mathrm{TE}_\varepsilon.
\]
We denote by $\mathrm{TE}_\varepsilon$ the transition energy, that is, the Dirichlet energy necessary to interpolate between $u^+$ and $u^-$ in a neighborhood of the connection regions. In particular, we have
\begin{equation} \label{eq:energy_partition}
    \frac{1}{\delta_\varepsilon} \int_{U_\varepsilon} \frac{1}{2} |\nabla u_\varepsilon|^2 - f u_\varepsilon \, dx \approx \left(\int_{U'} \frac{1}{2} |\nabla u^+|^2 + \frac{1}{2} |\nabla u^-|^2 - f(u^+ + u^-) \, dx' \right) + \frac{1}{2\delta_\varepsilon} \mathrm{TE}_\varepsilon.
\end{equation}
Since the left-hand side is minimized by $u_\varepsilon$ among all $H^1(U_\varepsilon)$ functions that vanish on $\partial U' \times (-\delta_\varepsilon, \delta_\varepsilon)$ and since the integral on the right-hand side does not depend on $\varepsilon$, we observe that $u_\varepsilon$ minimizes the transition energy as well. If $u^+$ and $u^-$ are continuous, then, in the vicinity of the fixed connection region $(\varepsilon y'_i + a_\varepsilon \rho_i T') \times \{0\}$, we can approximate $u_\varepsilon$ by $u^+(\varepsilon y'_i)$ in the upper half-space and by $u^-(\varepsilon y'_i)$ in the lower half-space. The minimal Dirichlet energy required to transition between these two values is approximately equal to
\begin{equation} \label{eq:approximate_transition_energy}
    \begin{aligned}
        \left(\frac{u^+(\varepsilon y'_i) - u^-(\varepsilon y'_i)}{2}\right)^2 \cpct_{\delta_\varepsilon}(\varepsilon y'_i + a_\varepsilon \rho_i T', &\mathbb{R}^{N - 1}) \\
        &= \frac{1}{4}(u^+(\varepsilon y'_i) - u^-(\varepsilon y'_i))^2 a_\varepsilon^{N - 2} \rho_i^{N - 2} \cpct_{\frac{\delta_\varepsilon}{a_\varepsilon \rho_i}}(T', \mathbb{R}^{N - 1}),
    \end{aligned}
\end{equation}
where we used the scaling property of the capacity in the last line. If we partition $U'$ into sets $Q'_1, \dots, Q'_n$ with sufficiently small but fixed diameters, then, summing up the energies over all contact regions in $T'_\varepsilon$ and grouping those falling in the same set $Q'_j$ yields
\begin{equation}
    \frac{1}{2\delta_\varepsilon} \mathrm{TE}_\varepsilon \approx \sum_{j = 1}^n \frac{1}{8} \fint_{Q'_j} |u^+ - u^-|^2 \, dx \sum_{\varepsilon y'_i \in Q'_j} \frac{a_\varepsilon^{N - 2}}{\delta_\varepsilon} \rho_i^{N - 2} \cpct_{\frac{\delta_\varepsilon}{a_\varepsilon \rho_i}}(T', \mathbb{R}^{N - 1})
\end{equation}
If we assume that, for any sufficiently nice set $Q' \subset \mathbb{R}^{N - 1}$, the spatial averages
\begin{equation} \label{eq:spatial_average_intro}
    \frac{1}{|Q'|}\sum_{\varepsilon y'_i \in Q'} \frac{1}{2} \frac{a_\varepsilon^{N - 2}}{\delta_\varepsilon} \rho_i^{N - 2} \cpct_{\frac{\delta_\varepsilon}{a_\varepsilon \rho_i}}(T', \mathbb{R}^{N - 1}) = \frac{1}{\left|\frac{1}{\varepsilon} Q'\right|} \sum_{\varepsilon y'_i \in Q'} \frac{1}{2}\frac{a_\varepsilon^{N - 2}}{\varepsilon^{N - 1} \delta_\varepsilon} \rho_i^{N - 2} \cpct_{\frac{\delta_\varepsilon}{a_\varepsilon \rho_i}}(T', \mathbb{R}^{N - 1})
\end{equation}
converge to a constant $\gamma$, then we obtain
\[
\frac{1}{2\delta_\varepsilon} \mathrm{TE}_\varepsilon \approx \frac{\gamma}{4}\int_{U' } |u^+ - u^-|^2 \, dx'.
\]
Going back to \eqref{eq:energy_partition}, we see as a consequence that the pair $(u^+, u^-)$ minimizes the energy
\[
\int_{U'} \frac{1}{2} |\nabla u^+|^2 + \frac{1}{2} |\nabla u^-|^2 - f(u^+ + u^-) + \frac{\gamma}{4}|u^+ - u^-|^2 \, dx'
\]
among all pairs in $H_0^1(U') \times H_0^1(U')$. The associated Euler-Lagrange equation is given by the coupled system of PDEs \eqref{eq:limit_equations}. Assume that the criticality conditions (i), (ii), or (iii) hold and recall \eqref{eq:a_bunch_of_capacities}. If we set $h_0 := \lim_{\varepsilon \to 0} \delta_\varepsilon/a_\varepsilon$ and define
\begin{equation*}
    J_{h_0}(\rho) :=
    \begin{cases}
        \rho^{N - 2} \cpct(T' \times \{0\}, \mathbb{R}^N) \quad &\text{if } h_0 = \infty, \\
        \rho^{N - 2} \cpct_\frac{h_0}{\rho}(T', \mathbb{R}^{N - 1}) \quad &\text{if } h_0 \in (0, \infty), \\
        2 \rho^{N - 3} \cpct_0(T', \mathbb{R}^{N - 1}) \quad &\text{if } h_0 = 0,
    \end{cases}
\end{equation*}
then we see that the spatial average \eqref{eq:spatial_average_intro} is roughly proportional to
\begin{equation} \label{eq:proportional_spatial_average}
    \frac{1}{\left|\frac{1}{\varepsilon} Q'\right|} \sum_{\varepsilon y'_i \in Q'} J_{h_0}(\rho_i).
\end{equation}
Although our heuristic derivation suggests the right limit problem, there are several technical difficulties associated with it, some of which stem from the randomness of the sieve geometry. A major problem is caused by contact regions that are too large or too close to each other, resulting in clusters. Our estimate \eqref{eq:approximate_transition_energy} of the transition energy in terms of the capacity requires the contact regions to be well-separated. However, unless we impose bounds on the distances $|y'_i - y'_j|$ for $i \neq j$ or on the radii $\rho_i$, the clusters occur with high probability. In our paper, we shall assume, roughly speaking, no more than the convergence of the averages \eqref{eq:proportional_spatial_average} for nice sets $Q' \subset \mathbb{R}^{N - 1}$ (see Theorem \ref{thm:ergodic_theorem_second_version} for details). When formulated in terms of marked point processes, the assumptions translate to the finiteness of the integral $\int_0^\infty J_{h_0}(\rho) \, d\lambda(\rho)$ for a stationary marked point process, where $\lambda$ is, informally, the probability distribution of the radii $(\rho_i)$. Nevertheless, based only on this assumption, we are able to show that the capacity of the clusters is negligible almost surely. Here, we follow an approach similar to the one in our previous paper \cite{B}.

For $N = 3$, the capacity function in \eqref{eq:height_dependent_capacity} is always zero due to the fact that the harmonic capacity vanishes on bounded subsets of $\mathbb{R}^2$. In this case, instead of \eqref{eq:approximate_transition_energy}, the correct transition energy is given by
\begin{equation} \label{eq:transition_energy_in_3_dimensions}
    \left(\frac{u^+(\varepsilon y'_i) - u^-(\varepsilon y'_i)}{2}\right)^2 \cpct((\varepsilon y'_i + a_\varepsilon \rho_i T') \times \{0\}, \mathbb{R}^3),
\end{equation}
where
\[
\cpct(K, \mathbb{R}^{N - 1}) := \inf \left\{\int_{\mathbb{R}^3} |\nabla v|^2 \, dx : v \in C_c^\infty(\mathbb{R}^3), \, K \subset \topint(\{v = 1\}) \right\}
\]
for $K \subset \mathbb{R}^3$ bounded. However, even then, \eqref{eq:transition_energy_in_3_dimensions} holds only when $\ln(\varepsilon/a_\varepsilon) \ll \delta_\varepsilon/a_\varepsilon$. When the relative thickness $\delta_\varepsilon/a_\varepsilon$ is asymptotically comparable to or smaller than $\ln(\varepsilon/a_\varepsilon)$, two-dimensional effects dominate, and the minimum transition energy is no longer given by the harmonic capacity. Hence, for $N = 3$, we consider the critical scaling given by
\[
\ln\left(\frac{\varepsilon}{a_\varepsilon}\right) \ll \frac{\delta_\varepsilon}{a_\varepsilon}, \quad \lim_{\varepsilon \to 0} \frac{a_\varepsilon}{\varepsilon^2 \delta_\varepsilon} \in (0, \infty).
\]
Given $(\delta_\varepsilon)$, a sequence $(a_\varepsilon)$ satisfying the criticality condition exists if and only if $\ln(1/\delta_\varepsilon) \ll 1/\varepsilon^2$.

We now describe the ideas of the proof of our homogenization result. The main step in the proof is the construction of a sequence $(w_\varepsilon)$ with $w_\varepsilon \in H^1(U_\varepsilon)$ such that the rescaled functions $\hat{w}_\varepsilon(x', x_N) := w_\varepsilon(x', \delta_\varepsilon x_N)$ converge weakly to $1$ in $H^1(U' \times (0, 1))$ and to $-1$ in $H^1(U' \times (-1, 0))$ almost surely. Given any $v^+, v^- \in H_0^1(U')$, the functions $(w_\varepsilon)$ allow us to produce a sequence $(v_\varepsilon)$ with $v_\varepsilon \in H^1(U_\varepsilon)$ whose scaled versions converge weakly to $v^+$ and $v^-$ in the upper and lower half-spaces, respectively. In turn, we use $(v_\varepsilon)$ as test functions in the weak formulation of \eqref{eq:random_bvp} and, granted that the Laplacians $(\Delta v_\varepsilon)$ converge, we pass to the limit. As we can choose $v^+$ and $v^-$ arbitrarily, in this way, we obtain a complete characterization of the Laplacians of $u^+$ and $u^-$. This is the essence of Tartar's \textit{method of oscillating test functions}. We construct $(w_\varepsilon)$ that minimize the transition energy near the connection regions, thereby ensuring that $w_\varepsilon$ approximates the solution of a PDE. This is the crucial ingredient for the Laplacians $(\Delta v_\varepsilon)$ to converge.

Neumann's sieve problem on thin domains with periodic perforations was studied previously in \cite{N} and \cite{ABZ} using $\Gamma$-convergence. Our result extends these results to randomly generated perforations in the case of Poisson's equation. For other problems related to Neumann's sieve and homogenization problems in similar geometries, we refer the reader to \cite{AP, CDGO, C1, C2, C, DMFZ, D, DV, M, O, P} and the references therein.

The Neumann sieve problem fits within the broader class of boundary value problems posed on perforated domains, a field whose analysis goes back to the foundational works of Marchenko and Khruslov \cite{MK} and of Cioranescu and Murat \cite{CM1, CM2} (see also \cite{CDG, DMD, DMG}). Homogenization in randomly perforated domains was first investigated in \cite{MK, PV}. Our approach is influenced by more recent developments that apply marked point processes to the homogenization of Poisson’s equation in random media \cite{GHV}. This framework has since been extended to the Stokes equation \cite{GH} and to nonlinear variational problems \cite{SZZ}.

To conclude the introduction, we outline the organization of the paper. In Section \ref{sec:preliminaries}, we recall preliminary results from the theory of marked point processes and properties of the capacity function. In Section \ref{sec:main_result}, we state the main homogenization result in Theorem \ref{thm:homogenization_theorem} as well as the existence of the sequence of oscillating test functions in Theorem \ref{thm:oscillating_test_functions}. A proof of Theorem \ref{thm:homogenization_theorem} is given immediately, using Theorem \ref{thm:oscillating_test_functions}. The rest of the paper is devoted to the proof of Theorem \ref{thm:oscillating_test_functions}. In Section \ref{sec:perforations}, we show that clusters have asymptotically vanishing capacity using methods initiated in \cite{GHV}. However, we use the framework that we set up in our previous work \cite{B}. The oscillating test functions are constructed in Section \ref{sec:otf}, with the important properties being proved in Section \ref{ssec:otf_construction}. Section \ref{ssec:variational_property} contains the proof of the fact that the oscillating test functions have minimal transition energy. Finally, in Section \ref{sec:proof_of_main_theorem}, we complete the proof of Theorem \ref{thm:oscillating_test_functions}. At the end, we include a section with the proofs of some technical results on capacity functions.

\section{Notation and Preliminaries} \label{sec:preliminaries}

In this section, we introduce some of the common notation, and we state some results about marked point processes and capacity functions that we will use throughout the paper.

\subsection{Notation}

We generally use the prime notation to indicate objects that are lower dimensional. Hence, we write $A'$ for a subset of $\mathbb{R}^{N - 1}$ and $A$ for a subset of $\mathbb{R}^N$ that is obtained from $A'$, e.g. $A = A' \times (0, 1)$. Given $x \in \mathbb{R}^N$, we set $x' := (x_1, \dots, x_{N - 1})$ and write $x = (x', x_N)$. The $N$-dimensional ball centered at $x$ with radius $r$ is denoted by $B(x, r)$, and the $(N - 1)$-dimensional ball centered at $x'$ with the same radius is denoted by $B'(x', r)$. We denote the $N$-dimensional cylinder centered at $0$ with radius $r$ and height $2h$ by $C(r, h)$. We write $\nabla$ for the gradient operator in $\mathbb{R}^N$, and in $\mathbb{R}^{N - 1}$ we write $\nabla'$.

The Lebesgue measure of a measurable set $A \subset \mathbb{R}^N$ is written as $|A|$. We denote the $N$-dimensional Hausdorff measure by $\mathcal{H}^N$. The cardinality of a set $A$ is denoted by $\card(A)$.

If $O \subset \mathbb{R}^N$ is open and $f : O \rightarrow \mathbb{R}$ is continuous, then we denote the closure of $\{x \in O : f(x) \neq 0\}$ in the subspace topology of $O$ by $\supp(f)$. For an arbitrary set $A \subset \mathbb{R}^N$, $\topint(A)$ denotes the topological interior of $A$.

If $(a_n)$ and $(b_n)$ are sequences of positive real numbers, we write $a_n \ll b_n$ or $b_n \gg a_n$ if $b_n/a_n \to \infty$. If $b_n/a_n$ is bounded from below, we write $a_n \lesssim b_n$ or $b_n \gtrsim a_n$. If $a_n \gtrsim b_n$ and $a_n \lesssim b_n$, we simply write $a_n \sim b_n$. The notation can be extended to functions of a continuous variable. In inequalities, constants that do not depend on the variables will be denoted by $C$. The value of $C$ may change from line to line without indication.

\subsection{Marked point processes} \label{ssec:mpp}

In this section, we present the necessary definitions and key results from the theory of marked point processes. For additional information and proofs of some of the results, we refer the reader to our previous paper \cite{B}. A more detailed introduction can be found in the books \cite{CSKM, DVJ}.

We formulate all definitions and theorems in the context of Euclidean spaces. Let $d \in \mathbb{N}$. If $O \subset \mathbb{R}^d$ is open, we denote the Borel $\sigma$-algebra on $O$ by $\mathcal{B}(O)$.

A point process on $\mathbb{R}^d$ is a set of randomly distributed points in $\mathbb{R}^d$. If every point has an associated characteristic (a \textit{mark}), then we call it a marked point process (in short m.p.p.). In our work, a mark is simply a positive number. We call $Y \subset \mathbb{R}^d \times (0, \infty)$ \textit{admissible} if it satisfies the following conditions:
\begin{enumerate}[label=(\roman*)]
    \item Each point has a unique mark, that is, if $(y, \rho_1), (y, \rho_2) \in Y$, then $\rho_1 = \rho_2$,
    \item The projection of $Y$ onto $\mathbb{R}^d$ is \textit{locally finite}. Equivalently, the set $Y \cap (B \times (0, \infty))$ has finite cardinality for all bounded Borel sets $B \in \mathcal{B}(\mathbb{R}^d)$.
\end{enumerate}
We define $\mathbb{M}^d$ to be the set of all admissible $Y \subset \mathbb{R}^d \times (0, \infty)$. Let $\mathcal{M}^d$ be the smallest $\sigma$-algebra on $\mathbb{M}^d$ that makes all functions $Y \mapsto \card(Y \cap E)$ measurable, as $E$ runs through the Borel sets in $\mathcal{B}(\mathbb{R}^d \times (0, \infty))$.

\begin{definition}
    A \textit{marked point process} on $\mathbb{R}^d$ with positive marks is a measurable map $M : (\Omega, \mathcal{F}, \mathbb{P}) \rightarrow (\mathbb{M}^d, \mathcal{M}^d)$ on a probability space $(\Omega, \mathcal{F}, \mathbb{P})$. Its \textit{probability distribution} is the measure $\mathcal{L}_M$ defined by $\mathcal{L}_M(\mathcal{A}) = \mathbb{P}(M^{-1}(\mathcal{A}))$ for all $\mathcal{A} \in \mathcal{M}^d$. Equivalently, $\mathcal{L}_M$ is the push-forward measure of $\mathbb{P}$ under $M$.
\end{definition}

By a marked point process we will exclusively mean an m.p.p. on $\mathbb{R}^d$ with positive marks. Let $M : (\Omega, \mathcal{F}, \mathbb{P}) \rightarrow (\mathbb{M}^d, \mathcal{M}^d)$ be an m.p.p. If $E \in \mathcal{B}(\mathbb{R}^d \times (0, \infty))$, then the integer-valued random variable $\omega \mapsto M(\omega)(E)$, where $\omega \in \Omega$, represents the number of points in $M(\omega)$ lying in $E$. We denote it briefly by $M(E)$. Thus, the average number of points in $M$ that belong to $E$ is given by the expectation $\mathbb{E}[M(E)]$. It is easy to check that the expectation is a Borel measure on $\mathbb{R}^d \times (0, \infty)$ as a function of $E$.

\begin{definition}
    Let $M : (\Omega, \mathcal{F}, \mathbb{P}) \rightarrow (\mathbb{M}^d, \mathcal{M}^d)$ be a marked point process. Define $\Lambda(E) := \mathbb{E}[M(E)]$ for all $E \in \mathcal{B}(\mathbb{R}^d \times (0, \infty))$. Then $\Lambda$ is a Borel measure called the \textit{intensity measure} of $M$. If $\Lambda(B \times (0, \infty)) < \infty$ for all bounded Borel sets $B \in \mathcal{B}(\mathbb{R}^d)$, then we say that $M$ has \textit{finite intensity}.
\end{definition}

Now, we introduce an important class of marked point processes. If $Y \in \mathbb{M}^d$ and $\tau \in \mathbb{R}^d$, we define
\[
Y_\tau := \{(y + \tau, \rho) : (y, \rho) \in Y\}.
\]
Analogously, if $\mathcal{A} \in \mathcal{M}^d$, then $\mathcal{A}_\tau := \{Y_\tau : Y \in \mathcal{A}\}$.

\begin{definition}
    A marked point process $M : (\Omega, \mathcal{F}, \mathbb{P}) \rightarrow (\mathbb{M}^d, \mathcal{M}^d)$ is called \textit{stationary} if $\mathcal{L}_M(\mathcal{A}) = \mathcal{L}_M(\mathcal{A}_\tau)$ for all $\tau \in \mathbb{R}^d$ and $\mathcal{A} \in \mathcal{M}^d$.
\end{definition}

Intuitively, stationarity means that the points of the m.p.p. are distributed spatially homogeneously. For a stationary m.p.p., the intensity measure can be factorized as shown in the following proposition. For a proof, see Proposition 2.5 in \cite{B}.

\begin{proposition} \label{pr:intensity_measure_stationary}
    Assume $M : (\Omega, \mathcal{F}, \mathbb{P}) \rightarrow (\mathbb{M}^d, \mathcal{M}^d)$ is a stationary marked point process with finite intensity. Then there exists a finite Borel measure $\lambda$ on $(0, \infty)$ such that
    \begin{equation} \label{eq:intensity_measure_splits}
        \Lambda =  \lambda \otimes \mathcal{L}^d,
    \end{equation}
    where $\mathcal{L}^d$ is the $d$-dimensional Lebesgue measure.
\end{proposition}

Assume $(C_k) \subset \mathbb{R}^d$ is an increasing sequence of sets such that $\bigcup_{k = 1}^\infty C_k = \mathbb{R}^d$. Often, one is interested in the spatial averages
\begin{equation} \label{eq:spatial_average}
    \frac{1}{|C_k|} \sum \limits_{\substack{(y, \rho) \in M(\omega) \\ y \in C_k}} g(\rho)
\end{equation}
for a nonnegative function $g$ and $\omega \in \Omega$. For instance, if we take $g(\rho) = \rho$, then \eqref{eq:spatial_average} is an approximation to the average mark size in $M(\omega)$ per unit volume.

Below, we present a version of the ergodic theorem with which the limit of \eqref{eq:spatial_average} as $k \to \infty$ can be represented as a conditional expectation with respect to the $\sigma$-algebra of \textit{translation invariant sets}.

\begin{definition}
    Let $M : (\Omega, \mathcal{F}, \mathbb{P}) \rightarrow (\mathbb{M}^d, \mathcal{M}^d)$ be a marked point process. A set $\mathcal{A} \in \mathcal{M}^d$ is called \textit{invariant} if $\mathcal{L}_M(\mathcal{A} \,\triangle\, \mathcal{A}_\tau) = 0$ for all $\tau \in \mathbb{R}^d$. The $\sigma$-algebra of invariant sets is denoted by $\mathcal{I}$.
\end{definition}
We denote the $\sigma$-algebra $\{M^{-1}(\mathcal{A}) \in \mathcal{F}: \mathcal{A} \in \mathcal{I}\}$ by $M^{-1}(\mathcal{I})$. By an $M^{-1}(\mathcal{I})$-measurable random measure we mean a function $\xi : \Omega \times \mathcal{B}(\mathbb{R}^d \times (0, \infty)) \rightarrow [0, \infty]$ such that $\xi(\omega, \cdot)$ is a Borel measure for all $\omega \in \Omega$, and $\xi(\cdot, E)$ is $M^{-1}(\mathcal{I})$-measurable for all $E \in \mathcal{B}(\mathbb{R}^d \times (0, \infty))$.

The following lemma shows that the conditional expectation of a random variable with respect to the $\sigma$-algebra $M^{-1}(\mathcal{I})$ can be expressed as an integral with respect to a random measure.

\begin{lemma}[Lemma 12.2.III \cite{DVJ}] \label{lm:conditional_expectation_random_measure}
    Let $M : (\Omega, \mathcal{F}, \mathbb{P}) \rightarrow (\mathbb{M}^d, \mathcal{M}^d)$ be a stationary marked point process with finite intensity. Then there exists an $M^{-1}(\mathcal{I})$-measurable random measure $\xi$ such that 
    \begin{equation} \label{eq:conditional_expectation_integral_representation}
        \mathbb{E}[G \circ M \,|\, M^{-1} (\mathcal{I})](\omega) = \int_{\mathbb{R}^d} \int_0^\infty g(y, \rho) \, d\xi(\omega, \rho) \, dy \quad \mathbb{P}\text{-a.s.}
    \end{equation}
    for all Borel measurable $g : \mathbb{R}^d \times (0, \infty) \rightarrow [0, \infty)$, where
    \[
    G(Y) := \int_{\mathbb{R}^d \times (0, \infty)} g(y, \rho) \, dY(y, \rho).
    \]
    In particular,
    \begin{equation} \label{eq:conditional_expectation_measure_representation}
        \mathbb{E}[M(B \times L) \,|\, M^{-1}(\mathcal{I})](\omega) = \xi(\omega, L) |B| \quad \mathbb{P}\text{-a.s.}
    \end{equation}
    for all bounded $B \in \mathcal{B}(\mathbb{R}^d)$ and for all $L \in \mathcal{B}(0, \infty)$.
\end{lemma}

A proof of the following version of the ergodic theorem can be found in Section 7.1 of \cite{B}.

\begin{theorem} \label{thm:ergodic_theorem_second_version}
    Let $M : (\Omega, \mathcal{F}, \mathbb{P}) \rightarrow (\mathbb{M}^d, \mathcal{M}^d)$ be a stationary marked point process with finite intensity. Let $\xi$ be the random measure defined in Lemma \ref{lm:conditional_expectation_random_measure}. Assume $B \in \mathcal{B}(\mathbb{R}^d)$ is bounded with nonempty interior such that $|\partial B| = 0$. Then
    \begin{equation} \label{eq:ergodic_theorem_second_version}
        \lim_{\varepsilon \to 0} \frac{1}{\left|\frac{1}{\varepsilon} B\right|} \sum \limits_{\substack{(y, \rho) \in M(\omega) \\ \varepsilon y \in B}} g(\rho) = \int_0^\infty g(\rho) \, d\xi(\omega, \rho) \quad \mathbb{P}\text{-a.s.}
    \end{equation}
    for all $\lambda$-integrable functions $g : (0, \infty) \rightarrow [0, \infty)$, where $\lambda$ is the measure on $(0, \infty)$ given by $\Lambda = \lambda \otimes \mathcal{L}^d$.
\end{theorem}

\subsection{Harmonic capacity} \label{ssec:harmonic_capacity}
     
In this section, we introduce several notions of harmonic capacity and establish their relationships, together with auxiliary results needed later. Let us recall that the classical harmonic capacity of a bounded set $K \subset \mathbb{R}^N$ relative to an open set $U \supset \overline{K}$ is defined by
\begin{equation} \label{eq:classical_capacity}
    \cpct(K, U) := \inf \left\{\int_U |\nabla v|^2 \, dx : v \in C_c^\infty(U), \, K \subset \topint (\{v = 1\})\right\}.
\end{equation}
We introduce two other capacity functions that will appear frequently. Let $K' \subset \mathbb{R}^{N - 1}$ be bounded, and let $U' \subset \mathbb{R}^{N - 1}$ be open with $\overline{K'} \subset U'$. For $h > 0$ we define
\begin{equation} \label{eq:capacity_in_a_strip}
    \cpct_h(K', U') := \inf \left\{\int_{U_h} |\nabla v|^2 \, dx : v \in C_h^\infty(U'), \, K' \times \{0\} \subset \topint(\{v = 1\})\right\},
\end{equation}
where
\[
U_h := U' \times (-h, h), \quad C_h^\infty(U') := \left\{v \in C^\infty(\overline{U_h}) : \supp(v) \csubset \mathbb{R}^N\right\}.
\]
If we set
\[
E_h(v) := \frac{1}{2} \int_{U_1} |\nabla' v|^2 + \frac{1}{h^2} |\partial_N v|^2 \, dx, \quad v \in C_1^\infty(U'),
\]
then a simple change of variables gives
\begin{equation} \label{eq:capacity_rescaling}
    \frac{1}{2h} \cpct_h(K', U') = \inf \{E_h(v) : v \in C_1^\infty(U'), \, K' \times \{0\} \subset \topint\{(v = 1\})\}.
\end{equation}
We also introduce
\[
\cpct_0(K', U') := \lim_{h \to 0} \frac{1}{2h} \cpct_h(K', U').
\]
In Proposition \ref{pr:minimization_problem_characterization} we show that $\cpct_0(K', U')$ is characterized by a minimization problem similar to \eqref{eq:classical_capacity} and \eqref{eq:capacity_in_a_strip}. First we prove that it is well-defined.

\begin{proposition} \label{pr:continuity_in_h}
    Let $K' \subset \mathbb{R}^{N - 1}$ be bounded, and let $U' \subset \mathbb{R}^{N - 1}$ be open with $\overline{K'} \subset U'$. If $h_1 \le h_2$, then
    \begin{equation} \label{eq:comparison_in_h}
        \frac{h_1}{h_2} \cpct_{h_2}(K', U') \le \cpct_{h_1}(K', U') \le \frac{h_2}{h_1} \cpct_{h_2}(K', U').
    \end{equation}
    In particular, $h \in (0, \infty) \mapsto \cpct_h(K', U')/(2h)$ is a continuous decreasing function. Furthermore,
    \begin{equation} \label{eq:finite_limit}
        \lim_{h \to 0} \frac{1}{2h} \cpct_h(K', U') < \infty.
    \end{equation}
\end{proposition}

\begin{remark} \label{rmk:jointly_equal_to_zero}
    The proposition implies that $\cpct_0(K', U') = 0$ if and only if $\cpct_h(K', U') = 0$ for any $h > 0$.
\end{remark}

\begin{proof}
    Let $v \in C_1^\infty(U')$ with $K' \times \{0\} \subset \topint\{(v = 1\})$. By \eqref{eq:capacity_rescaling} we have
    \[
    \frac{\cpct_{h_1}(K', U')}{2h_1} \le E_{h_1}(v) \le \left(\frac{h_2}{h_1}\right)^2 E_{h_2}(v).
    \]
    Taking the infimum over all $v$ yields $\cpct_{h_1}(K', U') \le (h_2/h_1) \cpct_{h_2}(K', U')$. 
    
    The first inequality in \eqref{eq:comparison_in_h} is equivalent to showing that $\cpct_h(K', U')/(2h)$ is decreasing in $h$. Since $E_{h_2}(v) \le E_{h_1}(v)$, this follows from \eqref{eq:capacity_rescaling} as well. Continuity is a direct corollary of \eqref{eq:comparison_in_h}. Finally, we prove \eqref{eq:finite_limit}. Let $w \in C_c^\infty(U')$ with $K' \subset \topint(\{w = 1\})$. Set $\hat{w}(x', x_N) := w(x')$ for $(x', x_N) \in U' \times (-1, 1)$. Then $\hat{w} \in C_1^\infty(U')$ and $K' \times \{0\} \subset \topint(\{\hat{w} = 1\})$. We have
    \[
    \frac{\cpct_h(K', U')}{2h} \le E_h(\hat{w}) = \int_{U'} |\nabla' w|^2 \, dx'.
    \]
    Hence, $\cpct_h(K', U')/(2h)$ is uniformly bounded in $h$. Since it is also monotone in $h$, the limit as $h \to 0$ exists and is finite.
\end{proof}

The following definition helps us extend the set of admissible functions in \eqref{eq:classical_capacity} and \eqref{eq:capacity_in_a_strip} to weakly differentiable functions.
\begin{definition} \label{def:vanishing_weakly}
    Let $U \subset \mathbb{R}^N$ be open, and let $A \subset \mathbb{R}^N$ be bounded. A function $v \in H_\mathrm{loc}^1(U)$ is said to \textit{vanish weakly} on $A$ if there exist an open neighborhood $W$ of $A$ and a sequence $(v_n) \subset C^\infty(U \cap W)$ such that 
    \begin{equation} \label{eq:vanishing_weakly}
        \lim_{n \to \infty} \|v_n - v\|_{H^1(U \cap W)} = 0, \quad A \cap \overline{\supp (v_n)} = \emptyset \text{ for all } n,
    \end{equation}
    where the closure is taken in $\mathbb{R}^N$. If $v - a$ vanishes weakly on $A$ for some $a \in \mathbb{R}$, we write $v = a$ weakly on $A$.
\end{definition}

\begin{remark} \label{rmk:global_approximation}
    If $v \in H^1(U)$, then $(v_n)$ can be taken in $C^\infty(U)$, and the convergence in \eqref{eq:vanishing_weakly} holds in $H^1(U)$. This follows by a cut-off argument and from the density of smooth functions in $H^1(U)$.
\end{remark}

\begin{remark} \label{rmk:vanishing_weakly_on_boundary}
    Note that $A$ does not have to be a subset of $U$. For example, if $U$ is bounded, then $u \in H^1(U)$ vanishes weakly on $\partial U$ if and only if $u \in H_0^1(U)$.
\end{remark}

The next lemma shows that we can define weak vanishing equivalently in terms of weak convergence in $H^1$.

\begin{lemma} \label{lm:vanishing_weakly}
    Let $U \subset \mathbb{R}^N$ be open, and let $A \subset \mathbb{R}^N$ be bounded. Let $v \in H_\mathrm{loc}^1(U)$. Assume there exist an open neighborhood $W$ of $A$ and a sequence $(v_n) \subset C^\infty(U \cap W)$ such that
    \[
    \lim_{n \to \infty} \|v_n - v\|_{L^2(U \cap W)} = 0, \quad \limsup_{n \to \infty} \|\nabla v_n\|_{L^2(U \cap W; \mathbb{R}^N)} < \infty, \quad A \cap \overline{\supp (v_n)} = \emptyset \text{ for all } n.
    \]
    Then $v = 0$ weakly on $A$.
\end{lemma}

\begin{proof}
    Without loss of generality, we assume $v \in L^2(U \cap W)$. Otherwise, we replace $W$ by a bounded open neighborhood of $A$. By hypothesis $v_n \rightharpoonup v$ in $H^1(U \cap W)$. Hence, by Mazur's lemma there exist convex combinations
    \[
    w_n := \sum_{j = 1}^n \lambda_j^{(n)} v_j, \quad 0 \le \lambda_j^{(n)} \le 1, \quad \sum_{j = 1}^n \lambda_j^{(n)} = 1
    \]
    such that $w_n \to v$ in $H^1(U \cap W)$. Since it is clear that $A \cap \overline{\supp (w_n)} = \emptyset$, the claim is proved.
\end{proof}

If $U$ and $U'$ are bounded, then we have
\begin{gather*}
    \cpct(K, U) = \inf \left\{\int_U |\nabla v|^2 \, dx : v \in H_0^1(U), \, v = 1 \text{ weakly on } K\right\}, \\
    \cpct_h(K', U') = \inf \left\{\int_{U_h} |\nabla v|^2 \, dx : v \in H^1(U_h), \, v = 0 \text{ weakly on } \Gamma_h, \, v = 1 \text{ weakly on } K' \times \{0\}\right\},
\end{gather*}
where we set $\Gamma_h := \partial U' \times (-h ,h)$. The equalities follow from Remarks \eqref{rmk:global_approximation}, \eqref{rmk:vanishing_weakly_on_boundary} and the direct method of calculus of variations. When $U = \mathbb{R}^N$ and $U' = \mathbb{R}^{N - 1}$ the correct function spaces are no longer Sobolev spaces. Set $\mathbb{R}_h^N := \mathbb{R}^{N - 1} \times (-h, h)$. We introduce
\begin{gather*}
    \mathcal{C}(\mathbb{R}^N) := \{v \in L^\frac{2N}{N - 2}(\mathbb{R}^N) : \nabla v \in L^2(\mathbb{R}^N; \mathbb{R}^N)\} \quad \text{for } N > 2, \\
    \mathcal{C}_h(\mathbb{R}^{N - 1}) := \{v \in L^2(-h , h; L^\frac{2(N - 1)}{N - 3}(\mathbb{R}^{N - 1})) : \, \nabla v \in L^2(\mathbb{R}_h^N; \mathbb{R}^N)\} \quad \text{for } N > 3.
\end{gather*}
By the Sobolev inequality, the completions of $C_c^\infty(\mathbb{R}^N)$ and $C_h^\infty(\mathbb{R}^{N - 1})$ under the $H_0^1$-norm yield $\mathcal{C}(\mathbb{R}^N)$ and $\mathcal{C}_h(\mathbb{R}^{N - 1})$, respectively. Consequently,
\begin{gather} \label{eq:extended_capacity_1}
    \cpct (K, \mathbb{R}^N) = \inf \left\{\int_{\mathbb{R}^N} |\nabla v|^2 \, dx : v \in \mathcal{C}(\mathbb{R}^N), \, v = 1 \text{ weakly on $K$}\right\} \quad \text{for } N > 2, \\ \label{eq:extended_capacity_2}
    \cpct_h(K', \mathbb{R}^{N - 1}) = \inf \left\{\int_{\mathbb{R}_h^N} |\nabla v|^2 \, dx : v \in \mathcal{C}_h(\mathbb{R}^{N - 1}), \, v = 1 \text{ weakly on $K' \times \{0\}$}\right\} \quad \text{for } N > 3.
\end{gather}
Using the larger set of admissible functions, we shall show that $\cpct_h(K', U') \to \cpct(K' \times \{0\}, U' \times \mathbb{R})$ as $h \to \infty$ for $N > 3$ and any open set $U' \supset \overline{K'}$. We begin by proving the equivalence of the function spaces involved.

\begin{lemma} \label{lm:equivalence_of_spaces}
    Assume $N > 3$. Let $v \in H_\mathrm{loc}^1(\mathbb{R}^N)$ with $\nabla v \in L^2(\mathbb{R}^N; \mathbb{R}^N)$. Then
    \begin{equation} \label{eq:equivalence_of_spaces}
        v \in L^\frac{2N}{N - 2}(\mathbb{R}^N) \iff v \in L^2(\mathbb{R}; L^\frac{2(N - 1)}{N - 3}(\mathbb{R}^{N - 1})).
    \end{equation}
\end{lemma}

\begin{proof}
    Since $\nabla v \in L^2(\mathbb{R}^N; \mathbb{R}^N)$, we can find a sequence $(\varphi_n) \subset C_c^\infty(\mathbb{R}^N)$ such that $\nabla \varphi_n \to \nabla v$ in $L^2(\mathbb{R}^N; \mathbb{R}^N)$. By Sobolev's inequality $(\varphi_n)$ is a Cauchy sequence in $L^\frac{2N}{N - 2}(\mathbb{R}^N)$ and $L^2(\mathbb{R}; L^\frac{2(N - 1)}{N - 3}(\mathbb{R}^{N - 1}))$. As convergence in these spaces implies convergence in $L_\mathrm{loc}^2(\mathbb{R}^N)$, $(\varphi_n)$ has the same limit $\varphi$ in both of them. Moreover, $\nabla \varphi = \nabla v$. Therefore, $\varphi - v = a$ for some $a \in \mathbb{R}$. Clearly, $a = 0$ if and only if
    \[
    v \in L^\frac{2N}{N - 2}(\mathbb{R}^N) \cap L^2(\mathbb{R}; L^\frac{2(N - 1)}{N - 3}(\mathbb{R}^{N - 1})).
    \]
    We now observe that $a = 0$ if any of the conditions in \eqref{eq:equivalence_of_spaces} holds, since a constant function cannot belong to either space if its nonzero.
\end{proof}

\begin{proposition} \label{pr:limit_at_infinity}
    Assume $N > 3$. Let $K' \subset \mathbb{R}^{N - 1}$ be bounded, and let $U' \subset \mathbb{R}^{N - 1}$ be open with $\overline{K'} \subset U'$. Then $\cpct_h(K', U') \le \cpct(K' \times \{0\}, U' \times \mathbb{R})$ for all $h > 0$ and
    \[
    \lim_{h \to \infty} \cpct_h(K', U') = \cpct(K' \times \{0\}, U' \times \mathbb{R}).
    \]
\end{proposition}

\begin{remark} \label{rmk:jointly_equal_to_zero_2}
    Combined with Remark \ref{rmk:jointly_equal_to_zero}, this proposition implies that $\cpct_0(K', U') = 0$ if and only if $\cpct(K' \times \{0\}, U' \times \mathbb{R}) = 0$ for $N > 3$. A similar remark applies to $\cpct_h(K', U')$ as well.
\end{remark}

\begin{proof}
    Let $w \in C_c^\infty(U' \times \mathbb{R})$ with $K' \times \{0\} \subset \topint(\{w = 1\})$. Then $w \in C_h^\infty(U')$ for all $h > 0$. Hence, $\cpct_h(K', U') \le \int_{U' \times \mathbb{R}} |\nabla w|^2 \, dx$. Since $w$ is arbitrary, we conclude that $\cpct_h(K', U') \le \cpct(K' \times \{0\}, U' \times \mathbb{R})$ for all $h > 0$. Now choose $v_h \in C_h^\infty(U')$ with $K' \times \{0\} \subset \topint(\{v_h = 1\})$ such that
    \[
    \int_{U_h} |\nabla v_h|^2 \, dx = \cpct_h(K', U') + \varepsilon_h, \quad \lim_{h \to \infty} \varepsilon_h = 0.
    \]
    Applying Sobolev's inequality along $(N - 1)$-dimensional cross-sections and integrating in $x_N$, we obtain 
    \begin{equation} \label{eq:sobolev's_inequality}
        \int_{-h}^h \left(\int_{\mathbb{R}^{N - 1}} |v_h|^\frac{2(N - 1)}{N - 3} \, dx' \right)^\frac{N - 3}{N - 1} \, dx \le C \int_{\mathbb{R}_h^N} |\nabla v_h|^2 \, dx,
    \end{equation}
    where $C$ is independent of $h$. If we extend $v_h$ by zero to $\mathbb{R}^N$, then it follows that $(v_h)$ is bounded in $L^2(\mathbb{R}; L^\frac{2(N - 1)}{N - 3}(\mathbb{R}^{N - 1}))$. Hence, we can pass to a subsequence (not relabeled) such that $v_h \rightharpoonup v$ in $L^2(\mathbb{R}; L^\frac{2(N - 1)}{N - 3}(\mathbb{R}^{N - 1}))$ and $\nabla v_h \rightharpoonup g$ in $L^2(\mathbb{R}^N; \mathbb{R}^N)$ for some $v$ and $g$. If $\varphi \in C_c^\infty(\mathbb{R}^N)$, then
    \[
    \int_{\mathbb{R}^N} v \nabla \varphi \, dx = \lim_{h \to \infty} \int_{\mathbb{R}_h^N} v_h \nabla \varphi \, dx = -\lim_{h \to \infty} \int_{\mathbb{R}_h^N} \varphi \nabla v_h \, dx = \int_{\mathbb{R}^N} \varphi g \, dx.
    \]
    This implies $\nabla v = g$. Consequently, Lemma \ref{lm:equivalence_of_spaces} applies to $v$, and we deduce that $v \in \mathcal{C}(\mathbb{R}^N)$. Next we show that $v = 0$ weakly on $K' \times \{0\}$. Let $W \supset K' \times \{0\}$ be open and bounded. Using \eqref{eq:sobolev's_inequality} and Hölder's inequality, we see that $(v_h)$ is uniformly bounded in $H^1(W)$. Consequently, the Rellich-Kondrachov compactness theorem gives $\|v_h - v\|_{L^2(W)} \to 0$ and we obtain the weak vanishing from Lemma \ref{lm:vanishing_weakly}. Therefore, $v$ is an admissible function in \eqref{eq:extended_capacity_1}. As a result
    \[
    \cpct(K' \times \{0\}, \mathbb{R}^N) \le \int_{\mathbb{R}^N} |\nabla v|^2 \, dx \le \liminf_{h \to \infty} \int_{\mathbb{R}_h^N} |\nabla v_h|^2 \, dx \le \liminf_{h \to \infty} \cpct_h(K', U') \le \cpct(K' \times \{0\}, \mathbb{R}^N),
    \]
    where we used the lower semi-continuity of the $L^2$-norm in the second inequality. This proves our claim.
\end{proof}

Next we characterize $\cpct_0(K', U')$ by a minimization problem.

\begin{proposition} \label{pr:minimization_problem_characterization}
    Assume $N > 3$. Let $K' \subset \mathbb{R}^{N - 1}$ be bounded, and let $U' \subset \mathbb{R}^{N - 1}$ be bounded and open with $\overline{K'} \subset U'$. For $v : \mathbb{R}^{N - 1} \rightarrow \mathbb{R}$, define $\hat{v}(x', x_N) := v(x')$ for $(x', x_N) \in \mathbb{R}^N$. Then,
    \begin{gather} \label{eq:minimization_problem_characterization}
        \cpct_0(K', U') = \inf \left\{\int_{U'} |\nabla' v|^2 \, dx' : v \in H_0^1(U'), \, \hat{v} = 1 \text{ weakly on $K' \times \{0\}$}\right\}, \\ \label{eq:extended_capacity_3}
        \cpct_0(K', \mathbb{R}^{N - 1}) = \inf \left\{\int_{\mathbb{R}^{N - 1}} |\nabla' v|^2 \, dx': v \in \mathcal{C}(\mathbb{R}^{N - 1}), \, \hat{v} = 1 \text{ weakly on } K' \times \{0\}\right\}.
    \end{gather}
\end{proposition}

\begin{remark}
    The condition $\hat{v} = 0$ weakly on $K' \times \{0\}$ is not equivalent to $v = 0$ weakly on $K'$. For this reason, $\cpct_0(K', U') \neq \cpct(K', U')$ in general. For example, if $N > 3$ and the Hausdorff dimension of $K'$ is strictly between $N - 2$ and $N - 1$, then $\cpct(K', U') > 0$, whereas $\cpct_0(K', U') = 0$.
\end{remark}

\begin{proof}
    We only prove \eqref{eq:minimization_problem_characterization} since the proof of \eqref{eq:extended_capacity_3} differs only slightly. Define $m_h := \cpct_h(K', U')/(2h)$ and $m_0 := \lim_{h \to 0} m_h$. Denote the right-hand side of \eqref{eq:minimization_problem_characterization} by $m$. Let $v \in H_0^1(U')$ with $\hat{v} = 1$ weakly on $K' \times \{0\}$. Then $\hat{v} \in H^1(U_h)$, and, by Remark \ref{rmk:vanishing_weakly_on_boundary}, $\hat{v} = 0$ weakly on $\partial U' \times (-h, h)$ for all $h > 0$. Thus, \eqref{eq:extended_capacity_2} implies
    \[
    \cpct_h(K', U') \le \int_{U_h} |\nabla \hat{v}|^2 \, dx = 2h \int_{U'} |\nabla' v|^2 \, dx'.
    \]
    Since $v$ is arbitrary, dividing by $h$ and letting $h \to 0$ yields $m_0 \le m$. To prove the other inequality, choose $v_h \in C_1^\infty(U')$ with $K' \times \{0\} \subset \topint(\{v_h = 1\})$ such that $E_h(v_h) \to m_0$. If $h_1, h_2 \le h$, then
    \[
    E_h\left(\frac{v_{h_1} - v_{h_2}}{2}\right) = \frac{1}{2}E_h(v_{h_1}) + \frac{1}{2}E_h(v_{h_2}) - E_h\left(\frac{v_{h_1} + v_{h_2}}{2}\right).
    \]
    We have $m_h \le E_h((v_{h_1} + v_{h_2})/2)$, as $(v_{h_1} + v_{h_2})/2$ is an admissible function. Thus,
    \[
    E_h\left(\frac{v_{h_1} - v_{h_2}}{2}\right) \le \frac{1}{2}E_{h_1}(v_{h_1}) + \frac{1}{2}E_{h_2}(v_{h_2}) - m_h.
    \]
    Then
    \[
    \limsup_{h_1, h_2 \to 0} \int_{U_1} \left|\nabla' \left(\frac{v_{h_1} - v_{h_2}}{2}\right)\right|^2 \, dx' \le 2 \limsup_{h_1, h_2 \to 0} E_h\left(\frac{v_{h_1} - v_{h_2}}{2}\right) \le 2(m_0 - m_h).
    \]
    Because $m_h \to m_0$, we conclude that $(\nabla' v_h)$ is a Cauchy sequence in $L^2(U_1; \mathbb{R}^N)$. Furthermore,
    \[
    \limsup_{h \to 0} \int_{U_1} |\partial_N v_h|^2 \, dx \le 2 \limsup_{h \to 0} h^2 E_h(v_h) = 0,
    \]
    since $(E_h(v_h))$ is bounded. Then, by Poincaré's inequality for functions that are $0$ on $\partial U' \times (-1, 1)$, we obtain that $(v_n)$ is Cauchy in $L^2(U_1)$. Consequently, there exists $v \in H_0^1(U')$ such that $v_h \to \hat{v}$ in $H^1(U_1)$. In particular, $\hat{v} = 1$ weakly on $K' \times \{0\}$. Finally, we have
    \[
    m \le \int_{U'} |\nabla' v|^2 \, dx' \le \lim_{h \to 0} \frac{1}{2} \int_{U_1} |\nabla' v_h|^2 \, dx \le \lim_{h \to 0} E_h(v_h) = m_0.
    \]
\end{proof}

We end this section by stating two propositions that will prove useful later. If a function $v$, equal to $1$ on a compact set, has Dirichlet energy smaller than the capacity relative to a larger open set, then $v$ cannot vanish on the boundary of that open set. Consequently, its values must be large somewhere. We now make this statement precise for cylindrical domains. Recall that $C(l, h) = B'(0, l) \times (-h, h)$ for $l, h > 0$.

\begin{proposition} \label{pr:subcapacitary_growth_1}
    Let $K \subset \mathbb{R}^3$ be bounded. Fix $\theta \in (0, 1)$. There exist positive constants $L$, $H$, $\tau$ and $c$ such that if
    \[
    l \ge L, \quad h \ge H, \quad \max \left\{\frac{\ln(l)}{h}, \frac{h}{l^2} \right\} \le \tau,
    \]
    and $v \in H^1(C(l, h))$ satisfies
    \[
    v = 1 \text{ weakly on } K, \quad \int_{C(l, h)} |\nabla v|^2 \, dx < \theta \cpct(K, \mathbb{R}^3),
    \]
    then we have the lower bound
    \[
    \fint_{C(l, h)} v \, dx > c.
    \]
\end{proposition}

\begin{proposition} \label{pr:subcapacitary_growth_2}
    Let $N > 3$ and let $K' \subset \mathbb{R}^{N - 1}$ be bounded. Fix $\theta \in (0, 1)$. There exist positive constants $L$, $\tau$ and $c$ such that if $l \ge L$, $h/l^{N - 1} \le \tau$, and $v \in H^1(C(l, h))$ satisfies
    \[
    v = 1 \text{ weakly on } K' \times \{0\}, \quad \int_{C(l, h)} |\nabla v|^2 \, dx < \theta \cpct_h(K', \mathbb{R}^{N - 1}),
    \]
    then we have the lower bound
    \[
    \fint_{C(l, h)} v \, dx > c.
    \]
\end{proposition}

\begin{remark}
    For the claim in the propositions to hold, it is not enough to only assume that $l$ and $h$ are sufficiently large. Without further restrictions on the relative sizes of $l$ and $h$, the result is not true. The conditions we provide are sharp and are related to the conditions in Lemma \ref{lm:vanishing_averages} in Section \ref{sec:not_appendix}. Counterexamples can be constructed with the help of Remark \ref{rmk:explanation_of_the_factors}.
\end{remark}

Since the proofs of Propositions \ref{pr:subcapacitary_growth_1} and \ref{pr:subcapacitary_growth_2} are long and technical, we relegate them to Section \ref{sec:not_appendix} at the end of the paper.

\section{Setting of the problem and the main result} \label{sec:main_result}

In this section, we formulate the problem carefully and state our main result. Fix a natural number $N > 2$. Let $U' \subset \mathbb{R}^{N - 1}$ be a Lipschitz domain. We denote by $\varepsilon$ a positive parameter, and we let $(\delta_\varepsilon)$ be a positive sequence with $\lim_{\varepsilon \to 0} \delta_\varepsilon = 0$. We set
\[
U^+ := U' \times (0, 1), \quad U^- := U' \times (-1, 0), \quad U_\varepsilon^+ := U' \times (0, \delta_\varepsilon), \quad U_\varepsilon^- := U' \times (-\delta_\varepsilon, 0).
\]
Let $M : (\Omega, \mathcal{F}, \mathbb{P}) \rightarrow (\mathbb{M}^{N - 1}, \mathcal{M}^{N - 1})$ be a m.p.p. on $\mathbb{R}^{N - 1}$ with the associated measures $\lambda$ and $\xi$ as defined in Proposition \ref{pr:intensity_measure_stationary} and Lemma \ref{lm:conditional_expectation_random_measure}. Fix an open and bounded set $T' \subset B'(0, 1) \subset \mathbb{R}^{N - 1}$ containing $0$ such that $\cpct(T' \times \{0\}, \mathbb{R}^N) > 0$. For every $\varepsilon > 0$ and $\omega \in \Omega$, we define a random subset of $U'$ by
\begin{equation} \label{eq:random_set}
    T'_\varepsilon(\omega) := \bigcup \limits_{\substack{(y', \rho) \in M(\omega) \\ \varepsilon y' \in U'}} (\varepsilon y' + a_\varepsilon \rho T') \cap U',
\end{equation}
where $(a_\varepsilon)$ is a sequence of positive numbers such that
\[
a_\varepsilon \ll \varepsilon, \quad \frac{\delta_\varepsilon}{a_\varepsilon} =: h_\varepsilon \to h_0 \in [0, \infty] \quad \text{as } \varepsilon \to 0.
\]
Since $T'$ is open in $\mathbb{R}^{N - 1}$, $T'_\varepsilon(\omega)$ is open in $\mathbb{R}^{N - 1}$ as well. We also set
\[
U_\varepsilon(\omega) := U_\varepsilon^- \cup (T'_\varepsilon(\omega) \times \{0\}) \cup U_\varepsilon^+, \quad \hat{U}_\varepsilon(\omega) := U^- \cup (T'_\varepsilon(\omega) \times \{0\}) \cup U^+.
\]
Clearly, $U_\varepsilon(\omega)$ and $\hat{U}_\varepsilon(\omega)$ are both open in $\mathbb{R}^N$. See Figure \ref{fig:domain} for an illustration of $U_\varepsilon(\omega)$ for spherical holes.

Let $f \in L^2(U')$ and define $\hat{f}(x', x_N) := f(x')$ for $x = (x', x_N) \in U' \times \mathbb{R}$. We consider the boundary value problems
\begin{equation} \label{eq:boundary_value_problem}
    \begin{cases}
        \begin{alignedat}{2}
            -\Delta u_\varepsilon(\omega, \cdot) &= \hat{f} \quad && \text{in } U_\varepsilon(\omega), \\
            u_\varepsilon(\omega, \cdot) &= 0 \quad && \text{on } \partial U' \times (-\delta_\varepsilon, \delta_\varepsilon), \\
            \nabla u_\varepsilon(\omega, \cdot) \cdot \nu &= 0 \quad && \text{on } ((U' \setminus T'_\varepsilon(\omega)) \times \{0\}) \cup (U' \times \{-\delta_\varepsilon, \delta_\varepsilon\}).
        \end{alignedat}
    \end{cases}
\end{equation}
The Neumann boundary condition holds separately on both sides of $(U' \setminus T'_\varepsilon(\omega)) \times \{0\}$, and the unit normal vector $\nu$ takes either the value $e_N$ or $-e_N$ depending on the side, with $e_N = (0, \dots, 0, 1)$. To state the weak formulation of the problem, we define the following subset of the Sobolev space $H^1(U_\varepsilon(\omega))$:
\[
V_\varepsilon(\omega) := \{v \in H^1(U_\varepsilon(\omega)) : v = 0 \text{ weakly on } \partial U' \times (-\delta_\varepsilon, \delta_\varepsilon)\}.
\]
In this space, the boundary value problem \eqref{eq:boundary_value_problem} is formally equivalent to finding $u_\varepsilon(\omega, \cdot) \in V_\varepsilon(\omega)$ such that
\begin{equation} \label{eq:weak_formulation}
    \int_{U_\varepsilon(\omega)} \nabla u_\varepsilon(\omega, x) \nabla \varphi(x) \, dx = \int_{U_\varepsilon(\omega)} \hat{f}(x) \varphi(x) \, dx \quad \text{for all } \varphi \in V_\varepsilon(\omega).
\end{equation}
It is immediate from the Lax-Milgram theorem that \eqref{eq:weak_formulation} possesses a unique solution for all $\varepsilon > 0$ and $\omega \in \Omega$.

We are interested in the asymptotic behavior of the solutions $(u_\varepsilon(\omega, \cdot))$. By introducing the rescaled functions
\begin{equation} \label{eq:rescaled_solutions}
    \hat{u}_\varepsilon(\omega, (x', x_N)) := u_\varepsilon(\omega, (x', \delta_\varepsilon x_N)), \quad x = (x', x_N) \in \hat{U}_\varepsilon(\omega),
\end{equation}
we reduce the problem to a domain of fixed height. Our main theorem is concerned with the convergence of the sequence $(\hat{u}_\varepsilon(\omega, \cdot))$. As we will see, the characterization of the limit function depends on $h_0$. We define
\begin{equation*}
    J_{h_0}(\rho) :=
    \begin{cases}
        \rho^{N - 2} \cpct(T' \times \{0\}, \mathbb{R}^N) \quad &\text{if } h_0 = \infty, \\
        \rho^{N - 2} \cpct_\frac{h_0}{\rho}(T', \mathbb{R}^{N - 1}) \quad &\text{if } h_0 \in (0, \infty), \\
        2 \rho^{N - 3} \cpct_0(T', \mathbb{R}^{N - 1}) \quad &\text{if } h_0 = 0.
    \end{cases}
\end{equation*}

\begin{theorem} \label{thm:homogenization_theorem}
    Set
    \[
    a_\varepsilon :=
    \begin{cases}
        \varepsilon^\frac{N - 1}{N - 3} \quad &\text{if } \delta_\varepsilon^{N - 3} \ll \varepsilon^{N - 1}, \\
        \varepsilon^\frac{N - 1}{N - 2} {\delta_\varepsilon}^\frac{1}{N - 2} \quad &\text{otherwise}.
    \end{cases}
    \]
    If $N = 3$, then assume additionally that $\ln(1/\delta_\varepsilon) \ll 1/\varepsilon^2$. Let $M : (\Omega, \mathcal{F}, \mathbb{P}) \rightarrow (\mathbb{M}^{N - 1}, \mathcal{M}^{N - 1})$ be a stationary m.p.p. with finite intensity such that
    \begin{equation} \label{eq:moment_is_finite}
        \mathbb{E}\Bigg[\sum \limits_{\substack{(y', \rho) \in M \\ y' \in Q'}} J_{h_0}(\rho)\Bigg] = \int_0^\infty J_{h_0}(\rho) \, d\lambda(\rho) < \infty,
    \end{equation}
    where $Q' \subset \mathbb{R}^{N - 1}$ is the unit cube. Let $u_\varepsilon(\omega, \cdot)$ be the unique solution of \eqref{eq:weak_formulation} and define $\hat{u}_\varepsilon(\omega, \cdot)$ as in \eqref{eq:rescaled_solutions}. Then, for $\mathbb{P}$-a.e. $\omega \in \Omega$, there exist functions $u^\pm(\omega, \cdot) \in H_0^1(U')$ such that
    \[
    \hat{u}_\varepsilon(\omega, \cdot) \rightharpoonup \hat{u}^+(\omega, \cdot) \quad \text{in } H^1(U^+), \quad \hat{u}_\varepsilon(\omega, \cdot) \rightharpoonup \hat{u}^-(\omega, \cdot) \quad \text{in } H^1(U^-) \text{ as } \varepsilon \to 0,
    \]
    where $\hat{u}^\pm(\omega, (x', x_N)) := u^\pm(\omega, x')$. Furthermore, $u^\pm(\omega, \cdot)$ satisfy
    \begin{equation} \label{eq:homogenized_equation}
        \int_{U'} \nabla' u^\pm(\omega, x') \nabla \varphi(x') \, dx' \pm \frac{\gamma_{h_0}(\omega)}{2} \int_{U'} (u^+(\omega, x') - u^-(\omega, x')) \varphi(x') \, dx' = \int_{U'} f(x') \varphi(x')
    \end{equation}
    for all $\varphi \in H_0^1(U')$, with
    \[
    \gamma_{h_0}(\omega) = \frac{1}{2} \int_0^\infty J_{h_0}(\rho) \, d\xi(\omega, \rho).
    \]
\end{theorem}

\begin{remark}
    The equality in \eqref{eq:moment_is_finite} follows from Campbell's theorem \cite[Theorem 2.3]{B} and Proposition \ref{pr:intensity_measure_stationary}.
\end{remark}

To prove Theorem \ref{thm:homogenization_theorem}, we construct a sequence of oscillating test functions that satisfy the boundary conditions on $U' \times \{0\}$ and have Laplacians with good convergence properties. The existence of such a sequence is given in Theorem \ref{thm:oscillating_test_functions} below.

\begin{definition} \label{def:omega_admissible}
    For a fixed $\omega \in \Omega$, we call a sequence $(v_\varepsilon)$ \textit{$\omega$-admissible} if $v_\varepsilon \in H^1(U_\varepsilon^+)$, $v_\varepsilon = 0$ weakly on $T'_\varepsilon(\omega) \times \{0\}$, and there exists a constant $C$ such that
    \[
    \frac{1}{\delta_\varepsilon}\int_{U_\varepsilon^+} |\nabla v_\varepsilon|^2 \, dx \le C \quad \text{for all } \varepsilon > 0.
    \]
    We also define $\hat{v}_\varepsilon(x', x_N) := v_\varepsilon(x', \delta_\varepsilon x_N)$ for $(x', x_N) \in U^+$.
        
\end{definition}

\begin{theorem} \label{thm:oscillating_test_functions}
    Set
    \[
    a_\varepsilon :=
    \begin{cases}
        \varepsilon^\frac{N - 1}{N - 3} \quad &\text{if } \delta_\varepsilon^{N - 3} \ll \varepsilon^{N - 1}, \\
        \varepsilon^\frac{N - 1}{N - 2} {\delta_\varepsilon}^\frac{1}{N - 2} \quad &\text{otherwise}.
    \end{cases}
    \]
    If $N = 3$, then assume additionally that $\ln(1/\delta_\varepsilon) \ll 1/\varepsilon^2$. Let $M : (\Omega, \mathcal{F}, \mathbb{P}) \rightarrow (\mathbb{M}^{N - 1}, \mathcal{M}^{N - 1})$ be a stationary m.p.p. with finite intensity such that
    \[
    \int_0^\infty J_{h_0}(\rho) \, d\lambda(\rho) < \infty.
    \]
    For $\mathbb{P}$-a.e. $\omega \in \Omega$, there exists an $\omega$-admissible sequence $(w_\varepsilon(\omega, \cdot))$ such that $\hat{w}_\varepsilon(\omega, \cdot) \rightharpoonup 1$ in $H^1(U^+)$. Moreover, it satisfies the following property: given an $\omega$-admissible sequence $(v_\varepsilon)$, if there exists a function $\hat{v} \in H^1(U^+)$ and a subsequence $(\varepsilon_n)$ tending to $0$ such that $\hat{v}_{\varepsilon_n} \rightharpoonup \hat{v}$ in $H^1(U^+)$ as $n \to \infty$, we have
    \begin{equation} \label{eq:good_laplacians}
        \lim_{n \to \infty} \int_{U^+} \nabla' \hat{w}_{\varepsilon_n}(\omega, x) \nabla' \hat{v}_{\varepsilon_n}(x) + \frac{1}{\delta_{\varepsilon_n}^2} \partial_N \hat{w}_{\varepsilon_n}(\omega, x) \partial_N \hat{v}_{\varepsilon_n}(x) \, dx = \gamma_{h_0}(\omega) \int_{U^+} \hat{v}(x) \, dx.
    \end{equation}
\end{theorem}

We will devote the rest of the paper to the proof of Theorem \ref{thm:oscillating_test_functions}. Assuming the result for the moment, we can prove our main result.

\begin{proof}[Proof of Theorem \ref{thm:homogenization_theorem}]
    We fix a realization $\omega \in \Omega$ for which Theorem \ref{thm:oscillating_test_functions} holds. For simplicity of notation, we omit $\omega$ from the argument. First we prove that $(\hat{u}_\varepsilon)$ has a weak limit in $U^+$ and $U^-$. Testing \eqref{eq:weak_formulation} with $u_\varepsilon$ gives the a priori estimate $\|\nabla u_\varepsilon\|_{L^2(U_\varepsilon; \mathbb{R}^{N})} \le C \|\hat{f}\|_{L^2(U_\varepsilon)}$. Then, rescaling in the $N$-th variable gives
    \begin{equation*}
            \delta_\varepsilon \int_{U^\pm} |\nabla' \hat{u}_\varepsilon|^2 + \frac{1}{\delta_\varepsilon^2}|\partial_N \hat{u}_\varepsilon|^2 \, dx = \int_{U_\varepsilon} |\nabla u_\varepsilon|^2 \, dx \lesssim \int_{U_\varepsilon} |\hat{f}|^2 \, dx = \delta_\varepsilon \int_{U'} |f|^2 \, dx'.
    \end{equation*}
    Therefore, there exists a constant $C$ such that
    \begin{equation} \label{eq:a_priori_bound}
            \int_{U^\pm} |\nabla' \hat{u}_\varepsilon|^2 + \frac{1}{\delta_\varepsilon^2}|\partial_N \hat{u}_\varepsilon|^2 \, dx \le C \quad \text{for all } \varepsilon > 0.
    \end{equation}
    It follows from Poincaré's inequality and the Rellich-Kondrachov compactness theorem that there exist $\hat{u}^\pm \in H^1(U^\pm)$ and a subsequence $(\varepsilon_n)$ tending to $0$ such that
    \begin{equation} \label{eq:existence_of_limit}
            \hat{u}_{\varepsilon_n} \to \hat{u}^\pm \text{ in } L^2(U^\pm), \quad \nabla \hat{u}_{\varepsilon_n} \rightharpoonup \nabla \hat{u}^\pm \text{ in } L^2(U^\pm; \mathbb{R}^N)
    \end{equation}
    as $n \to \infty$. Furthermore, $\hat{u}^+$ and $\hat{u}^-$ vanish weakly on $\partial U' \times (0, 1)$ and $\partial U' \times (-1, 0)$, respectively. Since $\delta_\varepsilon \to 0$, the bound \eqref{eq:a_priori_bound} implies $\partial_N \hat{u}^\pm = 0$. Hence, there exists $u^\pm \in H_0^1(U')$ such that
    \begin{alignat*}{2}
        \hat{u}^+(\cdot, x_N) &= u^+, \quad \nabla' \hat{u}^+(\cdot, x_N) = \nabla' u^+ \quad \text{for a.e. } x_N \in (0, 1), \\ \label{eq:cross_section_function_2}
        \hat{u}^-(\cdot, x_N) &= u^-, \quad \nabla' \hat{u}^-(\cdot, x_N) = \nabla' u^- \quad \text{for a.e. } x_N \in (-1, 0).
    \end{alignat*}
    Let $(w_\varepsilon)$ be the sequence provided by Theorem \ref{thm:oscillating_test_functions}, and let $\varphi \in C_c^\infty(U')$. We extend $w_\varepsilon$ to $U_\varepsilon^-$ by $w_\varepsilon(x', x_N) := -w_\varepsilon(x', -x_N)$, and set $\hat{w}_\varepsilon(x', x_N) := w_\varepsilon(x', \delta_\varepsilon x_N)$ for $(x', x_N) \in \hat{U}_\varepsilon$. Since $w_\varepsilon = 0$ weakly on $T'_\varepsilon \times \{0\}$, $w_\varepsilon \in H^1(U_\varepsilon)$. Furthermore, we have
    \begin{equation} \label{eq:convergence_to_signed_1}
        \hat{w}_\varepsilon \to \pm 1 \text{ in } L^2(U^\pm), \quad \nabla \hat{w}_\varepsilon \rightharpoonup 0 \text{ in } L^2(U^\pm; \mathbb{R}^N).
    \end{equation}
    Now we prove that $u^\pm$ solves the homogenized equation \eqref{eq:homogenized_equation}. We test \eqref{eq:weak_formulation} with $w_\varepsilon \varphi$ to obtain
    \begin{equation*}
        \int_{U_\varepsilon} (\nabla' u_\varepsilon \nabla' \varphi) w_\varepsilon + \nabla w_\varepsilon \nabla ( u_\varepsilon \varphi) - (\nabla' w_\varepsilon \nabla' \varphi) u_\varepsilon \, dx = \int_{U_\varepsilon} \hat{f} w_\varepsilon \varphi \, dx
    \end{equation*}
    If we rescale in the $N$-th variable and use the convergences \eqref{eq:existence_of_limit} and \eqref{eq:convergence_to_signed_1}, then we get
    \begin{multline} \label{eq:testing_testing}
        \int_{U^+} \nabla' \hat{u}^+ \nabla' \varphi \, dx - \int_{U^-} \nabla' \hat{u}^- \nabla' \varphi \, dx \\
        + \lim_{n \to \infty} \int_{\hat{U}_{\varepsilon_n}} \nabla' \hat{w}_{\varepsilon_n} \nabla' (\hat{u}_{\varepsilon_n} \varphi) + \frac{1}{\delta_{\varepsilon_n}^2} \partial_N \hat{w}_{\varepsilon_n} \partial_N (\hat{u}_{\varepsilon_n} \varphi) \, dx = \int_{U^+} \hat{f} \varphi \, dx - \int_{U^-} \hat{f} \varphi \, dx = 0,
    \end{multline}
    assuming that the limit exists. On the other hand, testing with $\varphi$ in \eqref{eq:weak_formulation} and passing to the limit gives
    \begin{equation} \label{eq:easy_limit}
        \int_{U^+} \nabla' \hat{u}^+ \nabla' \varphi \, dx + \int_{U^-} \nabla' \hat{u}^- \nabla' \varphi \, dx \\
        = \lim_{n \to \infty} \int_{\hat{U}_{\varepsilon_n}} \nabla' \hat{u}_{\varepsilon_n} \nabla' \varphi \, dx = \lim_{n \to \infty} \int_{\hat{U}_{\varepsilon_n}} \hat{f} \varphi \, dx = 2 \int_{U'} f \varphi \, dx'.
    \end{equation}
    Hence, if we denote the limit term on the left-hand side of \eqref{eq:testing_testing} by $L$, then \eqref{eq:testing_testing} and \eqref{eq:easy_limit} imply
    \begin{equation} \label{eq:primitive_homogenization}
        \int_{U'} \nabla' u^+ \nabla' \varphi \, dx' + \frac{L}{2} = \int_{U'} f \varphi \, dx', \quad \int_{U'} \nabla' u^- \nabla' \varphi \, dx' - \frac{L}{2} = \int_{U'} f \varphi \, dx'.
    \end{equation}
    Finally, we compute $L$. Define
    \[
    v_\varepsilon(x', x_N) := \varphi(x') \frac{u_\varepsilon(x', x_N) - u_\varepsilon(x', -x_N)}{2}, \quad (x', x_N) \in U_\varepsilon^+.
    \]
    Then $v_\varepsilon \in H^1(U_\varepsilon^+)$ and $v_\varepsilon = 0$ weakly on $T'_\varepsilon \times \{0\}$. Furthermore, it follows from \eqref{eq:a_priori_bound} that
    \[
    \frac{1}{\delta_\varepsilon} \int_{U_\varepsilon^+} |\nabla v_\varepsilon|^2 \, dx \lesssim \frac{1}{\delta_\varepsilon} \int_{U_\varepsilon} |\nabla u_\varepsilon|^2 \, dx = \int_{\hat{U}_\varepsilon} |\nabla' \hat{u}_\varepsilon|^2 + \frac{1}{\delta_\varepsilon^2} |\partial_N \hat{u}_\varepsilon|^2 \, dx \le C.
    \]
    Hence, $(v_\varepsilon)$ is $\omega$-admissible. In addition, we have $\hat{v}_{\varepsilon_n} \rightharpoonup \hat{v}$ in $H^1(U^+)$ as $n \to \infty$, where
    \[
    \hat{v}(x', x_N) := \varphi(x')\frac{u^+(x') - u^-(x')}{2}.
    \]
    Consequently, Theorem \ref{thm:oscillating_test_functions} yields
    \[
     \lim_{n \to \infty} \int_{U^+} \nabla' \hat{w}_{\varepsilon_n} \nabla' \hat{v}_{\varepsilon_n} + \frac{1}{\delta_{\varepsilon_n}^2} \partial_N \hat{w}_{\varepsilon_n} \partial_N \hat{v}_{\varepsilon_n} \, dx = \gamma_{h_0} \int_{U^+} \hat{v}(x) \, dx = \frac{\gamma_{h_0}}{2} \int_{U'} (u^+ - u^-) \varphi \, dx'.
    \]
    However, using the fact that $w_\varepsilon$ is an odd function in $x_N$, it is easy to show that
    \begin{equation*}
        \frac{1}{2} \int_{\hat{U}_{\varepsilon_n}} \nabla' \hat{w}_{\varepsilon_n} \nabla' (\hat{u}_{\varepsilon_n} \varphi) + \frac{1}{\delta_{\varepsilon_n}^2} \partial_N \hat{w}_{\varepsilon_n} \partial_N (\hat{u}_{\varepsilon_n} \varphi) \, dx \\
        = \int_{U^+} \nabla' \hat{w}_{\varepsilon_n} \nabla' \hat{v}_{\varepsilon_n} + \frac{1}{\delta_{\varepsilon_n}^2} \partial_N \hat{w}_{\varepsilon_n} \partial_N \hat{v}_{\varepsilon_n} \, dx,
    \end{equation*}
    so that
    \[
    \frac{L}{2} = \frac{\gamma_{h_0}}{2} \int_{U'} (u^+ - u^-) \varphi \, dx'.
    \]
    Together with \eqref{eq:primitive_homogenization}, this proves \eqref{eq:homogenized_equation} for $\varphi \in C_c^\infty(U')$. For arbitrary $\varphi$, we use the density of $C_c^\infty(U')$ in $H_0^1(U')$. Finally, since $u^\pm$ is uniquely characterized by \eqref{eq:homogenized_equation}, the convergence of $(\hat{u}_\varepsilon)$ holds for the entire sequence, not just for a subsequence.
\end{proof}

\section{Analysis close to the contact regions} \label{sec:perforations}

In constructing the oscillating test functions, we will distinguish between isolated contact regions and clusters of contact regions. While our construction in the neighborhood of isolated contact regions is based on solving a ``cell problem” given by a PDE, this approach breaks down near the clusters. In this section, however, we prove that the clusters can be enclosed in a region with respect to which they have small capacity. This allows us to construct functions with small gradients that vanish on the clusters. To this end, we adapt the methods from our previous paper \cite{B}. 

For the rest of the paper, we assume
\[
a_\varepsilon :=
\begin{cases}
    \varepsilon^\frac{N - 1}{N - 3} \quad &\text{if } \delta_\varepsilon^{N - 3} \ll \varepsilon^{N - 1}, \\
    \varepsilon^\frac{N - 1}{N - 2} {\delta_\varepsilon}^\frac{1}{N - 2} \quad &\text{otherwise}
\end{cases}
\]
and suppose that $h_0 := \lim_{\varepsilon \to 0} \delta_\varepsilon/a_\varepsilon$ exists. It is easy to verify that $h_0 = 0$ if and only if $\delta_\varepsilon^{N - 3} \ll \varepsilon^{N - 1}$. Similarly, $h_0 \in (0, \infty)$ and $h_0 = \infty$ if and only if $\delta_\varepsilon^{N - 3} \sim \varepsilon^{N - 1}$ and $\delta_\varepsilon^{N - 3} \gg \varepsilon^{N - 1}$, respectively. In particular, for $N = 3$ only $h_0 = \infty$ is possible.

We now introduce some new notation. For $Y \in \mathbb{M}^{N - 1}$ and $(y', \rho) \in Y$, we define
\begin{equation} \label{eq:distance_to_closest_neighbor}
    d_Y(y') = \min \{|y' - \tilde{y}'| :  (\tilde{y}', \tilde{\rho}) \in Y \setminus (y', \rho) \}
\end{equation}
if $Y \setminus (y', \rho)$ is nonempty. Otherwise, we set $d_Y(y') = \infty$. In other words, $d_Y(y')$ is the distance between $y'$ and its closest neighbor in $Y$. We note that 
\begin{equation} \label{eq:nonintersecting_balls}
    B'(y', d_Y(y')/2) \cap B'(\tilde{y}', d_Y(\tilde{y}')/2) = \emptyset
\end{equation}
for all $(y', \rho)$, $(\tilde{y}', \tilde{\rho}) \in Y$ with $y' \neq \tilde{y}'$.

Let $M :(\Omega, \mathcal{F}, \mathbb{P}) \rightarrow (\mathbb{M}^{N - 1}, \mathcal{M}^{N - 1})$ be an m.p.p.\!\! and define $T'_\varepsilon(\omega)$ as in \eqref{eq:random_set}. For $\omega \in \Omega$ and $(y', \rho) \in M(\omega)$, we define the truncated radius
\[
r(\omega, y') := \min \left\{\frac{1}{2}d_{M(\omega)}(y'), 1\right\}.
\]
If $\varepsilon y'$ and $\varepsilon \tilde{y}'$ are the centers of two distinct contact regions in $T'_\varepsilon(\omega)$, then \eqref{eq:nonintersecting_balls} implies
\begin{equation} \label{eq:nonintersecting_scaled_balls}
    B'(\varepsilon y', \varepsilon r(\omega, y')) \cap B'(\varepsilon \tilde{y}', \varepsilon r(\omega, \tilde{y}')) = \emptyset.
\end{equation}
Given $\varepsilon > 0$, $\omega \in \Omega$, our first objective is to carefully distinguish isolated contact regions in $T'_\varepsilon(\omega)$ from clusters.

\begin{definition} \label{def:isolated_and_cluster_points}
    Let $(y', \rho) \in M(\omega)$ with $\varepsilon y' \in U'$. Then $(y', \rho)$ is called $\varepsilon$-isolated if it satisfies the following conditions:
    \begin{enumerate}
        \item[(1a)] $2 a_\varepsilon\rho < \min \{\varepsilon r(\omega, y'), \delta_\varepsilon\}$ if $h_0 = \infty$,
        \item[(1b)] $2 a_\varepsilon\rho < \varepsilon r(\omega, y')$ if $h_0 < \infty$,
        \item[(2)] If $(\tilde{y}', \tilde{\rho}) \in M(\omega)$ with $\varepsilon \tilde{y}' \in U'$ and $y' \neq \tilde{y}'$, then $B'(\varepsilon \tilde{y}', 2 a_\varepsilon \tilde{\rho}) \cap B'(\varepsilon y', \varepsilon r(\omega, y')) = \emptyset$.
    \end{enumerate}
    Otherwise $(y', \rho)$ is called an $\varepsilon$-cluster point.
\end{definition}

According to our definition, a point $(y', \rho) \in M(\omega)$ with $\varepsilon y' \in U'$ is considered isolated if $B'(\varepsilon y', \varepsilon r(\omega, y'))$ separates the contact region $\varepsilon y' + a_\varepsilon \rho T'$ from the other contact regions in $T'_\varepsilon(\omega)$. When $h_0 = \infty$, we require furthermore that the size of the contact region is not much larger than the height $\delta_\varepsilon$.

We denote by $I_{\varepsilon, h_0}(\omega)$ and $C_{\varepsilon, h_0}(\omega)$ the set of $\varepsilon$-isolated and $\varepsilon$-cluster points, respectively. We also define
\begin{gather*}
    (T'_{\varepsilon, h_0})^C(\omega) := \bigcup_{(y', \rho) \in C_{\varepsilon, h_0}(\omega)} (\varepsilon y' + a_\varepsilon \rho T') \cap U', \\
    S'_{\varepsilon, h_0}(\omega) := \bigcup_{(y', \rho) \in C_{\varepsilon, h_0}(\omega)} B'(\varepsilon y', 2a_\varepsilon \rho).
\end{gather*}
Here, $(T'_{\varepsilon, h_0})^C(\omega)$ is the family of clusters, and $S'_{\varepsilon, h_0}(\omega)$ is a layer that separates the clusters from isolated contact regions. Indeed, it is easy to verify that
\begin{equation} \label{eq:separation_layer}
    B'(\varepsilon y', \varepsilon r(\omega, y')) \cap S'_{\varepsilon, h_0}(\omega) = \emptyset \quad \text{for all } (y', \rho) \in I_{\varepsilon, h_0}(\omega).
\end{equation}
For $h_0 = \infty$ we also introduce
\begin{gather*}
    \rho_{\max} := \max \{\rho: (y', \rho) \in C_{\varepsilon, \infty}(\omega)\}, \\
    S_{\varepsilon, \infty}(\omega) := S'_{\varepsilon, \infty}(\omega) \times (-\max\{\delta_\varepsilon, 2 a_\varepsilon \rho_{\max}\}, \max\{\delta_\varepsilon, 2 a_\varepsilon \rho_{\max}\}).
\end{gather*}
The main result of this section is the following theorem.

\begin{theorem} \label{thm:vanishing_capacity_and_measure}
    Let $M : (\Omega, \mathcal{F}, \mathbb{P}) \rightarrow (\mathbb{M}^{N - 1}, \mathcal{M}^{N - 1})$ be a stationary m.p.p. with finite intensity such that
    \[
    \int_0^\infty J_{h_0}(\rho) \, d\lambda(\rho) < \infty.
    \]
    Then
    \begin{equation} \label{eq:capacity_vanishes_infinity}
        \lim_{\varepsilon \to 0} \frac{1}{\delta_\varepsilon} \cpct\left((T'_{\varepsilon, \infty})^C(\omega), S_{\varepsilon, \infty}(\omega)\right) = 0 \quad \mathbb{P}\text{-a.s.}
    \end{equation}
    If $h_0 < \infty$, then we have
    \begin{equation} \label{eq:capacity_vanishes_noninfinity}
        \lim_{\varepsilon \to 0} \frac{1}{\delta_\varepsilon} \cpct_{\delta_\varepsilon}\left((T'_{\varepsilon, h_0})^C(\omega), S'_{\varepsilon, h_0}(\omega)\right) = 0 \quad \mathbb{P}\text{-a.s.}
    \end{equation}
    We also have
    \begin{equation} \label{eq:measure_vanishes}
        \lim_{\varepsilon \to 0} |S'_{\varepsilon, h_0}(\omega)| = 0 \quad \mathbb{P}\text{-a.s. for all } h_0 \in [0, \infty].
    \end{equation}
\end{theorem}

We reduce the proof of Theorem \ref{thm:vanishing_capacity_and_measure} to the computation of limits of weighted sums of $J_{h_0}$ and $\rho^{N - 1}$. The result is given in the next proposition.

\begin{proposition} \label{pr:weighted_sums_go_to_zero}
     Let $M : (\Omega, \mathcal{F}, \mathbb{P}) \rightarrow (\mathbb{M}^{N - 1}, \mathcal{M}^{N - 1})$ be a stationary m.p.p. with finite intensity such that
    \[
    \int_0^\infty J_{h_0}(\rho) \, d\lambda(\rho) < \infty.
    \]
    Then
    \begin{equation} \label{eq:weighted_sums_go_to_zero}
        \lim_{\varepsilon \to 0} \sum_{(y', \rho) \in C_{\varepsilon, h_0}(\omega)} \varepsilon^{N - 1} J_{h_0}(\rho) = \lim_{\varepsilon \to 0} \sum_{(y', \rho) \in C_{\varepsilon, h_0}(\omega)} a_\varepsilon^{N - 1} \rho^{N - 1} = 0 \quad \mathbb{P}\text{-a.s.}
    \end{equation}
\end{proposition}

To keep the computations compact, we introduce the functions $\tilde{J}_0$, $\tilde{J}_\infty$ and $\tilde{J}_{h_0}$ for $h_0 \in (0, \infty)$ as follows:
\begin{alignat*}{3}
    \tilde{J}_0 &: (0, \infty)^2 \rightarrow [0, \infty), &\quad  &\tilde{J}_0(\rho, l) := 2 \rho^{N - 3} \cpct_0(T', B'(0, l)), \\
    \tilde{J}_{h_0} &: (0, \infty)^3 \rightarrow [0, \infty), &\quad &\tilde{J}_{h_0}(\rho, l, h) := \rho^{N - 2} \cpct_h(T', B'(0, l)), \\
    \tilde{J}_\infty &: (0, \infty)^3 \rightarrow [0, \infty), &\quad &\tilde{J}_\infty(\rho, l, h) := \rho^{N - 2} \cpct(T' \times \{0\}, C(l, h)).
\end{alignat*}
Although the definition of $\tilde{J}_{h_0}$ does not depend on $h_0$ for $h_0 \in (0, \infty)$, the notation is convenient when expressing estimates uniformly in $h_0$. In the next lemma, we show that we can bound $\tilde{J}_{h_0}$ from above by $J_{h_0}$ for all $h_0 \in [0, \infty]$.

\begin{lemma} \label{lm:comparison_old_and_new}
    There exists a positive number $C$, depending only on $T'$ and $N$, such that for all $\rho > 0$,
    \begin{alignat}{3} \label{eq:upper_bound_zero}
        \tilde{J}_0(\rho, l) &\le C J_0(\rho), &\quad &\text{for all } l \ge 2, \\ \label{eq:upper_bound_finite}
        \tilde{J}_{h_0}\left(\rho, l, \frac{h}{\rho}\right) &\le C \max \left\{\frac{h_0}{h}, \frac{h}{h_0}\right\} J_{h_0}(\rho), &\quad &\text{for all } l \ge 2, \, h > 0, \\ \label{eq:upper_bound_infinity}
        \tilde{J}_\infty(\rho, l, h) &\le C J_\infty(\rho), &\quad &\text{for all } l, h \ge 2.
    \end{alignat}
    
\end{lemma}

\begin{proof}
    If $h_0 = \infty$ and $l, h \ge 2$, then
    \[
    \cpct(T' \times \{0\}, C(l, h)) \le \cpct(T' \times \{0\}, C(2, 2)) \lesssim \cpct(T' \times \{0\}, \mathbb{R}^N).
    \]
    Hence, \eqref{eq:upper_bound_infinity} follows immediately. If $h_0 = 0$, then
    \[
    \cpct_0(T', B'(0, l)) \le \cpct_0(T', B'(0, 2)) \lesssim \cpct_0(T', \mathbb{R}^{N - 1}),
    \]
    from which we deduce \eqref{eq:upper_bound_zero}. Assume $h_0 \in (0, \infty)$ and $l \ge 2$, $h > 0$. By Proposition \ref{pr:continuity_in_h}
    \[
    \cpct_\frac{h}{\rho}(T', B'(0, l)) \le \cpct_\frac{h}{\rho}(T', B'(0, 2)) \le \max \left\{\frac{h_0}{h}, \frac{h}{h_0}\right\} \cpct_\frac{h_0}{\rho}(T', B'(0, 2))
    \]
    We proceed to show that
    \[
    \sup_{\rho > 0} \frac{\cpct_\frac{h_0}{\rho}(T', B'(0, 2))}{\cpct_\frac{h_0}{\rho}(T', \mathbb{R}^{N - 1})} < \infty.
    \]
    Since $h_0 \in (0, \infty)$ is only possible for $N > 3$ and $\cpct(T' \times \{0\}, \mathbb{R}^N) > 0$, Remark \ref{rmk:jointly_equal_to_zero_2} implies that the quotient is finite for any $\rho$. By Proposition \ref{pr:limit_at_infinity}
    \[
    \lim_{\rho \to 0} \frac{\cpct_\frac{h_0}{\rho}(T', B'(0, 2))}{\cpct_\frac{h_0}{\rho}(T', \mathbb{R}^{N - 1})} = \frac{\cpct(T' \times \{0\}, B'(0, 2) \times \mathbb{R})}{\cpct(T' \times \{0\}, \mathbb{R}^N)}.
    \]
    On the other hand, by Proposition \ref{pr:continuity_in_h}
    \[
    \frac{\cpct_\frac{h_0}{\rho}(T', B'(0, 2))}{\cpct_\frac{h_0}{\rho}(T', \mathbb{R}^{N - 1})} = \frac{\cpct_\frac{h_0}{\rho}(T', B'(0, 2))}{2\frac{h_0}{\rho}} \frac{2\frac{h_0}{\rho}}{\cpct_\frac{h_0}{\rho}(T', \mathbb{R}^{N - 1})} \to \frac{\cpct_0(T', B'(0, 2))}{\cpct_0(T', \mathbb{R}^{N - 1})}
    \]
    as $\rho \to \infty$. As the limits for $\rho \to 0$ and $\rho \to \infty$ exist, and the capacity function is continuous for $\rho \in (0, \infty)$ by Proposition \ref{pr:continuity_in_h}, we conclude that the quotient is globally bounded.
\end{proof}

We can now prove our main result.

\begin{proof}[Proof of Theorem \ref{thm:vanishing_capacity_and_measure}]
    To begin with, assume $h_0 = \infty$. By the subadditivity and the scaling property of the capacity, we have
    \begin{align*}
        \cpct \big((T'_{\varepsilon, \infty})^C(\omega), S_{\varepsilon, \infty}(\omega)\big) &\le \sum_{(y', \rho) \in C_{\varepsilon, \infty}(\omega)} \cpct\left(\varepsilon y + a_\varepsilon \rho (T' \times \{0\}), S_{\varepsilon, \infty}(\omega))\right) \\
        &\le \sum_{(y', \rho) \in C_{\varepsilon, \infty}(\omega)} \cpct\left(\varepsilon y + a_\varepsilon \rho (T' \times \{0\}), B'(\varepsilon y', 2a_\varepsilon \rho) \times (-2a_\varepsilon \rho, 2a_\varepsilon \rho)\right) \\
        &= \sum_{(y', \rho) \in C_{\varepsilon, \infty}(\omega)} a_\varepsilon^{N - 2} \tilde{J}_\infty(\rho, 2, 2) \lesssim \sum_{(y', \rho) \in C_{\varepsilon, \infty}(\omega)} \varepsilon^{N - 1} \delta_\varepsilon J_\infty(\rho)
    \end{align*}
    Then \eqref{eq:capacity_vanishes_infinity} follows from the first limit in Proposition \ref{pr:weighted_sums_go_to_zero} by dividing by $\delta_\varepsilon$ and letting $\varepsilon \to 0$. Next, let $h_0 < \infty$. A similar computation yields
    \begin{equation*}
        \cpct_{\delta_\varepsilon}\left((T'_{\varepsilon, h_0})^C(\omega), S'_{\varepsilon, h_0}(\omega)\right) \\ 
        \le \sum_{(y', \rho) \in C_{\varepsilon, h_0}(\omega)} a_\varepsilon^{N - 2} \rho^{N - 2} \cpct_{\frac{h_\varepsilon}{\rho}}(T', B'(0, 2)).
    \end{equation*}
    If $h_0 \in (0, \infty)$, then by Lemma \ref{lm:comparison_old_and_new},
    \[
    a_\varepsilon^{N - 2} \rho^{N - 2} \cpct_{\frac{h_\varepsilon}{\rho}}(T', B'(0, 2)) = a_\varepsilon^{N - 2} \tilde{J}_{h_0}\left(\rho, 2, \frac{h_\varepsilon}{\rho}\right) \lesssim \max \left\{\frac{h_\varepsilon}{h_0}, \frac{h_0}{h_\varepsilon}\right\} \varepsilon^{N - 1} \delta_\varepsilon J_{h_0}(\rho).
    \]
    On the other hand, if $h_0 = 0$, then
    \[
    a_\varepsilon^{N - 2} \rho^{N - 2} \cpct_{\frac{h_\varepsilon}{\rho}}(T', B'(0, 2)) = a_\varepsilon^{N - 3} \delta_\varepsilon \rho^{N - 3} \frac{1}{2\frac{h_\varepsilon}{\rho}}\cpct_{\frac{h_\varepsilon}{\rho}}(T', B'(0, 2)) \le \varepsilon^{N - 1} \delta_\varepsilon \tilde{J}_0(\rho, 2) \lesssim \varepsilon^{N - 1} \delta_\varepsilon J_0(\rho),
    \]
    since $\cpct_h(T', B'(0, 2))/(2h) \le \cpct_0(T', B'(0, 2))$ for all $h$. Therefore, for all $h_0 < \infty$, we have
    \begin{gather*}
        \limsup_{\varepsilon \to 0} \frac{1}{\delta_\varepsilon} \cpct_{\delta_\varepsilon}\left((T'_{\varepsilon, h_0})^C(\omega), S'_{\varepsilon, h_0}(\omega)\right) \lesssim \limsup_{\varepsilon \to 0} \sum_{(y', \rho) \in C_{\varepsilon, h_0}(\omega)} \varepsilon^{N - 1} J_{h_0}(\rho).
    \end{gather*}
    Hence, \eqref{eq:capacity_vanishes_noninfinity} follows from Proposition \ref{pr:weighted_sums_go_to_zero}. To prove \eqref{eq:measure_vanishes}, we use subadditivity again:
    \[
    |S'_{\varepsilon, h_0}(\omega)| \le \sum_{(y', \rho) \in C_{\varepsilon, h_0}(\omega)} |B'(\varepsilon y', 2 a_\varepsilon \rho)| \lesssim \sum_{(y', \rho) \in C_{\varepsilon, h_0}(\omega)} a_\varepsilon^{N - 1} \rho^{N - 1}.
    \]
    We obtain the result by the second limit in Proposition \ref{pr:weighted_sums_go_to_zero}.
\end{proof}

In order to prove Proposition \ref{pr:weighted_sums_go_to_zero}, we consider cluster points that violate conditions (1a) or (1b) in Definition \ref{def:isolated_and_cluster_points} separately from the rest.
We define
\begin{equation} \label{eq:cluster_points_classification}
    \begin{aligned}
        C_{\varepsilon, \infty}^1(\omega) &:= \left\{(y', \rho) \in C_{\varepsilon, \infty}(\omega) : 2 a_\varepsilon \rho \ge \min \{\varepsilon r(\omega, y'), \delta_\varepsilon\} \right\}, \\
        C_{\varepsilon, h_0}^1(\omega) &:= \left\{(y', \rho) \in C_{\varepsilon, h_0}(\omega) : 2 a_\varepsilon \rho \ge \varepsilon r(\omega, y') \right\} \quad \text{if } h_0 < \infty, \\
        C_{\varepsilon, h_0}^2(\omega) &:= C_{\varepsilon, h_0}(\omega) \setminus C_{\varepsilon, h_0}^1(\omega) \quad \text{for all } h_0 \in [0, \infty]
    \end{aligned}
\end{equation}
for all $\varepsilon > 0$ and $\omega \in \Omega$. As we will see in the proof of Lemma \ref{lm:negligible_cluster_1}, $C_{\varepsilon, h_0}^1(\omega)$ consists of points $(y', \rho)$ with very large $\rho$ or very small $r(\omega, y')$. To be more quantitative, we introduce the notion of thinning of a marked point process. Let $\sigma > 0$. We define a new marked point process $M_\sigma : (\Omega, \mathcal{F}, \mathbb{P}) \rightarrow (\mathbb{M}^{N - 1}, \mathcal{M}^{N - 1})$ as follows: given $\omega \in \Omega$, the point $(y', \rho)$ belongs to $M_\sigma(\omega)$ if and only if $(y', \rho) \in M(\omega)$ and
\[
\min \left\{\frac{d_{M(\omega)}(y')}{2}, \frac{1}{\rho}\right\} < \sigma.
\]
We call $M_\sigma$ the \textit{thinned process}. The relevance of thinned processes in our analysis is mainly motivated by the following result.

\begin{lemma} \label{lm:thinned_limit}
    Let $M :(\Omega, \mathcal{F}, \mathbb{P}) \rightarrow (\mathbb{M}^{N - 1}, \mathcal{M}^{N - 1})$ be a stationary m.p.p. with finite intensity such that
    \[
    \int_0^\infty J_{h_0}(\rho) \, d\lambda(\rho) < \infty.
    \]
    Then
    \[
    \lim_{\sigma \to 0} \, \lim_{\varepsilon \to 0} \sum \limits_{\substack{(y', \rho) \in M_\sigma(\omega) \\ \varepsilon y' \in U'}} \varepsilon^{N - 1} J_{h_0}(\rho) = 0 \quad \mathbb{P}\text{-a.s.}
    \]
\end{lemma}

 A proof of the lemma can be found in Section 4 of \cite{B}. We will need the following result on the growth of $J_{h_0}(\rho)$ for $h_0 \in (0, \infty)$ in the proof of Lemma \ref{lm:negligible_cluster_1}.

\begin{lemma} \label{lm:equivalence_of_functions}
    If $h_0 \in (0, \infty)$, then there exists a positive constant $C$ such that 
    \begin{align} \label{eq:equivalence_of_functions_1}
        C^{-1} \rho^{N - 2} &\le J_{h_0}(\rho) \le C \rho^{N - 2} \quad \text{for } 0 < \rho \le 1, \\ \label{eq:equivalence_of_functions_2}
        C^{-1} \rho^{N - 3} &\le J_{h_0}(\rho) \le C \rho^{N - 3} \quad \text{for } 1 \le \rho < \infty.
    \end{align}
    In particular, $J_{h_0} \in L_\mathrm{loc}^\infty(0, \infty)$.
\end{lemma}

\begin{proof}
    Since $h_0 \in (0, \infty)$, we have $N > 3$. The quotient $J_{h_0}(\rho)/\rho^{N - 2}$ is equal to $\cpct_{h_0/\rho}(T', \mathbb{R}^{N - 1})$, which is continuous and positive for $\rho \in (0, 1]$. In addition, Proposition \ref{pr:limit_at_infinity} yields $\cpct_{h_0/\rho}(T', \mathbb{R}^{N - 1}) \to \cpct(T' \times \{0\}, \mathbb{R}^N) > 0$ as $\rho \to 0$. Hence, $J_{h_0}(\rho)/\rho^{N - 2}$ is bounded away from $0$ on $(0, 1]$. On the other hand, $J_{h_0}(\rho)/\rho^{N - 3}$ equals $\rho \cpct_{h_0/\rho}(T', \mathbb{R}^{N - 1})$, which is continuous and positive for $\rho \ge 1$. Proposition \ref{pr:continuity_in_h} gives $\rho \cpct_{h_0/\rho}(T', \mathbb{R}^{N - 1}) \to 2h_0 \cpct_0(T', \mathbb{R}^{N - 1}) > 0$ as $\rho \to \infty$. Therefore, $J_{h_0}(\rho)/\rho^{N - 3}$ is bounded at infinity. This concludes the proofs of \eqref{eq:equivalence_of_functions_1} and \eqref{eq:equivalence_of_functions_2}. The local boundedness of $J_{h_0}$ follows immediately.
\end{proof}

\begin{lemma} \label{lm:negligible_cluster_1}
    Let $M :(\Omega, \mathcal{F}, \mathbb{P}) \rightarrow (\mathbb{M}^{N - 1}, \mathcal{M}^{N - 1})$ be a stationary m.p.p. with finite intensity such that
    \[
    \int_0^\infty J_{h_0}(\rho) \, d\lambda(\rho) < \infty.
    \]
    Then
    \begin{equation} \label{eq:negligible_cluster_1}
        \lim_{\varepsilon \to 0} \sum_{(y', \rho) \in C_{\varepsilon, h_0}^1(\omega)} \varepsilon^{N - 1} J_{h_0}(\rho) = \lim_{\varepsilon \to 0} \sum_{(y', \rho) \in C_{\varepsilon, h_0}^1(\omega)} a_\varepsilon^{N - 1} \rho^{N - 1} = 0 \quad \mathbb{P}\text{-a.s.}
    \end{equation}
\end{lemma}

\begin{proof}
     Let $\varepsilon > 0$ and $\omega \in \Omega$. If $h_0 < \infty$ and $(y', \rho) \in C_{\varepsilon, h_0}^1(\omega)$, then $r(\omega, y')/\rho \le 2 a_\varepsilon/\varepsilon$ by definition. Consequently, we must have
    \[
    r(\omega, y') \le \left(\frac{2 a_\varepsilon}{\varepsilon}\right)^\frac{1}{2} \quad \text{or} \quad \frac{1}{\rho} \le \left(\frac{2 a_\varepsilon}{\varepsilon}\right)^\frac{1}{2} \quad \text{for all } (y', \rho) \in C_{\varepsilon, h_0}^1(\omega), \, h_0 < \infty.
    \]
    If $h_0 = \infty$, then $1/\rho < 2 a_\varepsilon/\delta_\varepsilon$ might alternatively hold. Therefore,
    \[
    r(\omega, y') \le \left(\frac{2 a_\varepsilon}{\varepsilon}\right)^\frac{1}{2} \quad \text{or} \quad \frac{1}{\rho} \le \max \left\{\left(\frac{2 a_\varepsilon}{\varepsilon}\right)^\frac{1}{2}, \frac{2}{h_\varepsilon}\right\} \quad \text{for all } (y', \rho) \in C_{\varepsilon, \infty}^1(\omega).
    \]
    We observe that $a_\varepsilon/\varepsilon \to 0$ as $\varepsilon \to 0$ for all $h_0 \in [0, \infty]$. Also, for $h_0 = \infty$, we have $h_\varepsilon \to \infty$ by definition. Hence, for all $h_0 \in [0, \infty]$ and $\sigma > 0$, if $\varepsilon$ is small enough, then
    \[
    \min \left\{\frac{d_{M(\omega)}(y')}{2}, \frac{1}{\rho}\right\} < \sigma \quad \text{for all } (y', \rho) \in C_{\varepsilon, h_0}^1(\omega).
    \]
    This yields
    \[
    \limsup_{\varepsilon \to 0} \sum_{(y', \rho) \in C_{\varepsilon, h_0}^1(\omega)} \varepsilon^{N - 1} J_{h_0}(\rho) \le \limsup_{\varepsilon \to 0} \sum \limits_{\substack{(y', \rho) \in M_\sigma(\omega) \\ \varepsilon y' \in U'}} \varepsilon^{N - 1} J_{h_0}(\rho).
    \]
    By passing to the limit in $\sigma$, we obtain
    \[
    \limsup_{\varepsilon \to 0} \sum_{(y', \rho) \in C_{\varepsilon, h_0}^1(\omega)} \varepsilon^{N - 1} J_{h_0}(\rho) \le \lim_{\sigma \to 0} \, \limsup_{\varepsilon \to 0} \sum \limits_{\substack{(y', \rho) \in M_\sigma(\omega) \\ \varepsilon y' \in U'}} \varepsilon^{N - 1} J_{h_0}(\rho) = 0 \quad \mathbb{P}\text{-a.s.},
    \]
    where the last equality follows from Lemma \ref{lm:thinned_limit}. 
    
    To prove that the second limit in \eqref{eq:negligible_cluster_1} is zero, we argue case by case and use the first step. If $h_0 = \infty$, then
    \[
    \Bigg(\sum_{(y', \rho) \in C_{\varepsilon, \infty}^1(\omega)} a_\varepsilon^{N - 1} \rho^{N - 1}\Bigg)^\frac{N - 2}{N - 1} \le \sum_{(y', \rho) \in C_{\varepsilon, \infty}^1(\omega)} a_\varepsilon^{N - 2} \rho^{N - 2} \lesssim \sum_{(y', \rho) \in C_{\varepsilon, \infty}^1(\omega)} \varepsilon^{N - 1} \delta_\varepsilon J_\infty(\rho),
    \]
    where we used $\|\cdot\|_{\ell^{N - 1}} \le \|\cdot\|_{\ell^{N - 2}}$ in the first inequality. For $h_0 \in (0, \infty)$, we split the sum in two parts and use Lemma \ref{lm:equivalence_of_functions}:
    \[
    \Bigg(\sum \limits_{\substack{(y', \rho) \in C_{\varepsilon, h_0}^1(\omega) \\ 0 < \rho \le 1}} a_\varepsilon^{N - 1} \rho^{N - 1}\Bigg)^\frac{N - 2}{N - 1} \le \sum \limits_{\substack{(y', \rho) \in C_{\varepsilon, h_0}^1(\omega) \\ 0 < \rho \le 1}} a_\varepsilon^{N - 2} \rho^{N - 2} \lesssim \sum_{(y', \rho) \in C_{\varepsilon, h_0}^1(\omega)} \varepsilon^{N - 1} \delta_\varepsilon J_{h_0}(\rho)
    \]
    and
    \[
    \Bigg(\sum \limits_{\substack{(y', \rho) \in C_{\varepsilon, h_0}^1(\omega) \\ \rho > 1}} a_\varepsilon^{N - 1} \rho^{N - 1}\Bigg)^\frac{N - 3}{N - 1} \le \sum \limits_{\substack{(y', \rho) \in C_{\varepsilon, h_0}^1(\omega) \\ \rho > 1}} a_\varepsilon^{N - 3} \rho^{N - 3} \lesssim \sum_{(y', \rho) \in C_{\varepsilon, h_0}^1(\omega)} \varepsilon^{N - 1} J_{h_0}(\rho).
    \]
    In the last inequality we used the fact that $h_0 \in (0, \infty)$ implies $a_\varepsilon \sim \delta_\varepsilon$ and $\delta_\varepsilon^{N - 3} \sim \varepsilon^{N - 1}$. Lastly, if $h_0 = 0$, then
    \[
    \Bigg(\sum_{(y', \rho) \in C_{\varepsilon, 0}^1(\omega)} a_\varepsilon^{N - 1} \rho^{N - 1}\Bigg)^\frac{N - 3}{N - 1} \le \sum_{(y', \rho) \in C_{\varepsilon, 0}^1(\omega)} a_\varepsilon^{N - 3} \rho^{N - 3} \lesssim \sum_{(y', \rho) \in C_{\varepsilon, 0}^1(\omega)} \varepsilon^{N - 1} J_0(\rho).
    \]
    In all three cases, the claim follows from the first limit in \eqref{eq:negligible_cluster_1}.
\end{proof}

To prove the corresponding claim for $C_{\varepsilon, h_0}^2(\omega)$, we show that its points cluster around points in $C_{\varepsilon, h_0}^1(\omega)$.

\begin{lemma} \label{lm:cluster_within_cluster}
    Let $\varepsilon > 0$ and $\omega \in \Omega$. For all $(y', \rho) \in C_{\varepsilon, h_0}^2(\omega)$, there exists $(\tilde{y}', \tilde{\rho}) \in C_{\varepsilon, h_0}^1(\omega)$ such that
    \begin{equation} \label{eq:absorbed_by_cluster_point}
        B'(\varepsilon y', \varepsilon r(\omega, y')) \subset B'\left(\varepsilon \tilde{y}', 6 a_\varepsilon \tilde{\rho}\right).
    \end{equation}
\end{lemma}

\begin{proof}
    Assume $(y', \rho) \in C_{\varepsilon, h_0}^2(\omega)$. As $(y', \rho)$ violates condition $(2)$ in Definition \ref{def:isolated_and_cluster_points}, there exists a point $(\tilde{y}', \tilde{\rho}) \in M(\omega)$ with $\varepsilon \tilde{y}' \in U'$ such that
    \begin{equation} \label{eq:cluster_equation}
        B'(\varepsilon \tilde{y}', 2 a_\varepsilon \tilde{\rho}) \cap B'(\varepsilon y', \varepsilon r(\omega, y')) \neq \emptyset \iff \varepsilon |y' - \tilde{y}'| < \varepsilon r(\omega, y') + 2 a_\varepsilon \tilde{\rho}.
    \end{equation}
    From \eqref{eq:nonintersecting_scaled_balls} and \eqref{eq:cluster_equation} we see that $2 a_\varepsilon \tilde{\rho} > \varepsilon r(\omega, \tilde{y}')$. Therefore, $(\tilde{y}', \tilde{\rho}) \in C_{\varepsilon, h_0}^1(\omega)$. Since $r(\omega, y') \le |y' - \tilde{y}'|/2$, \eqref{eq:cluster_equation} implies $\varepsilon |y' - \tilde{y}'| < 4 a_\varepsilon \tilde{\rho}$. Thus, $B'(\varepsilon y', \varepsilon r(\omega, y')) \subset B'(\varepsilon \tilde{y}', 6 a_\varepsilon \tilde{\rho})$.
\end{proof}

\begin{lemma} \label{lm:negligible_cluster_2}
    Let $M :(\Omega, \mathcal{F}, \mathbb{P}) \rightarrow (\mathbb{M}^{N - 1}, \mathcal{M}^{N - 1})$ be a stationary m.p.p. with finite intensity such that
    \[
    \int_0^\infty J_{h_0}(\rho) \, d\lambda(\rho) < \infty.
    \]
    Then
    \[
    \lim_{\varepsilon \to 0} \sum_{(y', \rho) \in C_{\varepsilon, h_0}^2(\omega)} \varepsilon^{N - 1} J_{h_0}(\rho) = 0 \quad \mathbb{P}\text{-a.s.}
    \]
\end{lemma}

\begin{proof}
    We prove
    \[
    \lim_{\varepsilon \to 0} \sum_{(y', \rho) \in C_{\varepsilon, h_0}^2(\omega) \setminus M_\sigma(\omega)} \varepsilon^{N - 1} J_{h_0}(\rho) = 0 \quad \mathbb{P}\text{-a.s.}
    \]
    for $\sigma < 1$. Let $\varepsilon > 0$ and $\omega \in \Omega$. As $\rho \le 1/\sigma$ for all $(y', \rho) \in M(\omega) \setminus M_\sigma(\omega)$, we have
    \begin{equation} \label{eq:bound_by_cardinality}
        \sum_{(y', \rho) \in C_{\varepsilon, h_0}^2(\omega) \setminus M_\sigma(\omega)} \varepsilon^{N - 1} J_{h_0}(\rho) \le \varepsilon^{N - 1} \|J_{h_0}\|_{L^\infty(0, 1/\sigma)} \card\left(C_{\varepsilon, h_0}^2(\omega) \setminus M_\sigma(\omega)\right).
    \end{equation}
    Fixing $\sigma$ and $\omega$ for the moment, let us define
    \[
    A(\tilde{y}', \tilde{\rho}) := \left\{(y', \rho) \in C_{\varepsilon, h_0}^2(\omega) \setminus M_\sigma(\omega) : B'(\varepsilon y', \varepsilon r(\omega, y')) \subset B'(\varepsilon \tilde{y}', 6 a_\varepsilon \tilde{\rho})\right\}.
    \]
    It follows from Lemma \ref{lm:cluster_within_cluster} that
    \begin{equation} \label{eq:cardinality_estimate}
        \card\left(C_{\varepsilon, h_0}^2(\omega) \setminus M_\sigma(\omega)\right) \le \sum_{(\tilde{y}', \tilde{\rho}) \in C_{\varepsilon, h_0}^1(\omega)} \card(A(\tilde{y}', \tilde{\rho})).
    \end{equation}
    To estimate $\card(A(\tilde{y}', \tilde{\rho}))$, we sum the volumes of the balls $B'(\varepsilon y', \varepsilon r(\omega, y'))$ for $(y', \rho) \in A(\tilde{y}', \tilde{\rho})$. Since they are disjoint and their union is contained in $B'(\varepsilon \tilde{y}', 6 a_\varepsilon \tilde{\rho})$, we get
    \[
    \sum_{(y', \rho) \in A(\tilde{y}', \tilde{\rho})} \varepsilon^{N - 1} r(\omega, y')^{N - 1} \le (6 a_\varepsilon \tilde{\rho})^{N - 1}.
    \]
    As $r(\omega, y') \ge \sigma$ for $(y', \rho) \in M(\omega) \setminus M_\sigma(\omega)$ and $A(\tilde{y}', \tilde{\rho}) \subset M(\omega) \setminus M_\sigma(\omega)$, we conclude that
    \[
    \card(A(\tilde{y}', \tilde{\rho})) \le \left(\frac{6 a_\varepsilon \tilde{\rho}}{\varepsilon \sigma}\right)^{N - 1}.
    \]
    Combining this with \eqref{eq:bound_by_cardinality} and \eqref{eq:cardinality_estimate} yields
    \begin{equation*}
        \sum_{(y', \rho) \in C_{\varepsilon, h_0}^2(\omega) \setminus M_\sigma(\omega)} \varepsilon^{N - 1} J_{h_0}(\rho) \le \|J_{h_0}\|_{L^\infty(0, 1/\sigma)}\left(\frac{6}{\sigma}\right)^{N - 1} \sum_{(\tilde{y}', \tilde{\rho}) \in C_{\varepsilon, h_0}^1(\omega)} a_\varepsilon^{N - 1} \tilde{\rho}^{N - 1}.
    \end{equation*}
    Thus, Lemma \ref{lm:negligible_cluster_1} and the local boundedness of $J_{h_0}$ implies
    \[
    \limsup_{\varepsilon \to 0} \sum_{(y', \rho) \in C_{\varepsilon, h_0}^2(\omega) \setminus M_\sigma(\omega)} \varepsilon^{N - 1} J_{h_0}(\rho) = 0 \quad \mathbb{P}\text{-a.s.}
    \]
    With the help of Lemma \ref{lm:thinned_limit}, we conclude the proof.
\end{proof}

\begin{proof}[Proof of Proposition \ref{pr:weighted_sums_go_to_zero}]
    The first limit in \eqref{eq:weighted_sums_go_to_zero} is zero by Lemmas \ref{lm:negligible_cluster_1} and \ref{lm:negligible_cluster_2}. The second limit follows from the first exactly as in the proof of Lemma \ref{lm:negligible_cluster_1}.
\end{proof}

\section{The oscillating test functions} \label{sec:otf}

In this section, we construct the oscillating test functions $(w_\varepsilon(\omega, \cdot))$ that satisfy Theorem \ref{thm:oscillating_test_functions}, and we prove their properties. The construction is based on the solutions to certain cell problems.

\subsection{The cell problem}

For $h_0 \in [0, \infty]$ and $l, h > 0$, we define a family of functions $\eta_{l, h}^{h_0} \in H^1(C(l, h))$ as solutions of boundary value problems. Let $\eta_{l, h}^\infty$ be the unique weak solution of
\begin{equation} \label{eq:cell_problem_1}
    \begin{cases}
        \begin{aligned}
            -\Delta \eta_{l, h}^\infty &= 0 \quad \text{in } C(l, h), \\
            \eta_{l, h}^\infty &= 0 \quad \text{on } \partial C(l, h), \\
            \eta_{l, h}^\infty &= 1 \quad \text{on } T' \times \{0\}. \\
        \end{aligned}
    \end{cases}
\end{equation}
For $h_0 \in (0, \infty)$, let $\eta_{l, h}^{h_0}$ be the unique weak solution of
\begin{equation} \label{eq:cell_problem_2}
    \begin{cases}
        \begin{aligned}
            -\Delta \eta_{l, h}^{h_0} &= 0 \quad \text{in } C(l, h), \\
            \eta_{l, h}^{h_0} &= 0 \quad \text{on } \partial B'(0, l) \times (-h, h), \\
            \eta_{l, h}^{h_0} &= 1 \quad \text{on } T' \times \{0\}, \\
            \nabla \eta_{l, h}^{h_0} \cdot \nu &= 0 \quad \text{on } (B'(0, l) \times \{-h, h\}).
        \end{aligned}
    \end{cases}
\end{equation}
Finally, for $h_0 = 0$ and $(x', x_N) \in C(l, h)$, we define $\eta_{l, h}^0(x', x_N) := \eta_l^0(x')$, where $\eta_l^0$ is the unique weak solution of
\begin{equation} \label{eq:cell_problem_3}
    \begin{cases}
        \begin{aligned}
            -\Delta \eta_l^0 &= 0 \quad \text{in } B'(0, l) \setminus T', \\
            \eta_l^0 &= 1 \quad \text{on } \partial B'(0, l), \\
            \eta_l^0 &= 0 \quad \text{on } T'.
        \end{aligned}
    \end{cases}
\end{equation}
It is easy to verify that the solutions of the boundary value problems are capacitary potentials:
\begin{equation} \label{eq:capacitary_potentials}
    \begin{aligned}
        \int_{C(l, h)} |\nabla \eta_{l, h}^\infty|^2 \, dx &= \cpct(T' \times \{0\}, C(l, h)), \\
        \int_{C(l, h)} |\nabla \eta_{l, h}^{h_0}|^2 \, dx &= \cpct_h(T', B'(0, l)) \quad \text{for } h_0 \in (0, \infty), \\
        \int_{B'(0, l)} |\nabla \eta_l^0|^2 \, dx &= \cpct_0(T', B'(0, l)).
    \end{aligned}
\end{equation}
Furthermore, due to symmetry, we have that
\begin{equation} \label{eq:symmetry_considerations}
    \int_{B'(0, l) \times (0, h)} |\nabla \eta_{l, h}^{h_0}|^2 \, dx = \frac{1}{2} \int_{C(l, h)} |\nabla \eta_{l, h}^{h_0}|^2 \, dx \quad \text{for all } h_0 \in [0, \infty].
\end{equation}
Finally, we note that $0 \le \eta_{l, h}^{h_0} \le 1$ a.e. in $C(l, h)$ by the weak maximum principle.

\subsection{Construction of the random oscillating test functions} \label{ssec:otf_construction}

For the remainder of this section, we fix a stationary m.p.p. $M :(\Omega, \mathcal{F}, \mathbb{P}) \rightarrow (\mathbb{M}^{N - 1}, \mathcal{M}^{N - 1})$ with finite intensity such that
\[
\int_0^\infty J_{h_0}(\rho) \, d\lambda(\rho) < \infty.
\]
Let $\varepsilon > 0$ and $\omega \in \Omega$ be fixed. Given $(y', \rho) \in I_{\varepsilon, h_0}(\omega)$, we define $w_\varepsilon(\omega, x)$ for $x \in B'(\varepsilon y', \varepsilon r(\omega, y')) \times (0, \delta_\varepsilon)$ by
\begin{equation} \label{eq:local_definition}
    w_\varepsilon(\omega, x) := 1 - \eta_{l, h}^{h_0}\left(\frac{x - \varepsilon y}{a_\varepsilon \rho}\right), \quad l := \frac{\varepsilon r(\omega, y')}{a_\varepsilon \rho}, \, h := \frac{h_\varepsilon}{\rho}.
\end{equation}
This defines $w_\varepsilon(\omega, x)$ uniquely for
\[
x \in \bigcup_{(y', \rho) \in I_{\varepsilon, h_0}(\omega)} B'(\varepsilon y', \varepsilon r(\omega, y')) \times (0, \delta_\varepsilon),
\]
since it is a disjoint union. We note that
\begin{equation} \label{eq:lateral_boundary_values_1}
    w_\varepsilon(\omega, x) = 1 \quad \text{for } x \in \bigcup_{(y', \rho) \in I_{\varepsilon, h_0}(\omega)} \partial B'(\varepsilon y', \varepsilon r(\omega, y')) \times (0, \delta_\varepsilon).
\end{equation}
Next, we define $w_\varepsilon(\omega, x)$ for $x \in S'_{\varepsilon, h_0}(\omega) \times (0, \delta_\varepsilon)$. First, we consider the case $h_0 = \infty$. Choose $\zeta_{\varepsilon, \infty} \in C_c^\infty(S_{\varepsilon, \infty}(\omega); [0, 1])$ such that $(T'_{\varepsilon, \infty})^C(\omega) \times \{0\} \subset \topint(\{\zeta_{\varepsilon, \infty} = 1\})$ and
\[
\int_{S_{\varepsilon, \infty}(\omega)} |\zeta_{\varepsilon, \infty}|^2 \, dx \le 2 \cpct\left((T'_{\varepsilon, \infty})^C(\omega) \times \{0\}, S_{\varepsilon, \infty}(\omega)\right).
\]
Set $w_\varepsilon(\omega, x) := 1 - \zeta_{\varepsilon, \infty}(x)$ for $x \in S'_{\varepsilon, \infty}(\omega) \times (0, \delta_{\varepsilon}) \subset S_{\varepsilon, \infty}(\omega)$. By \eqref{eq:separation_layer}, the previous definition of $w_\varepsilon(\omega, \cdot)$ does not interfere with the one in $S_{\varepsilon, \infty}(\omega)$. We also have
\begin{equation} \label{eq:vanishing_gradient_1}
    \frac{1}{\delta_\varepsilon} \int_{S'_{\varepsilon, \infty}(\omega) \times (0, \delta_\varepsilon)} |\nabla w_\varepsilon(\omega, x)|^2 \, dx \le \frac{2}{\delta_\varepsilon} \cpct\left((T'_{\varepsilon, \infty})^C(\omega) \times \{0\}, S_{\varepsilon, \infty}(\omega)\right).
\end{equation}
For $h_0 < \infty$ we similarly choose $\zeta_{\varepsilon, h_0} \in C_{\delta_\varepsilon}^\infty(S'_{\varepsilon, h_0}(\omega); [0, 1])$ such that $(T'_{\varepsilon, h_0})^C(\omega) \times \{0\} \subset \topint(\{\zeta_{\varepsilon, h_0} = 1\})$ and
\[
\int_{S'_{\varepsilon, h_0}(\omega) \times (-\delta_\varepsilon, \delta_\varepsilon)} |\nabla \zeta_{\varepsilon, h_0}|^2 \, dx \le 2 \cpct_{\delta_\varepsilon}\left((T'_{\varepsilon, h_0})^C(\omega), S'_{\varepsilon, h_0}(\omega)\right).
\]
We set $w_\varepsilon(\omega, x) := 1 - \zeta_{\varepsilon, h_0}(x)$ for $x \in S'_{\varepsilon, h_0}(\omega) \times (0, \delta_\varepsilon)$. Then
\begin{equation} \label{eq:vanishing_gradient_2}
    \frac{1}{\delta_\varepsilon} \int_{S'_{\varepsilon, h_0}(\omega) \times (0, \delta_\varepsilon)} |\nabla w_\varepsilon(\omega, x)|^2 \, dx \le \frac{2}{\delta_\varepsilon} \cpct_{\delta_\varepsilon}\left((T'_{\varepsilon, h_0})^C(\omega), S'_{\varepsilon, h_0}(\omega)\right).
\end{equation}
This completes the definition in $S'_{\varepsilon, h_0} \times (0, \delta_\varepsilon)$. Observe that 
\begin{equation} \label{eq:lateral_boundary_values_2}
    w_\varepsilon(\omega, x) = 1 \quad \text{for } x \in \partial S'_{\varepsilon, h_0}(\omega) \times (0, \delta_\varepsilon) \quad \text{for all } h_0 \in [0, \infty].
\end{equation}
Finally, we extend $w_\varepsilon(\omega, \cdot)$ by 1 to $\mathbb{R}^{N - 1} \times (0, \delta_\varepsilon)$. It is easy to verify that $w_\varepsilon(\omega, \cdot)$ is weakly differentiable in $\mathbb{R}^{N - 1} \times (0, \delta_\varepsilon)$ and it belongs to $H^1(U_\varepsilon^+)$ with $0 \le w_\varepsilon(\omega, \cdot) \le 1$ a.e. in $\mathbb{R}^{N - 1} \times (0, \delta_\varepsilon)$. Furthermore, it vanishes weakly on $T'_\varepsilon(\omega)$.

In the remainder of this section, we prove that for $\mathbb{P}$-a.e. $\omega \in \Omega$, $(w_\varepsilon(\omega, \cdot))$ is $\omega$-admissible. To this end, we first compute the limit of the scaled $L^2$-norm of the gradients locally. Let $(y', \rho) \in I_{\varepsilon, h_0}(\omega)$ and set $l := \varepsilon r(\omega, y')/(a_\varepsilon \rho)$, $h := h_\varepsilon/\rho$. Applying a change of variables to \eqref{eq:local_definition} and using \eqref{eq:capacitary_potentials} and \eqref{eq:symmetry_considerations} gives
\begin{equation*}
    \int_{B'(\varepsilon y', \varepsilon r(\omega, y')) \times (0, \delta_\varepsilon)} |\nabla w_\varepsilon(\omega, x)|^2 \, dx = 
    \begin{cases}
        \frac{1}{2} a_\varepsilon^{N - 2} \rho^{N - 2} \cpct(T' \times \{0\}, C(l, h)) \quad \text{if } h_0 = \infty, \\
        \frac{1}{2} a_\varepsilon^{N - 2} \rho^{N - 2} \cpct_h(T', B'(0, l)) \quad \text{if } h_0 \in (0, \infty), \\
        a_\varepsilon^{N - 3} \delta_\varepsilon \rho^{N - 3} \cpct_0(T', B'(0, l)) \quad \text{if } h_0 = 0.
    \end{cases}
\end{equation*}
Dividing by $\delta_\varepsilon$ and rewriting the right-hand side in terms of the function $\tilde{J}_{h_0}$ yields
\begin{equation} \label{eq:gradient_divided_by_delta}
    \frac{1}{\delta_\varepsilon} \int_{B'(\varepsilon y', \varepsilon r(\omega, y')) \times (0, \delta_\varepsilon)} |\nabla w_\varepsilon(\omega, x)|^2 \, dx = 
    \begin{cases}
        \frac{1}{2} \varepsilon^{N - 1} \tilde{J}_{h_0}(\rho, l, h) \quad \text{if } h_0 \in (0, \infty], \\
        \frac{1}{2} \varepsilon^{N - 1} \tilde{J}_0(\rho, l) \quad \text{if } h_0 = 0.
    \end{cases}
\end{equation}
In the next lemma, we sum the right-hand side in \eqref{eq:gradient_divided_by_delta} over certain subsets of $I_{\varepsilon, h_0}(\omega)$ and evaluate its limit as $\varepsilon \to 0$. The computation is an application of the ergodic theorem.

\begin{lemma} \label{lm:really_ugly_sum_limit}
    Let $O' \subset \mathbb{R}^{N - 1}$ be a Borel measurable set such that $|\partial O'| = 0$. Then,
    \begin{equation} \label{eq:really_ugly_sum_limit_zero}
        \lim_{\varepsilon \to 0} \frac{1}{2} \sum \limits_{\substack{(y', \rho) \in I_{\varepsilon, h_0}(\omega) \\ \varepsilon y' \in O'}} \varepsilon^{N - 1} \tilde{J}_0\left(\rho, \frac{\varepsilon r(\omega, y')}{a_\varepsilon \rho}\right) = \gamma_0(\omega) |U' \cap O'| \quad \mathbb{P}\text{-a.s.}
    \end{equation}
    If $h_0 \in (0, \infty]$, then
    \begin{equation} \label{eq:really_ugly_sum_limit_nonzero}
        \lim_{\varepsilon \to 0} \frac{1}{2} \sum \limits_{\substack{(y', \rho) \in I_{\varepsilon, h_0}(\omega) \\ \varepsilon y' \in O'}} \varepsilon^{N - 1} \tilde{J}_{h_0}\left(\rho, \frac{\varepsilon r(\omega, y')}{a_\varepsilon \rho}, \frac{h_\varepsilon}{\rho}\right) = \gamma_{h_0}(\omega) |U' \cap O'| \quad \mathbb{P}\text{-a.s.}
    \end{equation}
\end{lemma}

\begin{proof}
    First we prove that
    \begin{equation} \label{eq:fewer_arguments_limit}
        \lim_{\varepsilon \to 0} \frac{1}{2} \sum \limits_{\substack{(y', \rho) \in I_{\varepsilon, h_0}(\omega) \\ \varepsilon y' \in O'}} \varepsilon^{N - 1} J_{h_0}(\rho) = \gamma_{h_0}(\omega) |U' \cap O'| \quad \mathbb{P}\text{-a.s.}
    \end{equation}
    We observe that
    \begin{align*}
        \sum \limits_{\substack{(y', \rho) \in M(\omega) \\ \varepsilon y' \in U' \cap O'}} \varepsilon^{N - 1} J_{h_0}(\rho) &\ge \sum \limits_{\substack{(y', \rho) \in I_{\varepsilon, h_0}(\omega) \\ \varepsilon y' \in O'}} \varepsilon^{N - 1} J_{h_0}(\rho) \\
        &= \sum \limits_{\substack{(y', \rho) \in M(\omega) \\ \varepsilon y' \in U' \cap O'}} \varepsilon^{N - 1} J_{h_0}(\rho) - \sum \limits_{\substack{(y', \rho) \in C_{\varepsilon, h_0}(\omega) \\ \varepsilon y' \in O'}} \varepsilon^{N - 1} J_{h_0}(\rho).
    \end{align*}
    Applying Theorem \ref{thm:ergodic_theorem_second_version} with $g(\rho) = J_{h_0}(\rho)$ gives
    \begin{equation} \label{eq:simpler_limit}
        \lim_{\varepsilon \to 0} \sum \limits_{\substack{(y', \rho) \in M(\omega) \\ \varepsilon y' \in U' \cap O'}} \varepsilon^{N - 1} J_{h_0}(\rho) = |U' \cap O'| \int_0^\infty J_{h_0}(\rho) \, d\xi(\omega, \rho) = 2 \gamma_{h_0}(\omega) |U' \cap O'| \quad \mathbb{P}\text{-a.s.}
    \end{equation}
    Hence, \eqref{eq:fewer_arguments_limit} follows from \eqref{eq:simpler_limit} and Proposition \ref{pr:weighted_sums_go_to_zero}. Next we use \eqref{eq:fewer_arguments_limit} to obtain lower and upper bounds for the limits in \eqref{eq:really_ugly_sum_limit_zero} and \eqref{eq:really_ugly_sum_limit_nonzero}. We note that
    \begin{equation*}
        \begin{aligned}
            \tilde{J}_{h_0}\left(\rho, \frac{\varepsilon r(\omega, y')}{a_\varepsilon \rho}, \frac{h_\varepsilon}{\rho}\right) &\ge 
            \begin{cases}
                J_\infty(\rho) \quad &\text{if } h_0 = \infty, \\
                \min \left\{\frac{h_\varepsilon}{h_0}, \frac{h_0}{h_\varepsilon}\right\} J_{h_0}(\rho) \quad &\text{if } h_0 \in (0, \infty),
            \end{cases}
            \\
            \tilde{J}_0\left(\rho, \frac{\varepsilon r(\omega, y')}{a_\varepsilon \rho}\right) &\ge J_0(\rho).
        \end{aligned}
    \end{equation*}
    The inequalities above follow from the monotonicity of the capacity with respect to set inclusion and, for $h_0 \in (0, \infty)$, from Proposition \ref{pr:continuity_in_h}. If $h_0 \in (0, \infty]$, then from \eqref{eq:fewer_arguments_limit} we get the lower bound
    \begin{multline} \label{eq:liminf_bound}
        \liminf_{\varepsilon \to 0} \frac{1}{2} \sum \limits_{\substack{(y', \rho) \in I_{\varepsilon, h_0}(\omega) \\ \varepsilon y' \in O'}} \varepsilon^{N - 1} \tilde{J}_{h_0}\left(\rho, \frac{\varepsilon r(\omega, y')}{a_\varepsilon \rho}, \frac{h_\varepsilon}{\rho}\right) \\
        \ge \liminf_{\varepsilon \to 0} \frac{1}{2} \sum \limits_{\substack{(y', \rho) \in M(\omega) \\ \varepsilon y' \in U' \cap O'}} \varepsilon^{N - 1} J_{h_0}(\rho) = \gamma_{h_0}(\omega) |U' \cap O'| \quad \mathbb{P}\text{-a.s.}
    \end{multline}
    Similarly, we obtain $\gamma_0(\omega)|U' \cap O'|$ as a lower bound $\mathbb{P}$-a.s. for the limit in \eqref{eq:really_ugly_sum_limit_zero}.

    To prove an upper bound, we fix $h > 0$, $l \ge 2$. We assume $h_0 = \infty$. If $\sigma$ is fixed and $\varepsilon$ is sufficiently small, then every $(y', \rho) \in I_{\varepsilon, \infty}(\omega) \setminus M_\sigma(\omega)$ satisfies
    \begin{equation} \label{eq:better_isolation}
        \frac{\varepsilon r(\omega, y')}{a_\varepsilon \rho} \ge \frac{\varepsilon \sigma^2}{a_\varepsilon} > l, \quad \frac{h_\varepsilon}{\rho} \ge h_\varepsilon \sigma > h.
    \end{equation}
    Therefore, we deduce
    \begin{equation} \label{eq:good_upper_bound}
        \tilde{J}_\infty\left(\rho, \frac{\varepsilon r(\omega, y')}{a_\varepsilon \rho}, \frac{h_\varepsilon}{\rho}\right) \le \tilde{J}_\infty(\rho, l, h) \quad \text{for all } (y', \rho) \in I_{\varepsilon, h_0}(\omega) \setminus M_\sigma(\omega).
    \end{equation}
    On the other hand, by (1a) in Definition \ref{def:isolated_and_cluster_points}, every $(y', \rho) \in I_{\varepsilon, \infty}(\omega)$ satisfies
    \[
    \frac{\varepsilon r(\omega, y')}{a_\varepsilon \rho} \ge 2, \quad \frac{h_\varepsilon}{\rho} \ge 2.
    \]
    Consequently,
    \begin{equation} \label{eq:bad_upper_bound}
        \tilde{J}_\infty\left(\rho, \frac{\varepsilon r(\omega, y')}{a_\varepsilon \rho}, \frac{h_\varepsilon}{\rho}\right) \le \tilde{J}_\infty(\rho, 2 , 2) \quad \text{for all } (y', \rho) \in I_{\varepsilon, h_0}(\omega).
    \end{equation}
    Then \eqref{eq:good_upper_bound} and \eqref{eq:bad_upper_bound} give us the following upper bound for small enough $\varepsilon$:
    \begin{equation} \label{eq:splitting_sum_argument}
        \begin{aligned}
            \sum \limits_{\substack{(y', \rho) \in I_{\varepsilon, \infty}(\omega) \\ \varepsilon y' \in O'}} \varepsilon^{N - 1} &\tilde{J}_\infty\left(\rho, \frac{\varepsilon r(\omega, y')}{a_\varepsilon \rho}, \frac{h_\varepsilon}{\rho}\right) \\
            &\le \sum \limits_{\substack{(y', \rho) \in I_{\varepsilon, \infty}(\omega) \setminus M_\sigma(\omega) \\ \varepsilon y' \in O'}} \varepsilon^{N - 1} \tilde{J}_\infty(\rho, l, h) + \sum_{(y', \rho) \in I_{\varepsilon, \infty}(\omega) \cap M_\sigma(\omega)} \varepsilon^{N - 1} \tilde{J}_\infty(\rho, 2, 2) \\
            &\le \frac{\cpct(T' \times \{0\}, C(l, h))}{\cpct(T' \times \{0\}, \mathbb{R}^N)} \sum \limits_{\substack{(y', \rho) \in M(\omega) \\ \varepsilon y' \in U' \cap O'}} \varepsilon^{N - 1} J_\infty(\rho) + \sum \limits_{\substack{(y', \rho) \in M_\sigma(\omega) \\ \varepsilon y' \in U'}} \varepsilon^{N - 1} \tilde{J}_\infty(\rho, 2, 2)
        \end{aligned}
    \end{equation}
    By Lemma \ref{lm:thinned_limit}, we have
    \[
    \limsup_{\sigma \to 0} \limsup_{\varepsilon \to 0} \sum \limits_{\substack{(y', \rho) \in M_\sigma(\omega) \\ \varepsilon y' \in U'}} \varepsilon^{N - 1} \tilde{J}_\infty(\rho, 2, 2) \lesssim \limsup_{\sigma \to 0} \limsup_{\varepsilon \to 0} \sum \limits_{\substack{(y', \rho) \in M_\sigma(\omega) \\ \varepsilon y' \in U'}} \varepsilon^{N - 1} J_\infty(\rho) = 0 \quad \mathbb{P}\text{-a.s.},
    \]
    so that
    \begin{multline}
        \limsup_{\varepsilon \to 0}\sum \limits_{\substack{(y', \rho) \in I_{\varepsilon, \infty}(\omega) \\ \varepsilon y' \in O'}} \varepsilon^{N - 1} \tilde{J}_\infty\left(\rho, \frac{\varepsilon r(\omega, y')}{a_\varepsilon \rho}, \frac{h_\varepsilon}{\rho}\right)  \\
        \le \frac{\cpct(T' \times \{0\}, C(l, h))}{\cpct(T' \times \{0\}, \mathbb{R}^N)} \limsup_{\varepsilon \to 0} \sum \limits_{\substack{(y', \rho) \in M(\omega) \\ \varepsilon y' \in U' \cap O'}} \varepsilon^{N - 1} J_\infty(\rho) \quad \mathbb{P}\text{-a.s.}
    \end{multline}
    We can let $h, l \to \infty$ and use \eqref{eq:simpler_limit} to obtain
    \[
    \limsup_{\varepsilon \to 0}\sum \limits_{\substack{(y', \rho) \in I_{\varepsilon, \infty}(\omega) \\ \varepsilon y' \in O'}} \varepsilon^{N - 1} \tilde{J}_\infty\left(\rho, \frac{\varepsilon r(\omega, y')}{a_\varepsilon \rho}, \frac{h_\varepsilon}{\rho}\right) \le 2 \gamma_\infty(\omega) |U' \cap O'| \quad \mathbb{P}\text{-a.s.}
    \]
    Together with \eqref{eq:liminf_bound}, this proves the lemma for $h_0 = \infty$. The case $h_0 = 0$ is very similar so its proof is omitted. 
    
    Now assume $h_0 \in (0, \infty)$ and fix $l \ge 2$. If $\sigma$ is fixed and $\varepsilon$ is sufficiently small, then, similar to \eqref{eq:good_upper_bound} and \eqref{eq:bad_upper_bound}, we have the following inequalities:
    \begin{gather*}
        \tilde{J}_{h_0}\left(\rho, \frac{\varepsilon r(\omega, y')}{a_\varepsilon \rho}, \frac{h_\varepsilon}{\rho}\right) \le \max\left\{\frac{h_\varepsilon}{h_0}, \frac{h_0}{h_\varepsilon}\right\} \tilde{J}_{h_0}\left(\rho, l, \frac{h_0}{\rho}\right) \quad \text{for all } (y', \rho) \in I_{\varepsilon, h_0}(\omega) \setminus M_\sigma(\omega), \\
        \tilde{J}_{h_0}\left(\rho, \frac{\varepsilon r(\omega, y')}{a_\varepsilon \rho}, \frac{h_\varepsilon}{\rho}\right) \le \tilde{J}_{h_0}\left(\rho, 2 , \frac{h_0}{\rho}\right) \quad \text{for all } (y', \rho) \in I_{\varepsilon, h_0}(\omega).
    \end{gather*}
    Then, by splitting the sum in \eqref{eq:really_ugly_sum_limit_nonzero} as in \eqref{eq:splitting_sum_argument} and using Proposition \ref{pr:weighted_sums_go_to_zero}, we obtain    
    \begin{equation} \label{eq:same_for_finite_height}
        \limsup_{\varepsilon \to 0}\sum \limits_{\substack{(y', \rho) \in I_{\varepsilon, h_0}(\omega) \\ \varepsilon y' \in O'}} \varepsilon^{N - 1} \tilde{J}_{h_0}\left(\rho, \frac{\varepsilon r(\omega, y')}{a_\varepsilon \rho}, \frac{h_\varepsilon}{\rho}\right)  \\
        \le \limsup_{\varepsilon \to 0} \sum \limits_{\substack{(y', \rho) \in M(\omega) \\ \varepsilon y' \in U' \cap O'}} \varepsilon^{N - 1} q_{h_0}(\rho, l) J_{h_0}(\rho) \quad \mathbb{P}\text{-a.s.},
    \end{equation}
    where
    \[
    q_{h_0}(\rho, l) := \frac{\tilde{J}_{h_0}\left(\rho, l, \frac{h_0}{\rho}\right)}{J_{h_0}(\rho)} = \frac{\cpct_{\frac{h_0}{\rho}}(T', B'(0, l))}{\cpct_{\frac{h_0}{\rho}}(T', \mathbb{R}^{N - 1})}.
    \]
    We know by Lemma \ref{lm:comparison_old_and_new} that $q_{h_0}$ is bounded in $(0, \infty) \times [2, \infty)$. Thus, we can apply Theorem \ref{thm:ergodic_theorem_second_version} to the right-hand side of \eqref{eq:same_for_finite_height} to get
    \[
    \lim_{\varepsilon \to 0} \sum \limits_{\substack{(y', \rho) \in M(\omega) \\ \varepsilon y' \in U' \cap O'}} \varepsilon^{N - 1} q_{h_0}(\rho, l) J_{h_0}(\rho) = |U' \cap O'| \int_0^\infty q_{h_0}(\rho, l) J_{h_0}(\rho) \, d\xi(\omega, \rho) \quad \mathbb{P}\text{-a.s.}
    \]
    Since $q_{h_0}(\rho, l) \to 1$ as $l \to \infty$ for any fixed $\rho$, the dominated convergence theorem yields
    \[
    \lim_{\varepsilon \to 0} \frac{1}{2} \sum \limits_{\substack{(y', \rho) \in M(\omega) \\ \varepsilon y' \in U' \cap O'}} \varepsilon^{N - 1} q_{h_0}(\rho, l) J_{h_0}(\rho) = \frac{1}{2} |U' \cap O'| \int_0^\infty J_{h_0}(\rho) \, d\xi(\omega, \rho) = \gamma_{h_0}(\omega) |U' \cap O'| \quad \mathbb{P}\text{-a.s.}
    \]
    Finally, \eqref{eq:liminf_bound}, \eqref{eq:same_for_finite_height} and the limit above combine to give the result for $h_0 \in (0, \infty)$.
\end{proof}

\begin{corollary} \label{co:gradient_computation}
    Let $O' \subset \mathbb{R}^{N - 1}$ be a Borel measurable set such that $|\partial O'| = 0$. Then
    \[
    \lim_{\varepsilon \to 0} \frac{1}{\delta_\varepsilon} \sum \limits_{\substack{(y', \rho) \in I_{\varepsilon, h_0}(\omega) \\ \varepsilon y' \in O'}} \int_{B'(\varepsilon y', \varepsilon r(\omega, y')) \times (0, \delta_\varepsilon)} |\nabla w_\varepsilon(\omega, x)|^2 \, dx = \gamma_{h_0}(\omega) |U' \cap O'| \quad \mathbb{P}\text{-a.s.}
    \]
\end{corollary}

\begin{proof}
    The claim follows from \eqref{eq:gradient_divided_by_delta} and Lemma \ref{lm:really_ugly_sum_limit}.
\end{proof}

We proceed to compute the limit of the scaled $L^2$-norm of the gradients, localizing by continuous test functions.

\begin{proposition} \label{pr:local_gradient_limit}
    For $\mathbb{P}$-a.e. $\omega \in \Omega$ and all $\varphi \in C^0(\overline{U'})$, we have
    \begin{equation} \label{eq:local_gradient_limit}
        \lim_{\varepsilon \to 0} \frac{1}{\delta_\varepsilon} \int_{U_\varepsilon^+} |\nabla w_\varepsilon(\omega, x)|^2 \varphi(x') \, dx = \gamma_{h_0}(\omega) \int_{U'} \varphi(x') \, dx' \quad \mathbb{P}\text{-a.s.}
    \end{equation}
    In particular,
    \begin{equation} \label{eq:bounded_gradients}
        \lim_{\varepsilon \to 0} \frac{1}{\delta_\varepsilon} \int_{U_\varepsilon^+} |\nabla w_\varepsilon(\omega, x)|^2 \, dx = \gamma_{h_0}(\omega) |U'| \quad \mathbb{P}\text{-a.s.}
    \end{equation}
\end{proposition}

\begin{proof}
    We start by showing that
    \begin{equation} \label{eq:discrete_approximation}
        \lim_{\varepsilon \to 0} \frac{1}{\delta_\varepsilon} \int_{Q' \times (0, \delta_\varepsilon)} |\nabla w_\varepsilon(\omega, x)|^2 \, dx = \gamma_{h_0}(\omega) |U' \cap Q'| \quad \mathbb{P}\text{-a.s.}
    \end{equation}
    for any closed cube $Q' \subset \mathbb{R}^{N - 1}$. Let $\varepsilon > 0$ and $\omega \in \Omega$. Then
    \begin{multline} \label{eq:splitting_over_the_balls}
        \int_{Q' \times (0, \delta_\varepsilon)} |\nabla w_\varepsilon(\omega, x)|^2 \, dx = \int_{(Q' \cap S'_{\varepsilon, h_0}(\omega)) \times (0, \delta_\varepsilon)} |\nabla w_\varepsilon(\omega, x)|^2 \, dx \\
        + \sum_{(y', \rho) \in I_{\varepsilon, h_0}(\omega)} \int_{(Q' \cap B'(\varepsilon y', \varepsilon r(\omega, y'))) \times (0, \delta_\varepsilon)} |\nabla w_\varepsilon(\omega, x)|^2 \, dx.
    \end{multline}
    It follows from \eqref{eq:vanishing_gradient_1}, \eqref{eq:vanishing_gradient_2} and Theorem \ref{thm:vanishing_capacity_and_measure} that
    \[
    \lim_{\varepsilon \to 0} \frac{1}{\delta_\varepsilon} \int_{(Q' \cap S'_{\varepsilon, h_0}(\omega)) \times (0, \delta_\varepsilon)} |\nabla w_\varepsilon(\omega, x)|^2 \, dx = 0 \quad \mathbb{P}\text{-a.s.}
    \]
    To deal with the second term on the right-hand side of \eqref{eq:splitting_over_the_balls}, we choose closed cubes $Q'_a$, $Q'_b$, different from $Q'$, such that they have the same center as $Q'$ and $Q'_a \subset Q' \subset Q'_b$. If $Q' \cap B'(\varepsilon y', \varepsilon r(\omega, y')) \neq \emptyset$ for some $(y', \rho) \in I_{\varepsilon, h_0}(\omega)$, then $\varepsilon y' \in Q'_b$ for small enough $\varepsilon$, since $r(\omega, y') \le 1$. Similarly, if $\varepsilon y' \in Q'_a$, then $B'(\varepsilon y', \varepsilon r(\omega, y')) \subset Q'$ for small enough $\varepsilon$. Thus,
    \begin{multline*}
        \sum \limits_{\substack{(y', \rho) \in I_{\varepsilon, h_0}(\omega) \\ \varepsilon y' \in Q'_a}} \int_{B'(\varepsilon y', \varepsilon r(\omega, y')) \times (0, \delta_\varepsilon)} |\nabla w_\varepsilon(\omega, x)|^2 \, dx \\
        \le \sum_{(y', \rho) \in I_{\varepsilon, h_0}(\omega)} \int_{(Q' \cap B'(\varepsilon y', \varepsilon r(\omega, y')))\times (0, \delta_\varepsilon)} |\nabla w_\varepsilon(\omega, x)|^2 \, dx \\
        \le \sum \limits_{\substack{(y', \rho) \in I_{\varepsilon, h_0}(\omega) \\ \varepsilon y' \in Q'_b}} \int_{B'(\varepsilon y', \varepsilon r(\omega, y')) \times (0, \delta_\varepsilon)} |\nabla w_\varepsilon(\omega, x)|^2 \, dx.
    \end{multline*}
    for small enough $\varepsilon$. Corollary \ref{co:gradient_computation} immediately gives
    \begin{multline*}
        \gamma_{h_0}(\omega) |U' \cap Q'_a| \le \liminf_{\varepsilon \to 0} \frac{1}{\delta_\varepsilon} \sum_{(y', \rho) \in I_{\varepsilon, h_0}(\omega)} \int_{B'(\varepsilon y', \varepsilon r(\omega, y')) \times (0, \delta_\varepsilon)} |\nabla w_\varepsilon(\omega, x)|^2 \, dx\\
        \le \limsup_{\varepsilon \to 0} \frac{1}{\delta_\varepsilon} \sum_{(y', \rho) \in I_{\varepsilon, h_0}(\omega)} \int_{B'(\varepsilon y', \varepsilon r(\omega, y')) \times (0, \delta_\varepsilon)} |\nabla w_\varepsilon(\omega, x)|^2 \, dx \le \gamma_{h_0}(\omega) |U' \cap Q'_b| \quad \mathbb{P}\text{-a.s.}
    \end{multline*}
    Since $|U' \cap Q'_a|$ and $|U' \cap Q'_b|$ can be made arbitrarily close to $|U' \cap Q'|$, we conclude the proof of \eqref{eq:discrete_approximation}.

    Let $\varphi \in C^0(\overline{U'})$. In order to prove \eqref{eq:local_gradient_limit} we approximate $U'$ by cubes. Let $(Q'_1)^k, \dots, (Q'_{m_k})^k$ and $(P'_1)^k, \dots, (P'_{n_k})^k$ denote those cubes of side length $1/(2k)$ centered at points of $\frac{1}{k}\mathbb{Z}^N$ such that $(Q'_j)^k \subset U'$ and $(P'_j)^k \cap \partial U' \neq \emptyset$ for all $j$. We claim that
    \begin{equation} \label{eq:divide_and_conquer}
        \lim_{k \to \infty} \lim_{\varepsilon \to 0} \frac{1}{\delta_\varepsilon}\left|\int_{U' \times (0, \delta_\varepsilon)} |\nabla w_\varepsilon(\omega, x)|^2 \varphi(x') \, dx - \sum_{j = 1}^{m_k} \int_{(Q'_j)^k \times (0, \delta_\varepsilon)} |\nabla w_\varepsilon(\omega, x)|^2 \varphi(x') \, dx \right| = 0 \quad \mathbb{P}\text{-a.s.}
    \end{equation}
    The difference is clearly bounded by
    \[
    \|\varphi\|_{C^0({\overline{U'}})} \sum_{j = 1}^{n_k} \int_{(P'_j)^k} |\nabla w_\varepsilon(\omega, x)|^2 \, dx.
    \]
    Then, \eqref{eq:discrete_approximation} immediately yields
    \begin{align*}
        &\limsup_{k \to \infty} \limsup_{\varepsilon \to 0} \frac{1}{\delta_\varepsilon} \left|\int_{U' \times (0, \delta_\varepsilon)} |\nabla w_\varepsilon(\omega, x)|^2 \varphi(x') \, dx - \sum_{j = 1}^{m_k} \int_{(Q'_j)^k \times (0, \delta_\varepsilon)} |\nabla w_\varepsilon(\omega, x)|^2 \varphi(x') \, dx \right| \\
        &\le \limsup_{k \to \infty} \gamma_{h_0}(\omega) \|\varphi\|_{C^0(\overline{U'})}\left|U' \cap \bigcup_{j = 1}^{n_k} (P'_j)^k \right| = 0 \quad \mathbb{P}\text{-a.s.,}
    \end{align*}
    since $|\partial U'| = 0$. Next, we evaluate the limit of the second term in \eqref{eq:divide_and_conquer}. We shall do this by approximating $\varphi$ by its average in the cubes. Due to the uniform continuity of $\varphi$ and \eqref{eq:discrete_approximation}, we know that
    \[
    \lim_{k \to \infty} \lim_{\varepsilon \to 0} \frac{1}{\delta_\varepsilon} \sum_{j = 1}^{m_k} \int_{(Q'_j)^k \times (0, \delta_\varepsilon)} |\nabla w_\varepsilon(\omega, x)|^2 |\varphi(x') - (\varphi)_{(Q'_j)^k}| \, dx = 0 \quad \mathbb{P}\text{-a.s.},
    \]
    where $(\varphi)_{(Q'_j)^k}$ denotes the average of $\varphi$ in $(Q'_j)^k$. Another application of \eqref{eq:discrete_approximation} also yields
    \[
    \lim_{\varepsilon \to 0} \frac{1}{\delta_\varepsilon} \sum_{j = 1}^{m_k} \int_{(Q'_j)^k \times (0, \delta_\varepsilon)} |\nabla w_\varepsilon(\omega, x)|^2 (\varphi)_{(Q'_j)^k} \, dx = \gamma_{h_0}(\omega) \sum_{j = 1}^{m_k} |U' \cap (Q'_j)^k| (\varphi)_{(Q'_j)^k} \quad \mathbb{P}\text{-a.s.}
    \]
    Finally, it is clear that
    \[
    \lim_{k \to \infty} \gamma_{h_0}(\omega) \sum_{j = 1}^{m_k} |U' \cap (Q'_j)^k| (\varphi)_{(Q'_j)^k} = \gamma_{h_0}(\omega) \int_{U'} \varphi \, dx'.
    \]
    This concludes the proof.
\end{proof}

Recall from Definition \ref{def:omega_admissible} the notion of an $\omega$-admissible sequence.

\begin{theorem} \label{thm:omega_admissible}
    For $\mathbb{P}$-a.e. $\omega \in \Omega$, $(w_\varepsilon(\omega, \cdot))$ is $\omega$-admissible and $\hat{w}_\varepsilon(\omega, \cdot) \rightharpoonup 1$ in $H^1(U^+)$.
\end{theorem}

\begin{proof}
    The function $w_\varepsilon(\omega, \cdot)$ vanishes weakly on $T'_\varepsilon(\omega) \times \{0\}$ by construction. Hence, $(w_\varepsilon(\omega, \cdot))$ is $\omega$-admissible $\mathbb{P}$-a.s. by \eqref{eq:bounded_gradients} in Proposition \ref{pr:local_gradient_limit}. We need to prove that $\hat{w}_\varepsilon(\omega, \cdot) \to 1$ in $L^2(U^+)$. We have
    \begin{multline*}
        \int_{U^+} |\hat{w}_\varepsilon(\omega, x) - 1|^2 \, dx \le \int_{S'_{\varepsilon, h_0}(\omega) \times (0, 1)} |\hat{w}_\varepsilon(\omega, x) - 1|^2 \, dx \\
        + \sum_{(y', \rho) \in I_{\varepsilon, h_0}(\omega)} \int_{B'(\varepsilon y', \varepsilon r(\omega, y')) \times (0, 1)} |\hat{w}_\varepsilon(\omega, x) - 1|^2 \, dx
    \end{multline*}
    Since $|\hat{w}_\varepsilon(\omega, \cdot)| \le 1$ a.e., Theorem \ref{thm:vanishing_capacity_and_measure} gives
    \[
    \limsup_{\varepsilon \to 0} \int_{S'_{\varepsilon, h_0}(\omega) \times (0, 1)} |\hat{w}_\varepsilon(\omega, x) - 1|^2 \, dx \lesssim \limsup_{\varepsilon \to 0} |S'_{\varepsilon, h_0}(\omega)| \quad \mathbb{P}\text{-a.s.}
    \]
    On the other hand, by Poincaré's inequality we have
    \begin{align*}
        \int_{B'(\varepsilon y', \varepsilon r(\omega, y')) \times (0, 1)} |\hat{w}_\varepsilon(\omega, x) - 1|^2 \, dx &= \frac{1}{\delta_\varepsilon} \int_0^{\delta_\varepsilon} \int_{B'(\varepsilon y', \varepsilon r(\omega, y'))} |w_\varepsilon(\omega, x) - 1|^2 \, dx' \, dx_N \\
        &\lesssim \frac{\varepsilon^2}{\delta_\varepsilon} \int_0^{\delta_\varepsilon} \int_{B'(\varepsilon y', \varepsilon r(\omega, y'))} |\nabla' w_\varepsilon(\omega, x)|^2 \, dx' \, dx_N.
    \end{align*}
    Hence, Corollary \ref{co:gradient_computation} yields
    \begin{multline*}
    \limsup_{\varepsilon \to 0} \sum_{(y', \rho) \in I_{\varepsilon, h_0}(\omega)} \int_{B'(\varepsilon y', \varepsilon r(\omega, y')) \times (0, 1)} |\hat{w}_\varepsilon(\omega, x) - 1|^2 \, dx \\
    \lesssim \limsup_{\varepsilon \to 0} \sum_{(y', \rho) \in I_{\varepsilon, h_0}(\omega)} \frac{\varepsilon^2}{\delta_\varepsilon} \int_{B'(\varepsilon y', \varepsilon r(\omega, y')) \times (0, \delta_\varepsilon)} |\nabla w_\varepsilon(\omega, x)|^2 \, dx = 0 \quad \mathbb{P}\text{-a.s.}
    \end{multline*}
\end{proof}

\subsection{A variational property of the oscillating test functions} \label{ssec:variational_property}

In this section, we establish a special case of Theorem \ref{thm:oscillating_test_functions}. As we will see later, the special case is an important step in the proof of the theorem itself. From now on, we assume that if $N = 3$, then $\ln(1/\delta_\varepsilon) \ll 1/\varepsilon^2$. The main result is stated in the next proposition.

\begin{proposition} \label{pr:special_case}
    For $\mathbb{P}$-a.e. $\omega \in \Omega$, the sequence $(w_\varepsilon(\omega, \cdot))$ satisfies the following property: Given an $\omega$-admissible sequence $(v_\varepsilon)$, if there exists a subsequence $(\varepsilon_n)$ tending to $0$ such that $\hat{v}_{\varepsilon_n} \rightharpoonup 0$ in $H^1(U^+)$ as $n \to \infty$, we have
    \begin{equation} 
        \lim_{n \to \infty} \frac{1}{\delta_{\varepsilon_n}} \int_{U_{\varepsilon_n}^+} \nabla w_{\varepsilon_n}(\omega, x) \nabla v_{\varepsilon_n}(x) \, dx = 0.
    \end{equation}
\end{proposition}

The proof of Proposition \ref{pr:special_case} is based on a minimization property of $(w_\varepsilon(\omega, \cdot))$. We shall compare the oscillating test functions with a subclass of $\omega$-admissible sequences.

\begin{definition}
    For $\omega \in \Omega$, let $\mathcal{O}(\omega)$ be the set of all $\omega$-admissible sequences $(v_\varepsilon)$ such that $\hat{v}_\varepsilon \rightharpoonup 1$ in $H^1(U^+)$ as $\varepsilon \to 0$.    
\end{definition}

\begin{proposition} \label{pr:minimization_proposition}
    For $\mathbb{P}$-a.e. $\omega \in \Omega$ and for all $(v_\varepsilon) \in \mathcal{O}(\omega)$, we have
    \begin{equation} \label{eq:liminf_inequality}
        \liminf_{\varepsilon \to 0} \frac{1}{\delta_\varepsilon} \int_{U_\varepsilon^+} |\nabla w_\varepsilon(\omega, x)|^2 \, dx \le \liminf_{\varepsilon \to 0} \frac{1}{\delta_\varepsilon} \int_{U_\varepsilon^+} |\nabla v_\varepsilon(x)|^2 \, dx.
    \end{equation}
\end{proposition}

\begin{proof}[Proof of Proposition \ref{pr:special_case}]
    Fix an $\omega \in \Omega$ for which $(w_\varepsilon(\omega, \cdot)) \in \mathcal{O}(\omega)$ and Proposition \ref{pr:minimization_proposition} holds. As $\omega$ is fixed, we omit it from the notation. Assume $(v_\varepsilon)$ is an $\omega$-admissible sequence such that $\hat{v}_\varepsilon \rightharpoonup 0$ in $H^1(U^+)$. Clearly, $(w_\varepsilon + tv_\varepsilon)$ lies in $\mathcal{O}(\omega)$ for all $t > 0$. Hence, Proposition \ref{pr:minimization_proposition} implies
        \begin{equation} \label{eq:nonnegative_liminf}
            \liminf_{\varepsilon \to 0} \frac{1}{\delta_\varepsilon} \int_{U_\varepsilon^+} (|\nabla (w_\varepsilon + tv_\varepsilon)|^2 - |\nabla w_\varepsilon|^2) \, dx \ge 0.
        \end{equation}
    Since
    \[
    \nabla w_\varepsilon \nabla v_\varepsilon = \frac{1}{2t}\left(|\nabla (w_\varepsilon + tv_\varepsilon)|^2 - |\nabla w_\varepsilon|^2 - t^2|\nabla v_\varepsilon|^2\right),
    \]
    by \eqref{eq:nonnegative_liminf} we get
    \[
    \liminf_{\varepsilon \to 0} \frac{1}{\delta_\varepsilon} \int_{U_\varepsilon^+} \nabla w_\varepsilon \nabla v_\varepsilon \, dx \ge \liminf_{\varepsilon \to 0} \frac{1}{\delta_\varepsilon} \frac{1}{2t} \int_{U_\varepsilon^+} (|\nabla (w_\varepsilon + tv_\varepsilon)|^2 - |\nabla w_\varepsilon|^2) \, dx - \frac{Ct}{2} \ge -\frac{Ct}{2},
    \]
    where $\sup_{\varepsilon} \|\nabla v_\varepsilon\|_{L^2(U_\varepsilon^+; \mathbb{R}^N)}^2 \le C \delta_\varepsilon$. Letting $t \to 0$ yields
    \[
    \liminf_{\varepsilon \to 0} \frac{1}{\delta_\varepsilon} \int_{U_\varepsilon^+} \nabla w_\varepsilon \nabla v_\varepsilon \, dx \ge 0.
    \]
    Replacing $v_\varepsilon$ by $-v_\varepsilon$ provides also the reverse inequality.

    Given a sequence $(v_\varepsilon)$ for which $(\hat{v}_\varepsilon)$ converges weakly to $0$ in $H^1(U^+)$ only along a subsequence $(\varepsilon_n)$, we define
    \[
    \vartheta_\varepsilon :=
        \begin{cases}
            v_{\varepsilon_n} \quad &\text{if } \varepsilon = \varepsilon_n \text{ for some } n \in \mathbb{N}, \\
            0 \quad &\text{otherwise}.
        \end{cases}
    \]
    Then $(\vartheta_\varepsilon)$ is $\omega$-admissible and $\vartheta_\varepsilon \rightharpoonup 0$ in $H^1(U^+)$ as $\varepsilon \to 0$. Hence, we can apply the argument given above to conclude that
    \[
    \lim_{n \to \infty} \frac{1}{\delta_{\varepsilon_n}} \int_{U_{\varepsilon_n}^+} \nabla w_{\varepsilon_n} \nabla v_{\varepsilon_n} \, dx = 0.
    \]    
\end{proof}

Before we give the proof of Proposition \ref{pr:minimization_proposition}, we argue heuristically to see why the result holds. Fix $\omega \in \Omega$ and $(v_\varepsilon) \in \mathcal{O}(\omega)$. Assume $(y', \rho) \in I_{\varepsilon, h_0}(\omega)$ and $B(\varepsilon y', \varepsilon r(\omega, y')) \subset U'$. Since $\delta_\varepsilon^{-1} \int_{U_\varepsilon^+} |1 - v_\varepsilon|^2 \, dx \to 0$, we expect that $1 - v_\varepsilon \approx 0$ in $B'(\varepsilon y', \varepsilon r(\omega, y')) \times (0, \delta_\varepsilon)$. At the same time, $1 - v_\varepsilon = 1$ weakly on the contact region $\varepsilon y + a_\varepsilon \rho(T' \times \{0\})$. Among all $H^1$ functions that are close to $0$ in $B'(\varepsilon y', \varepsilon r(\omega, y')) \times (0, \delta_\varepsilon)$ and equal $1$ on the contact region, we expect the capacitary potential of $\varepsilon y + a_\varepsilon \rho(T' \times \{0\})$ in $B'(\varepsilon y', \varepsilon r(\omega, y')) \times (-\delta_\varepsilon, \delta_\varepsilon)$ to have the smallest Dirichlet energy. As $w_\varepsilon(\omega, \cdot)$ is defined to be the capacitary potential near isolated contact regions, we expect
\[
\int_{B'(\varepsilon y', \varepsilon r(\omega, y')) \times (0, \delta_\varepsilon)} |\nabla v_\varepsilon|^2 \, dx \gtrsim \int_{B'(\varepsilon y', \varepsilon r(\omega, y')) \times (0, \delta_\varepsilon)} |\nabla w_\varepsilon(\omega, x)|^2 \, dx \approx \frac{1}{2} \varepsilon^{N - 1} \delta_\varepsilon J_{h_0}(\rho),
\]
where the last approximation follows from \eqref{eq:gradient_divided_by_delta}.

To make this idea rigorous, we single out the points where the desired inequality holds and prove that, asymptotically, they constitute the majority of the isolated points. Our main tools will be Propositions \ref{pr:subcapacitary_growth_1} and \ref{pr:subcapacitary_growth_2}.

\begin{definition}
    Let $\sigma, \theta \in (0, 1)$. Let $O' \csubset U'$ be open.  For a fixed sequence $(v_\varepsilon) \in \mathcal{O}(\omega)$, we denote by $\mathcal{G}_{\varepsilon, h_0}(\sigma, \theta, O')$ the set of points $(y', \rho) \in I_{\varepsilon, h_0}(\omega) \setminus M_\sigma(\omega)$ such that $B'(\varepsilon y', \varepsilon r(\omega, y')) \subset O'$ and
    \begin{equation} \label{eq:good_point}
        \int_{B'(\varepsilon y', \varepsilon r(\omega, y')) \times (0, \delta_\varepsilon)} |\nabla v_\varepsilon|^2 \, dx \ge \frac{\theta}{2} \varepsilon^{N - 1} \delta_\varepsilon J_{h_0}(\rho).
    \end{equation}
    In contrast, $\mathcal{N}_{\varepsilon, h_0}(\sigma, \theta, O')$ denotes the set of points $(y', \rho) \in I_{\varepsilon, h_0}(\omega) \setminus M_\sigma(\omega)$ such that $B'(\varepsilon y', \varepsilon r(\omega, y')) \subset O'$ and \eqref{eq:good_point} fails.
\end{definition}

\begin{lemma} \label{lm:bad_point_big_average}
    Let $\omega \in \Omega$ and $(v_\varepsilon) \in \mathcal{O}(\omega)$ be fixed. Let $\sigma, \theta \in (0, 1)$ and let $O' \csubset U'$ be open. There exist positive constants $c$ and $\varepsilon_0$ such that if $\varepsilon < \varepsilon_0$, then we have
    \begin{equation} \label{eq:bad_point_big_average}
        \fint_{B'(\varepsilon y', \varepsilon r(\omega, y')) \times (0, 1)} 1 - \hat{v}_\varepsilon \, dx > c
    \end{equation}
    for all $(y', \rho) \in \mathcal{N}_{\varepsilon, h_0}(\sigma, \theta, O')$.
\end{lemma}

\begin{proof}
    We extend $v_\varepsilon$ to $U' \times (-\delta_\varepsilon, \delta_\varepsilon)$ by reflection, that is, we set $v_\varepsilon(x', x_N) := v_\varepsilon(x', -x_N)$ for $(x', x_N) \in U' \times (-\delta_\varepsilon, 0)$. Then $v_\varepsilon \in H^1(U' \times (-\delta_\varepsilon, \delta_\varepsilon))$. If $(y', \rho) \in \mathcal{N}_{\varepsilon, h_0}(\sigma, \theta, O')$, we have
    \begin{equation} \label{eq:blow_up}
        \int_{B'(\varepsilon y', \varepsilon r(\omega, y')) \times (-\delta_\varepsilon, \delta_\varepsilon)} |\nabla v_\varepsilon|^2  \, dx < \theta \varepsilon^{N - 1} \delta_\varepsilon J_{h_0}(\rho).
    \end{equation}
    An easy computation shows that
    \begin{equation} \label{eq:capacity_identities}
    \varepsilon^{N - 1} \delta_\varepsilon J_{h_0}(\rho) = 
        \begin{cases}
            a_\varepsilon^{N - 2} \rho^{N - 2} \cpct(T' \times \{0\}, \mathbb{R}^N) \quad &\text{if } h_0 = \infty, \\
            a_\varepsilon^{N - 2} \rho^{N - 2} \cpct_{\frac{h_0}{\rho}}(T', \mathbb{R}^{N - 1}) \quad &\text{if } h_0 \in (0, \infty), \\
            2 a_\varepsilon^{N - 3} \delta_\varepsilon \rho^{N - 3} \cpct_0(T', \mathbb{R}^{N - 1}) \quad &\text{if } h_0 = 0.
        \end{cases}
    \end{equation}
    Let us now assume $N = 3$. In this case $h_0 = \infty$. Then \eqref{eq:blow_up} and \eqref{eq:capacity_identities} give
    \[
    \frac{1}{a_\varepsilon^{N - 2} \rho^{N - 2}} \int_{B'(\varepsilon y', \varepsilon r(\omega, y')) \times (-\delta_\varepsilon, \delta_\varepsilon)} |\nabla v_\varepsilon|^2 \, dx < \theta \cpct(T' \times \{0\}, \mathbb{R}^N).
    \]
    By applying a change of variables in Proposition \ref{pr:subcapacitary_growth_1}, we obtain positive constants $L$, $H$, $\tau$ and $c$, depending on $\theta$, such that if
    \begin{equation} \label{eq:desired_implication}
         \frac{\varepsilon r(\omega, y')}{a_\varepsilon \rho} \ge L, \quad \frac{h_\varepsilon}{\rho} \ge H, \quad \max \left\{\ln\left(\frac{\varepsilon r(\omega, y')}{a_\varepsilon \rho}\right) \frac{a_\varepsilon \rho}{\delta_\varepsilon}, \frac{h_\varepsilon}{\rho}\frac{a_\varepsilon^2 \rho^2}{\varepsilon^2 r(\omega, y')^2}\right\} \le \tau,
    \end{equation}
    then
    \begin{equation}
        c < \fint_{B'(\varepsilon y', \varepsilon r(\omega, y')) \times (-\delta_\varepsilon, \delta_\varepsilon)} 1 - v_\varepsilon \, dx = \fint_{B'(\varepsilon y', \varepsilon r(\omega, y')) \times (0, 1)} 1 - \hat{v}_\varepsilon \, dx.
    \end{equation}
    We show that there exists $\varepsilon_0$ such that the conditions in \eqref{eq:desired_implication} are satisfied for all $(y', \rho) \in \mathcal{N}_{\varepsilon, h_0}(\sigma, \theta, O')$ if $\varepsilon < \varepsilon_0$. We write the first two conditions as
    \begin{equation} \label{eq:equivalent_bounds}
        \frac{\varepsilon}{a_\varepsilon} \ge \frac{L \rho}{r(\omega, y')}, \quad h_\varepsilon \ge H \rho.
    \end{equation}
    Since $(y', \rho) \in M(\omega) \setminus M_\sigma(\omega)$, we know that $r(\omega, y') \ge \sigma$ and $\rho \le 1/\sigma$. Hence, the following conditions imply \eqref{eq:equivalent_bounds} and are independent of $(y', \rho)$:
    \begin{equation} \label{eq:equivalent_bounds_independent}
        \frac{\varepsilon}{a_\varepsilon} \ge \frac{L}{\sigma^2}, \quad h_\varepsilon \ge \frac{H}{\sigma}.
    \end{equation}
    Next, we rewrite the third condition in \eqref{eq:desired_implication} as
    \begin{equation} \label{eq:equivalent_bounds_2}
        \left(\ln\left(\frac{r(\omega, y')}{\varepsilon \rho}\right) + \ln\left(\frac{1}{\delta_\varepsilon}\right) \right) \varepsilon^2 \rho \le \tau, \quad \frac{\rho}{r(\omega, y')} \delta_\varepsilon^2 \le \tau.
    \end{equation}
    We know that $x \ln(1/x)$ is increasing for $x \le 1/e$. Hence, using $r(\omega, y') \le 1$ and $\rho \le 1/\sigma$, we get
    \[
    \varepsilon^2 \rho \ln\left(\frac{r(\omega, y')}{\varepsilon \rho}\right) + \varepsilon^2 \rho \ln\left(\frac{1}{\delta_\varepsilon}\right) \le \frac{1}{\sigma} \left(\varepsilon^2 \ln\left(\frac{\sigma}{\varepsilon}\right) + \varepsilon^2 \ln\left(\frac{1}{\delta_\varepsilon}\right)\right)
    \]
    for $\varepsilon \le \sigma/e$. Thus
    \begin{equation} \label{eq:equivalent_bounds_independent_2}
        \tau \ge \frac{1}{\sigma} \left(\varepsilon^2 \ln\left(\frac{\sigma}{\varepsilon}\right) + \varepsilon^2 \ln\left(\frac{1}{\delta_\varepsilon}\right)\right)
    \end{equation}
    implies the first inequality in \eqref{eq:equivalent_bounds_2} for sufficiently small $\varepsilon$ and is independent of $(y', \rho)$. On the other hand,
    \begin{equation} \label{eq:equivalent_bounds_independent_3}
        \tau \ge \frac{\delta_\varepsilon^2}{\sigma^3}
    \end{equation}
    implies the second inequality in \eqref{eq:equivalent_bounds_2} and is independent of $(y', \rho)$. To conclude the proof for $N = 3$, we need to show that \eqref{eq:equivalent_bounds_independent}, \eqref{eq:equivalent_bounds_independent_2} and \eqref{eq:equivalent_bounds_independent_3} are all satisfied for sufficiently small $\varepsilon$. However, this follows easily from the following limits:
    \[
    \lim_{\varepsilon \to 0} \varepsilon^2 \ln\left(\frac{1}{\delta_\varepsilon}\right) = \lim_{\varepsilon \to 0} \delta_\varepsilon = 0, \quad \lim_{\varepsilon \to 0} \frac{\varepsilon}{a_\varepsilon} = \lim_{\varepsilon \to 0} h_\varepsilon = \infty.
    \]
    The first limit holds by assumption, and the last limit holds because $h_0 = \infty$.

    Let $N > 3$. We prove that for any $\tilde{\theta} \in (\theta, 1)$ and for sufficiently small $\varepsilon$, we have
    \begin{equation} \label{eq:change_of_capacity_function}
        \frac{1}{a_\varepsilon^{N - 2} \rho^{N - 2}} \int_{B'(\varepsilon y', \varepsilon r(\omega, y')) \times (-\delta_\varepsilon, \delta_\varepsilon)} |\nabla v_\varepsilon|^2 \, dx \le \tilde{\theta} \cpct_{\frac{h_\varepsilon}{\rho}}(T', \mathbb{R}^{N - 1})
    \end{equation}
    for all $(y', \rho) \in \mathcal{N}_{\varepsilon, h_0}(\sigma, \theta, O')$. Given the claim, we can use Proposition \ref{pr:subcapacitary_growth_2} to obtain $L$, $c$ and $\tau$, depending on $\tilde{\theta}$, such that if
    \begin{equation} \label{eq:desired_implication_2}
        \frac{\varepsilon r(\omega, y')}{a_\varepsilon \rho} \ge L, \quad \frac{h_\varepsilon}{\rho} \left(\frac{a_\varepsilon \rho}{\varepsilon r(\omega, y')}\right)^{N - 1} \le \tau,
    \end{equation}
    then
    \[
    c < \fint_{B'(\varepsilon y', \varepsilon r(\omega, y')) \times (-\delta_\varepsilon, \delta_\varepsilon)} 1 - v_\varepsilon \, dx = \fint_{B'(\varepsilon y', \varepsilon r(\omega, y')) \times (0, 1)} 1 - \hat{v}_\varepsilon \, dx.
    \]
    We can establish, as in the case $N = 3$, that there exists $\varepsilon_0$ such that if $\varepsilon < \varepsilon_0$, then the conditions of \eqref{eq:desired_implication_2} are satisfied for all $(y', \rho) \in \mathcal{N}_{\varepsilon, h_0}(\sigma, \theta, O')$. We do not carry out the computations here and turn to the proof of \eqref{eq:change_of_capacity_function} instead.

    Let $\tilde{\theta} \in (\theta, 1)$  be arbitrary. Assume $h_0 = \infty$. By Proposition \ref{pr:limit_at_infinity}, we have that $\theta \cpct(T' \times \{0\}, \mathbb{R}^N) \le \tilde{\theta} \cpct_h(T', \mathbb{R}^{N - 1})$ for sufficiently large $h$. Hence, \eqref{eq:blow_up} and \eqref{eq:capacity_identities} give
    \begin{equation*}
        \frac{1}{a_\varepsilon^{N - 2} \rho^{N - 2}} \int_{B'(\varepsilon y', \varepsilon r(\omega, y')) \times (-\delta_\varepsilon, \delta_\varepsilon)} |\nabla v_\varepsilon|^2 \, dx \le \theta \cpct(T' \times \{0\}, \mathbb{R}^N) \le \tilde{\theta} \cpct_{\frac{h_\varepsilon}{\rho}}(T', \mathbb{R}^{N - 1}).
    \end{equation*}
    for small enough $\varepsilon$. Next, assume $h_0 = 0$. Since $\cpct_h(T', \mathbb{R}^{N - 1})/(2h) \to \cpct_0(T', \mathbb{R}^{N - 1})$ as $h \to 0$ by definition, we know that $\theta \, 2h \cpct_0(T', \mathbb{R}^{N - 1}) \le \tilde{\theta} \cpct_h(T', \mathbb{R}^{N - 1})$ for sufficiently small $h$. Therefore, by \eqref{eq:blow_up} and \eqref{eq:capacity_identities},
    \begin{equation*}
        \frac{1}{a_\varepsilon^{N - 2} \rho^{N - 2}} \int_{B'(\varepsilon y', \varepsilon r(\omega, y')) \times (-\delta_\varepsilon, \delta_\varepsilon)} |\nabla v_\varepsilon|^2 \, dx \le \theta \frac{2h_\varepsilon}{\rho} \cpct_0(T', \mathbb{R}^{N - 1}) \le \tilde{\theta} \cpct_{\frac{h_\varepsilon}{\rho}}(T', \mathbb{R}^{N - 1})
    \end{equation*}
    for small enough $\varepsilon$. Finally, let $h_0 \in (0, \infty)$. Pick $\varepsilon$ small enough such that $\theta \max \{h_\varepsilon/h_0, h_0/h_\varepsilon\} \le \tilde{\theta}$. Then, Proposition \ref{pr:continuity_in_h}, together with \eqref{eq:blow_up} and \eqref{eq:capacity_identities}, implies
    \begin{equation*}
        \frac{1}{a_\varepsilon^{N - 2} \rho^{N - 2}} \int_{B'(\varepsilon y', \varepsilon r(\omega, y')) \times (-\delta_\varepsilon, \delta_\varepsilon)} |\nabla v_\varepsilon|^2 \, dx < \theta \cpct_{\frac{h_0}{\rho}}(T', \mathbb{R}^{N - 1}) \le \tilde{\theta} \cpct_\frac{h_\varepsilon}{\rho}(T', \mathbb{R}^{N - 1}).
    \end{equation*}
    This concludes the proof of the inequality \eqref{eq:change_of_capacity_function}.
\end{proof}

\begin{corollary} \label{co:counting_bad_points}
    Let $\omega \in \Omega$ and $(v_\varepsilon) \in \mathcal{O}(\omega)$ be fixed. Let $\sigma, \theta \in (0, 1)$ and let $O' \csubset U'$ be open. Then
    \[
    \lim_{\varepsilon \to 0} \varepsilon^{N - 1} \card\left(\mathcal{N}_{\varepsilon, h_0}(\sigma, \theta, O')\right) = 0.
    \]
\end{corollary}

\begin{proof}
    By Lemma \ref{lm:bad_point_big_average}, there exist positive constants $c$ and $\varepsilon_0$, such that
    \[
    \int_{B'(\varepsilon y', \varepsilon r(\omega, y')) \times (0, 1)} |1 - \hat{v}_\varepsilon|^2 \, dx \ge c \varepsilon^{N - 1} r(\omega, y')^{N - 1} \gtrsim \varepsilon^{N - 1}
    \]
    for all $\varepsilon < \varepsilon_0$ and $(y', \rho) \in \mathcal{N}_{\varepsilon, h_0}(\sigma, \theta, O')$. The first inequality follows from \eqref{eq:bad_point_big_average} via Hölder's inequality, and the second one follows from the fact that $r(\omega, y') \ge \sigma$ for all $(y', \rho) \in \mathcal{N}_{\varepsilon, h_0}(\sigma, \theta, O')$. Since $(v_\varepsilon) \in \mathcal{O}(\omega)$, we immediately obtain
    \begin{multline*}
        \limsup_{\varepsilon \to 0} \varepsilon^{N - 1} \card\left(\mathcal{N}_{\varepsilon, h_0}(\sigma, \theta, O')\right) \\
        \lesssim \limsup_{\varepsilon \to 0} \sum_{(y', \rho) \in \mathcal{N}_{\varepsilon, h_0}(\sigma, \theta, O')} \int_{B'(\varepsilon y', \varepsilon r(\omega, y')) \times (0, 1)} |1 - \hat{v}_\varepsilon|^2 \, dx \le \limsup_{\varepsilon \to 0} \int_{U^+} |1 - \hat{v}_\varepsilon|^2 \, dx = 0.
    \end{multline*}
\end{proof}

We now give a proof of Proposition \ref{pr:minimization_proposition} using a thinning argument.

\begin{proof}[Proof of Proposition \ref{pr:minimization_proposition}]
    Fix $\omega \in \Omega$ and $(v_\varepsilon) \in \mathcal{O}(\omega)$. Let $\sigma, \theta \in (0, 1)$ and let $O' \csubset U'$ be open with $|\partial O'| = 0$. Then
    \begin{multline} \label{eq:good_points_lower_bound}
        \frac{1}{\delta_\varepsilon} \int_{U_\varepsilon^+} |\nabla v_\varepsilon|^2 \, dx \\
        \ge \sum_{(y', \rho) \in \mathcal{G}_{\varepsilon, h_0}(\sigma, \theta, O')} \frac{1}{\delta_\varepsilon} \int_{B'(\varepsilon y', \varepsilon r(\omega, y')) \times (0, \delta_\varepsilon)} |\nabla v_\varepsilon|^2 \, dx \ge \sum_{(y', \rho) \in \mathcal{G}_{\varepsilon, h_0}(\sigma, \theta, O')} \frac{\theta}{2} \varepsilon^{N - 1} J_{h_0}(\rho).
    \end{multline}
    We have
    \begin{equation} \label{eq:add_and_subtract_bad_points}
        \sum_{(y', \rho) \in \mathcal{G}_{\varepsilon, h_0}(\sigma, \theta, O')} \frac{\theta}{2} \varepsilon^{N - 1} J_{h_0}(\rho) = \sum \limits_{\substack{(y', \rho) \in I_{\varepsilon, h_0}(\omega) \setminus M_\sigma(\omega) \\ \varepsilon y' \in O'}} \frac{\theta}{2} \varepsilon^{N - 1} J_{h_0}(\rho) - \sum_{(y', \rho) \in \mathcal{N}_{\varepsilon, h_0}(\sigma, \theta, O')} \frac{\theta}{2} \varepsilon^{N - 1} J_{h_0}(\rho)
    \end{equation}
    for sufficiently small $\varepsilon$. Since $\rho \le 1/\sigma$ and $J_{h_0}$ is locally bounded, Corollary \ref{co:counting_bad_points} implies
    \begin{equation} \label{eq:bad_points_estimate}
        \sum_{(y', \rho) \in \mathcal{N}_{\varepsilon, h_0}(\sigma, \theta, O')} \frac{\theta}{2} \varepsilon^{N - 1} J_{h_0}(\rho) \lesssim \varepsilon^{N - 1} \card\left(\mathcal{N}_{\varepsilon, h_0}(\sigma, \theta, O')\right) \to 0.
    \end{equation}
    On the other hand, it follows from Lemma \ref{lm:thinned_limit} that
    \begin{equation} \label{eq:two_sums_are_close}
        \limsup_{\sigma \to 0} \limsup_{\varepsilon \to 0} \Bigg|\sum \limits_{\substack{(y', \rho) \in I_{\varepsilon, h_0}(\omega) \setminus M_\sigma(\omega) \\ \varepsilon y' \in O'}} \varepsilon^{N - 1} J_{h_0}(\rho) - \sum \limits_{\substack{(y', \rho) \in I_{\varepsilon, h_0}(\omega) \\ \varepsilon y' \in O'}} \varepsilon^{N - 1} J_{h_0}(\rho)\Bigg| = 0 \quad \mathbb{P}\text{-a.s.}
    \end{equation}
    Hence, \eqref{eq:fewer_arguments_limit} gives
    \begin{equation} \label{eq:good_points_limit}
        \lim_{\sigma \to 0} \liminf_{\varepsilon \to 0} \sum \limits_{\substack{(y', \rho) \in I_{\varepsilon, h_0}(\omega) \setminus M_\sigma(\omega) \\ \varepsilon y' \in O'}} \frac{\theta}{2} \varepsilon^{N - 1} J_{h_0}(\rho) = \theta \gamma_{h_0}(\omega) |O'| \quad \mathbb{P}\text{-a.s.}
    \end{equation}
    If $\omega$ satisfies \eqref{eq:two_sums_are_close} and \eqref{eq:good_points_limit}, then by \eqref{eq:good_points_lower_bound}, \eqref{eq:add_and_subtract_bad_points} and \eqref{eq:bad_points_estimate} we obtain
    \[
    \liminf_{\varepsilon \to 0} \frac{1}{\delta_\varepsilon} \int_{U_\varepsilon^+} |\nabla v_\varepsilon|^2 \, dx \ge \theta \gamma_{h_0}(\omega) |O'|. 
    \]
    Since $\theta$ and $O'$ are arbitrary, we have just proved that for $\mathbb{P}$-a.e. $\omega \in \Omega$ and for all $(v_\varepsilon) \in \mathcal{O}(\omega)$, we have
    \[
    \liminf_{\varepsilon \to 0} \frac{1}{\delta_\varepsilon} \int_{U_\varepsilon^+} |\nabla v_\varepsilon|^2 \, dx \ge \gamma_{h_0}(\omega) |U'| = \lim_{\varepsilon \to 0} \frac{1}{\delta_\varepsilon} \int_{U_\varepsilon^+} |\nabla w_\varepsilon(\omega, x)|^2 \, dx,
    \]
    where the last equality follows from Proposition \ref{pr:local_gradient_limit}.
\end{proof}

\section{Proof of Theorem \ref{thm:oscillating_test_functions}} \label{sec:proof_of_main_theorem}

In this section, we conclude the proof of Theorem \ref{thm:oscillating_test_functions} by showing that the oscillating test functions $(w_\varepsilon(\omega, \cdot))$ satisfy \eqref{eq:good_laplacians} for $\mathbb{P}$-a.e. $\omega \in \Omega$. Our proof is based on an idea introduced by Casado-Díaz in \cite[Theorem 2.1]{CD}. We need a generalization of Proposition \ref{pr:local_gradient_limit} to bounded test functions in $H^1$.

\begin{lemma} \label{lm:extended_local_gradient}
    For $\mathbb{P}$-a.e. $\omega \in \Omega$ and all $v \in H^1(U') \cap L^\infty(U')$, we have
    \[
    \lim_{\varepsilon \to 0} \frac{1}{\delta_\varepsilon} \int_{U_\varepsilon^+} |\nabla w_\varepsilon(\omega, x)|^2 v(x') \, dx = \gamma_{h_0}(\omega) \int_{U'} v(x') \, dx'.
    \]
\end{lemma}

\begin{proof}
    Fix some $\omega \in \Omega$ that satisfies Propositions \ref{pr:local_gradient_limit}, \ref{pr:special_case} and Theorem \ref{thm:omega_admissible}. For simplicity, we omit $\omega$ from the notation. Let $v \in H^1(U') \cap L^\infty(U')$, and approximate it by a sequence $(\varphi_n) \subset C^0(\overline{U'})$ in the $H^1$-norm. Furthermore, we can take $(\varphi_n)$ uniformly bounded in $C^0(\overline{U'})$. Then Proposition \ref{pr:local_gradient_limit} implies
    \[
    \lim_{n \to \infty} \lim_{\varepsilon \to 0} \frac{1}{\delta_\varepsilon} \int_{U_\varepsilon^+} |\nabla w_\varepsilon|^2 \varphi_n \, dx = \lim_{n \to \infty} \gamma_{h_0} \int_{U'} \varphi_n \, dx' = \gamma_{h_0} \int_{U'} v \, dx'.
    \]
    Hence, to conclude, we need to show that
    \[
    \lim_{n \to \infty} \lim_{\varepsilon \to 0} \frac{1}{\delta_\varepsilon} \int_{U_\varepsilon^+} |\nabla w_\varepsilon|^2 |v - \varphi_n| \, dx = 0.
    \]
    This is implied by the assertion that
    \begin{equation} \label{eq:does_it_go_to_zero?}
        \lim_{n \to \infty} \frac{1}{\delta_{\varepsilon_n}} \int_{U_{\varepsilon_n}^+} |\nabla w_{\varepsilon_n}|^2 |v - \varphi_n| \, dx = 0
    \end{equation}
    for all subsequences $(\varepsilon_n)$ converging to $0$. Given a subsequence $(\varepsilon_n)$, we define $v_n := w_{\varepsilon_n}|v - \varphi_n|$. Then $(v_n)$ is $\omega$-admissible and $\hat{v}_n \rightharpoonup 0$ in $H^1(U^+)$. Thus, Proposition \ref{pr:special_case} implies
    \[
    \lim_{n \to \infty} \frac{1}{\delta_{\varepsilon_n}}\int_{U_{\varepsilon_n}^+} \nabla w_{\varepsilon_n} \nabla v_n \, dx = 0.
    \]
    Computing the gradient of $v_n$ gives
    \begin{equation} \label{eq:sum_of_two_terms}
        0 = \lim_{n \to \infty} \frac{1}{\delta_{\varepsilon_n}} \int_{U_{\varepsilon_n}^+} |\nabla w_{\varepsilon_n}|^2 |v - \varphi_n| + (\nabla' w_{\varepsilon_n} \nabla' |v - \varphi_n|) w_{\varepsilon_n} \, dx.
    \end{equation}
    Since $(w_\varepsilon)$ is $\omega$-admissible and $|w_\varepsilon| \le 1$ a.e. in $U_{\varepsilon_n}^+$, we get
    \[
     \lim_{n \to \infty} \frac{1}{\delta_{\varepsilon_n}} \int_{U_{\varepsilon_n}^+} (\nabla' w_{\varepsilon_n} \nabla' |v - \varphi_n|) w_{\varepsilon_n} \, dx = 0.
    \]
    Therefore, equation \eqref{eq:sum_of_two_terms} yields \eqref{eq:does_it_go_to_zero?}.
\end{proof}

\begin{proof}[Proof of Theorem \ref{thm:oscillating_test_functions}]
     We fix some $\omega \in \Omega$ for which Theorem \ref{thm:omega_admissible}, Proposition \ref{pr:special_case} and Lemma \ref{lm:extended_local_gradient} hold. Having fixed $\omega$, we omit it from the notation. Let $(v_\varepsilon)$ be an $\omega$-admissible sequence such that $\hat{v}_{\varepsilon_n} \rightharpoonup \hat{v}$ in $H^1(U^+)$ as $n \to \infty$ for some $\hat{v}$. Since $(v_\varepsilon)$ is $\omega$-admissible, $\partial_N \hat{v}_\varepsilon \to 0$ in $L^2(U^+)$ so that $\partial_N \hat{v} = 0$. Hence, there exists $v \in H^1(U')$ such that $\hat{v}(\cdot, x_N) = v$ for a.e. $x_N \in (0, 1)$. We prove that
     \[
     \lim_{n \to \infty} \frac{1}{\delta_{\varepsilon_n}} \int_{U_{\varepsilon_n}^+} \nabla w_{\varepsilon_n} \nabla v_{\varepsilon_n} \, dx = \gamma_{h_0}(\omega) \int_{U'} v \, dx',
     \]
    which is equivalent to \eqref{eq:good_laplacians}.
     
    In the first part of the proof, we assume additionally that $v \in L^\infty(U')$. Set $\vartheta_\varepsilon := v_\varepsilon - w_\varepsilon v$. We can check easily that $\vartheta_\varepsilon$ is $\omega$-admissible and $\hat{\vartheta}_{\varepsilon_n} \rightharpoonup 0$ in $H^1(U^+)$ as $n \to \infty$. Therefore, Proposition \ref{pr:special_case} implies
    \[
    \lim_{n \to \infty} \frac{1}{\delta_{\varepsilon_n}} \int_{U_{\varepsilon_n}^+} \nabla w_{\varepsilon_n} \nabla \vartheta_{\varepsilon_n} \, dx = 0.
    \]
    Computing the gradient of $\vartheta_{\varepsilon_n}$ explicitly, we obtain
    \begin{equation} \label{eq:explicitly_written_out}
        \frac{1}{\delta_{\varepsilon_n}} \int_{U_{\varepsilon_n}^+} \nabla w_{\varepsilon_n} \nabla v_{\varepsilon_n} - |\nabla w_{\varepsilon_n}|^2 v  -  (\nabla' w_{\varepsilon_n} \nabla' v) w_{\varepsilon_n} \, dx \to 0.
    \end{equation}
    It follows from Theorem \ref{thm:omega_admissible} that
    \[
    \lim_{n \to \infty} \frac{1}{\delta_{\varepsilon_n}} \int_{U_{\varepsilon_n}^+} (\nabla' w_{\varepsilon_n} \nabla' v) w_{\varepsilon_n} \, dx = 0.
    \]
    Hence, with the help of Lemma \ref{lm:extended_local_gradient}, we deduce from \eqref{eq:explicitly_written_out} that
    \[
    \lim_{n \to \infty} \frac{1}{\delta_{\varepsilon_n}} \int_{U_{\varepsilon_n}^+} \nabla w_{\varepsilon_n} \nabla v_{\varepsilon_n} \, dx = \gamma_{h_0} \int_{U'} v \, dx'.
    \]

    Let us now consider the general case. We define $v_\varepsilon^k := \min \{\max \{v_\varepsilon, -k\}, k\}$, $v^k := \min \{\max \{v, -k\}, k\}$ and set $\hat{v}^k(x', x_N) := v^k(x')$ for $(x', x_N) \in U^+$. Then, $\hat{v}_{\varepsilon_n}^k \rightharpoonup \hat{v}^k$ in $H^1(U^+)$. Therefore, by the first step we have
    \begin{equation*}
        \lim_{k \to \infty} \lim_{n \to \infty} \frac{1}{\delta_{\varepsilon_n}} \int_{U_{\varepsilon_n}^+} \nabla w_{\varepsilon_n} \nabla v_{\varepsilon_n}^k \, dx = \lim_{k \to \infty} \gamma_{h_0} \int_{U'} v^k \, dx' = \gamma_{h_0} \int_{U'} v \, dx'.
    \end{equation*}
    Thus, all that remains to be proved is
    \[
    \lim_{k \to \infty} \lim_{n \to \infty} \frac{1}{\delta_{\varepsilon_n}} \int_{U_{\varepsilon_n}^+} \nabla w_{\varepsilon_n} (\nabla v_{\varepsilon_n} - \nabla v_{\varepsilon_n}^k) \, dx = 0.
    \]
    To this end, we let $(\varepsilon_{n_k})$ be a further subsequence. Then, it is easy to verify that $(v_{\varepsilon_{n_k}} - v_{\varepsilon_{n_k}}^k)$ is $\omega$-admissible and $\hat{v}_{\varepsilon_{n_k}} - \hat{v}_{\varepsilon_{n_k}}^k \rightharpoonup 0$ in $H^1(U^+)$ as $k \to \infty$. Hence, as a result of Proposition \ref{pr:special_case}, we obtain
    \[
    \lim_{k \to \infty} \frac{1}{\delta_{\varepsilon_{n_k}}} \int_{U_{\varepsilon_{n_k}}^+} \nabla w_{\varepsilon_{n_k}} \nabla (v_{\varepsilon_{n_k}} - v_{\varepsilon_{n_k}}^k) \, dx = 0.
    \]
    The claim now follows from the arbitrariness of the sequence $(\varepsilon_{n_k})$.
\end{proof}

\section{Proofs of Propositions \ref{pr:subcapacitary_growth_1} and \ref{pr:subcapacitary_growth_2}} \label{sec:not_appendix}

An important role in the proof of Theorem \ref{thm:oscillating_test_functions} was played by Propositions \ref{pr:subcapacitary_growth_1} and \ref{pr:subcapacitary_growth_2}. The remainder of the paper is devoted to their proofs.

\subsection{Auxiliary lemmas}

Here, we collect some preliminary results. We begin with a lemma that provides sufficient conditions for the local averages of a sequence of zero-mean functions to converge to $0$.

\begin{lemma} \label{lm:vanishing_averages}
    Let $(r_n)$ and $(\rho_n)$ be sequences of positive numbers with
    \[
    \lim_{n \to \infty} r_n = \lim_{n \to \infty} \rho_n = \infty, \quad \rho_n \le r_n \text{ for all } n.
    \]
    Assume $v_n \in C^\infty(\overline{B(0, r_n)})$ has zero mean. Assume one of the following conditions hold:
    \begin{enumerate}
        \item[(1)] $N = 1$ and $\lim_{n \to \infty} r_n \int_{-r_n}^{r_n} |v_n'|^2 \, dx = 0$,
        \item[(2)] $N = 2$ and $\lim_{n \to \infty} \ln(r_n) \int_{B(0, r_n)} |\nabla v_n|^2 \, dx = 0$,
        \item[(3)] $N > 2$ and $\lim_{n \to \infty} \int_{B(0, r_n)} |\nabla v_n|^2 \, dx = 0$.
    \end{enumerate}
    Then $v_n \to 0$ in $H_\mathrm{loc}^1(\mathbb{R}^N)$ and
    \begin{equation} \label{eq:subaverages_go_to_zero}
        \lim_{n \to \infty} \fint_{B(0, \rho_n)} v_n \, dx = 0.
    \end{equation}
\end{lemma}

\begin{remark} \label{rmk:explanation_of_the_factors}
    The conditions given in the hypotheses are sharp. For $N = 1$, define
    \[
    u_n(x) := \frac{1}{r_n}(x + r_n) \chi_{(-r_n, 0)}(x) - \frac{1}{r_n}(x - r_n) \chi_{[0, r_n)}(x), \quad a_n := \fint_{-r_n}^{r_n} u_n \, dx.
    \]
    If $v_n := u_n - a_n$, then $v_n$ has zero mean and $r_n \int_{-r_n}^{r_n} |v_n'|^2 \, dx = 1$. On the other hand, $v_n \to 1/2$ in $H_\mathrm{loc}^1(\mathbb{R})$ and $\fint_{-r_n/2}^{r_n/2} v_n \, dx \to 1/4$. For $N > 1$, counterexamples are provided by taking $u_n$ as the solution of
    \begin{equation*}
        \begin{cases}
            \begin{alignedat}{1}
                - \Delta u_n &= 0 \quad \text{in } B(0, r_n) \setminus B(0, 1), \\
                u_n &= 1 \quad \text{on } \partial B(0, 1), \\
                u_n &= 0 \quad \text{on } \partial B(0, r_n).
            \end{alignedat}
        \end{cases}
    \end{equation*}
    and setting $v_n := u_n - a_n$, where $a_n := \fint_{B(0, r_n)} u_n \, dx$. The solutions $u_n$ can be computed explicitly, and the reader can verify that $\ln(r_n) \int_{B(0, r_n)} |\nabla v_n|^2 \, dx \gtrsim 1$ for $N = 2$ and $\int_{B(0, r_n)} |\nabla v_n|^2 \, dx \gtrsim 1$ for $N > 2$. On the other hand, we have $v_n \to 1$ in $H_\mathrm{loc}^1(\mathbb{R}^N)$. We leave it to the reader to check that \eqref{eq:subaverages_go_to_zero} also fails for some sequence $(\rho_n)$.
\end{remark}

\begin{proof}
    \underline{Step 1:} Let $N = 1$. For $x < y$ we have
    \[
    |v_n(y) - v_n(x)| \le \int_x^y |v_n'| \, dt \le |y - x|^\frac{1}{2} \left(\int_x^y |v_n'|^2 \, dt\right)^\frac{1}{2} \lesssim \left(r_n \int_{-r_n}^{r_n} |v_n'|^2 \, dt\right)^\frac{1}{2}.
    \]
    Thus, $\sup_{x, y \in (-r_n, r_n)} |v_n(y) - v_n(x)| \to 0$ as $n \to \infty$. Since $v_n$ is continuous, there exists $x_n \in (-r_n, r_n)$ such that $v_n(x_n) = \fint_{-r_n}^{r_n} v_n \, dx = 0$. Hence, $\|v_n\|_{L^\infty(-r_n, r_n)} \to 0$. The claim now follows easily.

    \underline{Step 2:} Let $N = 2$. Define the circular means
    \[
    \bar{v}_n(r) := \frac{1}{2\pi r} \int_{\partial B(0, r)} v_n \, d\mathcal{H}^1, \quad r \in (0, r_n].
    \]
    Then
    \[
    |\bar{v}_n'(r)|^2 = \left|\frac{1}{2\pi r} \int_{\partial B(0, r)} \nabla v_n \cdot \nu \, d\mathcal{H}^1\right|^2 \le \frac{1}{2\pi r} \int_{\partial B(0, r)} |\nabla v_n|^2 \, d\mathcal{H}^1, 
    \]
    where $\nu$ is the outer unit normal of $B(0, r)$. Therefore,
    \begin{equation} \label{eq:circular_means_energy}
        \int_0^{r_n} r |\bar{v}_n'(r)|^2 \, dr \le \frac{1}{2\pi} \int_0^{r_n} \int_{\partial B(0, r)} |\nabla v_n(x)|^2 \, d\mathcal{H}^1(x) \, dr = \frac{1}{2\pi} \int_{B(0, r_n)} |\nabla v_n(x)|^2 \, dx.
    \end{equation}
    Consequently, for $1 < r < r_n$ we have
    \[
    |\bar{v}_n(r_n) - \bar{v}_n(r)| \le \int_r^{r_n} |\bar{v}_n'(s)| \, ds \le \left(\int_r^{r_n} \frac{1}{s} \, ds\right)^\frac{1}{2} \left(\int_r^{r_n} s |\bar{v}_n'(s)|^2 \, ds\right)^\frac{1}{2} \lesssim \left(\ln(r_n) \int_{B(0, r_n)} |\nabla v_n(x)|^2 \, dx\right)^\frac{1}{2}.
    \]
    From the hypothesis, we have $\|\bar{v}_n(r_n) - \bar{v}_n\|_{L^\infty(1, r_n)} \to 0$ as $n \to \infty$. Consequently, if $\bar{v}_n(r_n) \to 0$, then $\|\bar{v}_n\|_{L^\infty(1, r_n)} \to 0$ as well. Since $\bar{v}_n$ has zero mean, we have
    \[
    \bar{v}_n(r_n) = \bar{v}_n(r_n) - \fint_{B(0, r_n)} v_n \, dx = \frac{2}{r_n^2} \int_0^{r_n} r(\bar{v}_n(r_n) - \bar{v}_n(r)) \, dr = \frac{2}{r_n^2} \int_0^{r_n} \int_r^{r_n} r \bar{v}_n'(s) \, ds \, dr.
    \]
    By Fubini's theorem and Hölder's inequality we get
    \[
    |\bar{v}_n(r_n)| = \left|\frac{2}{r_n^2} \int_0^{r_n} \bar{v}_n'(s) \int_0^s r \, dr \, ds \right| = \left|\frac{1}{r_n^2} \int_0^{r_n} s^2 \bar{v}_n'(s) \, ds \right| \le \frac{1}{r_n^2} \left(\int_0^{r_n} s^3 \, ds \right)^\frac{1}{2} \left(\int_0^{r_n} s |\bar{v}_n'(s)|^2 \, ds\right)^\frac{1}{2}
    \]
    Hence, \eqref{eq:circular_means_energy} implies
    \[
    |\bar{v}_n(r_n)|^2 \lesssim \int_{B(0, r_n)} |\nabla v_n|^2 \, dx \to 0 \quad \text{as } n \to \infty.
    \]
    Next, we show that $v_n \to 0$ in $H_\mathrm{loc}^1(\mathbb{R}^2)$. Fix $r > 0$. Set
    \[
    a_n := \fint_{B(0, r) \setminus B(0, 1)} v_n \, dx, \quad b_n := \fint_{B(0, r)} v_n \, dx.
    \]
    By Hölder's and Poincaré's inequalities we have
    \begin{equation} \label{eq:difference_of_averages}
        |a_n - b_n|^2 \le \fint_{B(0, r) \setminus B(0, 1)} |v_n - b_n|^2 \, dx \lesssim \fint_{B(0, r)} |\nabla v_n|^2 \, dx \to 0.
    \end{equation}
    Since $\|\bar{v}_n\|_{L^\infty(1, r_n)} \to 0$, we see that
    \begin{equation} \label{eq:annular_average}
        |a_n| = \left|\frac{2}{r^2 - 1} \int_1^{r} s \bar{v}_n(s) \, ds\right| \le \|\bar{v}_n\|_{L^\infty(1, r_n)} \to 0 \quad \text{as } n \to \infty.
    \end{equation}
    Thus \eqref{eq:difference_of_averages} implies $b_n \to 0$. Applying Poincaré's inequality again, we obtain
    \[
    \|v_n\|_{L^2(B(0, r))} \le \|v_n - b_n\|_{L^2(B(0, r))} + |b_n||B(0, r)|^\frac{1}{2} \lesssim \|\nabla v_n\|_{L^2(B(0, r); \mathbb{R}^N)} + |b_n| \to 0.
    \]
    Finally, we prove that $\fint_{B(0, \rho_n)} v_n \, dx \to 0$ for any sequence $(\rho_n)$ with $\rho_n \to \infty$ and $\rho_n \le r_n$ for all $n$.
    We decompose the integral as follows:
    \[
    \fint_{B(0, \rho_n)} v_n \, dx = \frac{2}{\rho_n^2} \int_1^{\rho_n} s \bar{v}_n(s) \, ds + \frac{1}{|B(0, \rho_n)|} \int_{B(0, 1)} v_n \, dx
    \]
    The first term on the right-hand side goes to $0$ by a computation similar to \eqref{eq:annular_average}. The second term goes to $0$ due to the local $H^1$-convergence. This concludes the proof of the claim in $N = 2$.

    \underline{Step 3:} Let $N > 2$. By the Poincaré-Sobolev inequality and the hypothesis,
    \[
    \limsup_{n \to \infty} \left(\int_{B(0, r_n)} |v_n|^\frac{2N}{N - 2} \, dx\right)^\frac{N - 2}{2N} \le C \limsup_{n \to \infty} \left(\int_{B(0, r_n)} |\nabla v_n|^2 \, dx\right)^\frac{1}{2} = 0.
    \]
    Then Hölder's inequality immediately yields the local $H^1$-convergence. On the other hand
    \[
    \limsup_{n \to \infty} \left|\fint_{B(0, \rho_n)} v_n \, dx\right| \le \limsup_{n \to \infty} \left(\fint_{B(0, \rho_n)} |v_n|^\frac{2N}{N - 2} \, dx\right)^\frac{N - 2}{2N} = 0.
    \]
    Therefore, the claim is proved.
\end{proof}

\begin{lemma} \label{lm:infinite_height}
    Let $N > 2$ and let $K \subset \mathbb{R}^N$ be a bounded set. Assume $(l_n)$ and $(h_n)$ are sequences of real numbers such that
    \[
    \lim_{n \to \infty} l_n = \lim_{n \to \infty} h_n = \infty, \quad \lim_{n \to \infty} \frac{h_n}{l_n^{N - 1}} = 0.
    \]
    If $N = 3$, we also assume $\ln(l_n)/h_n \to 0$. Suppose $w_n \in C^\infty(\overline{C(l_n, h_n)})$ has zero mean and $K \subset \topint(\{w_n = \alpha_n\})$. If $\alpha_n \to \alpha \neq 0$ and there is a constant $\Lambda > 0$ such that
    \begin{equation} \label{eq:finite_limsup}
        \limsup_{n \to \infty} \int_{C(l_n, h_n)} |\nabla w_n|^2 \, dx \le \Lambda,
    \end{equation}
    then
    \[
    \cpct(K, \mathbb{R}^N) \le \frac{\Lambda}{\alpha^2}.
    \]
\end{lemma}

\begin{proof}
    Without loss of generality, we assume that $h_n/l_n$ converges to a limit in $[0, \infty]$. Otherwise, we can prove the statement for a subsequence. Set $C_n := C(l_n, h_n)$.

    \underline{Step 1:} Let $\lim_{n \to \infty} h_n/l_n$ be finite and nonzero. Scaling the Poincaré-Sobolev inequality in the unit cylinder $C(1, 1)$ appropriately yields
    \begin{equation} \label{eq:poincaré_sobolev}
        \left(\int_{C_n} |w_n|^{\frac{2N}{N - 2}} \, dx\right)^\frac{N - 2}{2N} \le C \left(\int_{C_n} \left(\frac{l_n}{h_n}\right)^\frac{2}{N}|\nabla' w_n|^2 +  \left(\frac{h_n}{l_n}\right)^\frac{2(N - 1)}{N} |\partial_N w_n|^2 \, dx\right)^\frac{1}{2},
    \end{equation}
    where $C$ depends only on $N$. Since $h_n/l_n$ converges to a finite, nonzero limit, the bound \eqref{eq:finite_limsup} implies
    \begin{equation} \label{eq:bound_in_sobolev_exponent}
        \limsup_{n \to \infty} \left(\int_{C_n} |w_n|^{\frac{2N}{N - 2}} \, dx\right)^\frac{N - 2}{2N} < \infty.
    \end{equation}
    Passing to a subsequence (not relabeled), we know that there exist functions $w \in L^\frac{2N}{N - 2}(\mathbb{R}^N)$, $g \in L^2(\mathbb{R}^N; \mathbb{R}^N)$ such that $w_n \rightharpoonup w$ in $L^\frac{2N}{N - 2}(\mathbb{R}^N)$ and $\nabla w_n \rightharpoonup g$ in $L^2(\mathbb{R}^N; \mathbb{R}^N)$. If $\varphi \in C_c^\infty(\mathbb{R}^N)$, then
    \[
    \int_{\mathbb{R}^N} w \nabla \varphi \, dx = \lim_{n \to \infty} \int_{\mathbb{R}^N} w_n \nabla \varphi \, dx = -\lim_{n \to \infty} \int_{\mathbb{R}^N} \varphi \nabla w_n \, dx = \int_{\mathbb{R}^N} \varphi g \, dx.
    \]
    Hence, $\nabla w = g$ and we deduce that $w \in \mathcal{C}(\mathbb{R}^N)$. Next, we show that $w = \alpha$ weakly on $K$. Let $W$ be a bounded open neighborhood of $K$. Using \eqref{eq:finite_limsup}, \eqref{eq:bound_in_sobolev_exponent} and Hölder's inequality, we see that $(w_n)$ is uniformly bounded in $H^1(W)$. Consequently, the Rellich-Kondrachov compactness theorem gives $\lim_{n \to \infty} \|w_n - w\|_{L^2(W)} = 0$. On the other hand, $K \cap \overline{\supp(w_n - \alpha_n)} = \emptyset$ holds by assumption. Hence, by Lemma \ref{lm:vanishing_weakly}, $w = \alpha$ weakly on $K$. As a result, $w/\alpha$ is an admissible test function in \eqref{eq:extended_capacity_1}. Therefore,
    \[
    \alpha^2 \cpct(K, \mathbb{R}^N) \le \int_{\mathbb{R}^N} |\nabla w|^2 \, dx.
    \]
    It also follows from the weak lower semi-continuity of the $L^2$-norm and \eqref{eq:finite_limsup} that
    \[
    \int_{\mathbb{R}^N} |\nabla w|^2 \, dx \le \liminf_{n \to \infty} \int_{C_n} |\nabla w_n|^2 \, dx \le \Lambda.
    \]
    We conclude that $\cpct(K, \mathbb{R}^N) \le \Lambda/\alpha^2$.

    \underline{Step 2:} Next, we assume that $\lim_{n \to \infty} h_n/l_n = \infty$. In this case, the proof in Step 1 breaks down since the right-hand side of \eqref{eq:poincaré_sobolev} does not necessarily remain bounded. However, we will show that
    \[
    \fint_{C(l_n, l_n)} w_n \, dx \to 0 \quad \text{as } n \to \infty.
    \]
    It is then enough to apply Step 1 to $w_n - \fint_{C(l_n, l_n)} w_n \, dx$ in the cylinders $C(l_n, l_n)$.
    
    As $h_n \gg l_n$, the domain $C_n$ behaves like a $1$-dimensional set at large scales. Accordingly, we reduce the problem to the study of the growth of functions of a single variable. For $x_N \in (-h_n, h_n)$, define the $(N - 1)$-dimensional cross-sectional average
    \[
    \bar{w}_n(x_N) := \fint_{B'(0, l_n)} w_n(x', x_N) \, dx'.
    \]
    Then,
    \[
    \fint_{-h_n}^{h_n} \bar{w}_n \, dx_N = \fint_{C_n} w_n \, dx = 0.
    \]
    By Jensen's inequality and \eqref{eq:finite_limsup}, we also obtain
    \[
    \int_{-h_n}^{h_n} |\bar{w}_n'|^2 \, dx_N = \int_{-h_n}^{h_n} \left| \fint_{B'(0, l_n)} \partial_N w_n \, dx' \right|^2 \, dx_N \le \int_{-h_n}^{h_n} \fint_{B'(0, l_n)} |\partial_N w_n|^2 \, dx' \lesssim \frac{1}{l_n^{N - 1}}
    \]
    Thus, our hypothesis $h_n \ll l_n^{N - 1}$ implies that
    \[
    h_n \int_{-h_n}^{h_n} |\bar{w}_n'|^2 \, dx_N \lesssim \frac{h_n}{l_n^{N - 1}} \to 0 \quad \text{as } n \to \infty.
    \]
    With the help of Lemma \ref{lm:vanishing_averages}, we conclude that
    \[
    \fint_{C(l_n, l_n)} w_n \, dx = \fint_{-l_n}^{l_n} \bar{w}_n \, dx \to 0 \quad \text{as } n \to \infty.
    \]

    \underline{Step 3:} To finish the proof, we finally assume $\lim_{n \to \infty} h_n/l_n = 0$. Our approach will be similar to Step 2. We will prove that
    \[
    \fint_{C(h_n, h_n)} w_n \, dx \to 0 \quad \text{as } n \to \infty.
    \]
    The rest of the proof is analogous to Step 1. This time, we observe that $C_n$ behaves like an $(N - 1)$-dimensional set at large scales. Hence, we consider the functions defined by
    \[
    \bar{w}_n(x') := \fint_{-h_n}^{h_n} w_n(x', x_N) \, dx_N, \quad x' \in B'(0, l_n)
    \]
    By \eqref{eq:finite_limsup}
    \[
    \int_{B'(0, l_n)} |\nabla' \bar{w}_n|^2 \, dx' = \int_{B'(0, l_n)} \left|\fint_{-h_n}^{h_n} \nabla' w_n \, dx_N \right|^2 \, dx' \le \int_{B'(0, l_n)} \fint_{-h_n}^{h_n} |\nabla' w_n|^2 \, dx \lesssim \frac{1}{h_n}.
    \]
    Hence $\|\nabla' \bar{w}_n\|_{L^2(B'(0, l_n); \mathbb{R}^N)} \to 0$. If $N = 3$, then by assumption
    \[
    \limsup_{n \to \infty} \ln(l_n) \int_{B'(0, l_n)} |\nabla \bar{w}_n|^2 \, dx' \le \lim_{n \to \infty} \frac{\ln(l_n)}{h_n} = 0.
    \]
    Since $\bar{w}_n$ has zero mean, we deduce from Lemma \ref{lm:vanishing_averages} that
    \[
    \fint_{C(h_n, h_n)} w_n \, dx = \fint_{B'(0, h_n)} \bar{w}_n \, dx' \to 0 \quad \text{as } n \to \infty.
    \]
    for all $N$.
\end{proof}

\begin{lemma} \label{lm:finite_height}
    Let $N > 3$ and let $K' \subset \mathbb{R}^{N - 1}$ be a bounded set. Assume $(l_n)$ and $(h_n)$ are sequences of real numbers such that
    \[
    \lim_{n \to \infty} l_n = \infty, \quad \lim_{n \to \infty} h_n = h_\infty \in [0, \infty).
    \]
    Assume $\Lambda : (0, \infty) \rightarrow (0, \infty)$ is a continuous function satisfying
    \[
    \limsup_{h \to 0} \frac{\Lambda(h)}{2h} =: \Lambda_0 < \infty.
    \]
    Suppose $w_n \in C^\infty(\overline{C(l_n, h_n)})$ has zero mean and $K' \times \{0\} \subset \topint(\{w_n = \alpha_n\})$. If $\alpha_n \to \alpha \neq 0$ and
    \begin{equation} \label{eq:lambda_bound}
        \int_{C(l_n, h_n)} |\nabla w_n|^2 \, dx \le \Lambda(h_n) \quad \text{for all } n,
    \end{equation}
    then
    \[
    \alpha^2 \cpct_{h_\infty}(K', \mathbb{R}^{N - 1}) \le 
    \begin{cases}
        \Lambda_0 \quad &\text{if } h_\infty = 0, \\
        \Lambda(h_\infty) \quad &\text{if } h_\infty \in (0, \infty).
    \end{cases}
    \]
\end{lemma}

\begin{proof}
    As in the previous proof, we set $C_n := C(l_n, h_n)$.
    
    \underline{Step 1:} Assume $h_\infty \in (0, \infty)$. Set
    \[
    \bar{w}_n(x_N) := \fint_{B'(0, l_n)} w_n(x', x_N) \, dx', \quad x_N \in (-h_n, h_n).
    \]
    Applying the Sobolev-Poincaré inequality along $(N - 1)$-dimensional cross-sections and integrating in $x_N$, we obtain
    \begin{equation} \label{eq:poincaré_sobolev_2}
        \int_{-h_n}^{h_n} \left(\int_{B'(0, l_n)} |w_n(x) - \bar{w}_n(x_N)|^\frac{2(N - 1)}{N - 3} \, dx'\right)^\frac{N - 3}{N - 1} \, dx_N \le C \int_{C_n} |\nabla w_n(x)|^2 \, dx,
    \end{equation}
    where $C$ depends only on $N$. Since $\bar{w}_n$ has zero mean, Poincaré's and Jensen's inequalities yield
    \[
    \int_{-h_n}^{h_n} |\bar{w}_n|^2 \, dx_N \lesssim h_n^2 \int_{-h_n}^{h_n} |\bar{w}_n'|^2 \, dx_N \le h_n^2 \int_{-h_n}^{h_n} \fint_{B'(0, l_n)} |\partial_N w_n|^2 \, dx' \, dx_N.
    \]
    As a result, we get
    \begin{equation} \label{eq:approximation_of_the_average}
        \int_{-h_n}^{h_n} \left(\int_{B'(0, l_n)} |\bar{w}_n|^\frac{2(N - 1)}{N - 3} \, dx'\right)^\frac{N - 3}{N - 1} \, dx_N = |B'(0, l_n)|^\frac{N - 3}{N - 1} \int_{-h_n}^{h_n} |\bar{w}_n|^2 \, dx_N \lesssim \left(\frac{h_n}{l_n}\right)^2 \int_{C_n} |\nabla w_n|^2 \, dx
    \end{equation}
    Combining \eqref{eq:lambda_bound}, \eqref{eq:poincaré_sobolev_2} and \eqref{eq:approximation_of_the_average}, we see that
    \[
    \limsup_{n \to \infty} \int_{-h_n}^{h_n} \left(\int_{B'(0, l_n)} |w_n|^\frac{2(N - 1)}{N - 3} \, dx'\right)^\frac{N - 3}{N - 1} \, dx_N \lesssim \Lambda(h_\infty) < \infty.
    \]
    Passing to a subsequence (not relabeled), we know that there exist functions 
    \[
    w \in L^2(-h_\infty, h_\infty; L^\frac{2(N - 1)}{N - 3}(\mathbb{R}^{N - 1})), \quad g \in L^2(\mathbb{R}_{h_\infty}^N; \mathbb{R}^N)
    \]
    such that $w_n \rightharpoonup w$ in $L^2(-h_\infty, h_\infty; L^\frac{2(N - 1)}{N - 3}(\mathbb{R}^{N - 1}))$ and $\nabla w_n \rightharpoonup g$ in $L^2(\mathbb{R}_{h_\infty}^N; \mathbb{R}^N)$. Arguing as in Step 1 of Lemma \ref{lm:infinite_height}, we deduce that $w$ is weakly differentiable with $\nabla w = g$ and that $w = \alpha$ weakly on $K' \times \{0\}$. Hence, $w/\alpha$ is an admissible test function in \eqref{eq:extended_capacity_2}. Therefore, by the weak lower semi-continuity of $L^2$-norm and \eqref{eq:lambda_bound},
    \[
    \alpha^2 \cpct_{h_\infty}(K', \mathbb{R}^{N - 1}) \le \int_{\mathbb{R}_{h_\infty}^N} |\nabla w|^2 \, dx \le \Lambda(h_\infty).
    \]

    \underline{Step 2:} Next, we assume $h_\infty = 0$. Set
    \begin{gather*}
        \hat{w}_n(x', x_N) := w_n(x', h_n x_N), \quad (x', x_N) \in B'(0, l_n) \times (-1, 1), \\
        \bar{w}_n(x_N) := \fint_{B'(0, l_n)} \hat{w}_n(x', x_N) \, dx', \quad x_N \in (-1, 1).
    \end{gather*}
    By \eqref{eq:lambda_bound}
    \begin{equation} \label{eq:rescaled_gradient_inequality}
        \int_{-1}^1 \int_{B'(0, l_n)} |\nabla' \hat{w}_n|^2 + \frac{1}{h_n^2} |\partial_N \hat{w}_n|^2 \, dx' \, dx_N = \frac{1}{h_n} \int_{C_n} |\nabla w_n|^2 \, dx \le \frac{\Lambda(h_n)}{h_n}.
    \end{equation}
    Applying the Poincaré-Sobolev inequality to $\hat{w}_n$ and integrating in $x_N$ gives
    \begin{equation} \label{eq:poincaré_sobolev_3}
        \int_{-1}^1 \left(\int_{B'(0, l_n)} |\hat{w}_n(x) - \bar{w}_n(x_N)|^\frac{2(N - 1)}{N - 3} \, dx'\right)^\frac{N - 3}{N - 1} \, dx_N \le C \int_{-1}^1 \int_{B'(0, l_n)} |\nabla' \hat{w}_n(x)|^2 \, dx' \, dx_N,
    \end{equation}
    where $C$ depends only on $N$. It is easy to check that $\bar{w}_n$ has zero mean. From Poincaré's and Jensen's inequalities and \eqref{eq:rescaled_gradient_inequality} we obtain
    \[
    \int_{-1}^1 |\bar{w}_n|^2 \, dx_N \lesssim \int_{-1}^1 |\bar{w}_n'|^2 \, dx_N \le \int_{-1}^1 \fint_{B'(0, l_n)} |\partial_N \hat{w}_n|^2 \, dx' \, dx_N \lesssim \frac{h_n^2}{l_n^{N - 1}}.
    \]
    Then
    \begin{equation} \label{eq:approximation_of_the_average_2}
        \int_{-1}^1 \left(\int_{B'(0, l_n)} |\bar{w}_n|^\frac{2(N - 1)}{N - 3} \, dx'\right)^\frac{N - 3}{N - 1} \, dx_N = |B'(0, l_n)|^\frac{N - 3}{N - 1} \int_{-1}^1 |\bar{w}_n|^2 \, dx_N \lesssim \left(\frac{h_n}{l_n}\right)^2
    \end{equation}
    Thus, \eqref{eq:rescaled_gradient_inequality}, \eqref{eq:poincaré_sobolev_3} and \eqref{eq:approximation_of_the_average_2} imply
    \[
    \limsup_{n \to \infty}\int_{-1}^1 \left(\int_{B'(0, l_n)} |\hat{w}_n|^\frac{2(N - 1)}{N - 3} \, dx'\right)^\frac{N - 3}{N - 1} \, dx_N \le \limsup_{n \to \infty} \frac{\Lambda(h_n)}{h_n} < \infty.
    \]
    Passing to a subsequence (not relabeled), we know that there exist functions $\hat{w} \in L^\frac{2(N - 1)}{N - 3}(\mathbb{R}_1^N)$, $g \in L^2(\mathbb{R}_1^N; \mathbb{R}^N)$ such that $\hat{w}_n \rightharpoonup \hat{w}$ in $L^\frac{2(N - 1)}{N - 3}(\mathbb{R}_1^N)$ and $\nabla \hat{w}_n \rightharpoonup g$ in $L^2(\mathbb{R}_1^N; \mathbb{R}^N)$. Then $\hat{w}$ is weakly differentiable with $\nabla \hat{w} = g$. Moreover, by \eqref{eq:rescaled_gradient_inequality} we have $\partial_N \hat{w}_n \to 0$ in $L^2(\mathbb{R}_1^N)$. Consequently, $\partial_N \hat{w} = 0$ and there exists $w \in \mathcal{C}(\mathbb{R}^{N - 1})$ such that $\hat{w}(\cdot, x_N) = w$ for a.e. $x_N \in (0, 1)$. We can also conclude that $\hat{w} = \alpha$ weakly on $K' \times \{0\}$ as in Step 1 of Lemma \ref{lm:infinite_height}. We now observe that $w/\alpha$ is an admissible test function in \eqref{eq:extended_capacity_3}. Hence, by the weak lower semi-continuity of the $L^2$-norm and \eqref{eq:rescaled_gradient_inequality},
    \begin{multline}
        \alpha^2 \cpct_0(K', \mathbb{R}^{N - 1}) \le \int_{\mathbb{R}^{N - 1}} |\nabla' w|^2 \, dx' = \frac{1}{2} \int_{\mathbb{R}_1^N} |\nabla' \hat{w}|^2 \, dx \\
        \le \liminf_{n \to \infty} \frac{1}{2} \int_{-1}^1 \int_{B'(0, l_n)} |\nabla' \hat{w}_n|^2 \, dx' \, dx_N \le \limsup_{n \to \infty} \frac{\Lambda(h_n)}{2h_n} = \Lambda_0.
    \end{multline}
\end{proof}

\subsection{Proof of Proposition \ref{pr:subcapacitary_growth_1}}

Fix $\theta \in (0, 1)$. By Remark \ref{rmk:global_approximation} it suffices to prove the proposition for smooth functions that equal $1$ in an open neighborhood of $K$. We argue by contradiction. If the claim is false, then there exist sequences $(l_n)$, $(h_n)$ such that
\[
\lim_{n \to \infty} l_n = \lim_{n \to \infty} h_n = \infty, \quad \lim_{n \to \infty} \max \left\{\frac{\ln(l_n)}{h_n}, \frac{h_n}{l_n^2} \right\} = 0,
\]
and there exist functions $v_n \in C^\infty(\overline{C(l_{\varepsilon_n}, h_{\varepsilon_n})})$ such that $K \subset \topint(\{v_n = 1\})$ and
\begin{gather*}
    \int_{C(l_n, h_n)} |\nabla v_n|^2 \, dx < \theta \cpct(K, \mathbb{R}^3) \quad \text{for all } n, \\
    \fint_{C(l_n, h_n)} v_n \, dx =: a_n \to a \le 0 \quad \text{as } n \to \infty.
\end{gather*}
Set $\Lambda := \theta \cpct(K, \mathbb{R}^3)$. Define $w_n := v_n - a_n$ and $\alpha_n := 1 - a_n$. Note that $w_n$ has zero mean, $K \subset \topint(\{w_n = \alpha_n\})$ and
\[
\limsup_{n \to \infty} \int_{C(l_n, h_n)} |\nabla w_n|^2 \, dx \le \Lambda.
\]
Since $\alpha_n \to 1 - a$, Lemma \ref{lm:infinite_height} implies
\[
\cpct(K, \mathbb{R}^3) \le \frac{\Lambda}{(1 - a)^2} = \frac{\theta}{(1 - a)^2} \cpct(K, \mathbb{R}^3) < \cpct(K, \mathbb{R}^3),
\]
which is a contradiction.

\subsection{Proof of Proposition \ref{pr:subcapacitary_growth_2}}

The proof is similar to the proof of Proposition \ref{pr:subcapacitary_growth_1}. Fix $\theta \in (0, 1)$. By Remark \ref{rmk:global_approximation} it suffices to prove the proposition for smooth functions that equal $1$ in an open neighborhood of $K' \times \{0\}$. We argue by contradiction. If the claim is false, then there exist sequences $(l_n)$ and $(h_n)$ such that
\[
\lim_{n \to \infty} l_n = \infty, \quad \lim_{n \to \infty} \frac{h_n}{l_n^{N - 1}} = 0,
\]
and there exist functions $v_n \in C^\infty(\overline{C(l_{\varepsilon_n}, h_{\varepsilon_n})})$ such that $K' \times \{0\} \subset \topint(\{v_n = 1\})$ and
\begin{gather*}
    \int_{C(l_n, h_n)} |\nabla v_n|^2 \, dx < \theta \cpct_{h_n}(K', \mathbb{R}^{N - 1}) \quad \text{for all } n, \\
    \fint_{C(l_n, h_n)} v_n \, dx =: a_n \to a \le 0 \quad \text{as } n \to \infty.
\end{gather*}
Without loss of generality, we assume that $h_\infty := \lim_{n \to \infty} h_n$ exists in $[0, \infty]$. Set $\Lambda(h) := \theta \cpct_h(K', \mathbb{R}^{N - 1})$. Define $w_n := v_n - a_n$ and $\alpha_n := 1 - a_n$. Note that $w_n$ has zero mean and $K' \times \{0\} \subset \topint(\{w_n = \alpha_n\})$. If $h_\infty = \infty$, then by Proposition \ref{pr:limit_at_infinity}
\[
\limsup_{n \to \infty} \int_{C(l_n, h_n)} |\nabla w_n|^2 \, dx \le \limsup_{n \to \infty} \Lambda(h_n) = \theta \cpct(K' \times \{0\}, \mathbb{R}^N).
\]
Since $\alpha_n \to 1 - a$, Lemma \ref{lm:infinite_height} implies
\[
\cpct(K' \times \{0\}, \mathbb{R}^N) \le \frac{\theta}{(1 - a)^2} \cpct(K' \times \{0\}, \mathbb{R}^N) < \cpct(K' \times \{0\}, \mathbb{R}^N),
\]
which is impossible.

By Proposition \ref{pr:continuity_in_h}, $\Lambda$ is a continuous function and $\lim_{h \to 0} \Lambda(h)/(2h) = \theta \cpct_0(K', \mathbb{R}^{N - 1})$. If $h_\infty < \infty$, then Lemma \ref{lm:finite_height} gives
\[
\cpct_{h_\infty}(K', \mathbb{R}^{N - 1}) \le \frac{\theta}{(1 - a)^2} \cpct_{h_\infty}(K', \mathbb{R}^{N - 1}) < \cpct_{h_\infty}(K', \mathbb{R}^{N - 1}),
\]
which is another contradiction.

\section*{Acknowledgments}

This work is funded by the Deutsche Forschungsgemeinschaft (DFG, German Research Foundation) under Germany's Excellence Strategy EXC 2044–390685587, Mathematics Münster: Dynamics–Geometry–Structure. I would like to thank Caterina Ida Zeppieri for suggesting me to work on this problem and for the helpful discussions.

\bibliography{references}

@article {D,
    AUTHOR = {Damlamian, A.},
     TITLE = {Le probl\`eme de la passoire de {N}eumann},
   JOURNAL = {Rend. Sem. Mat. Univ. Politec. Torino},
  FJOURNAL = {Rendiconti del Seminario Matematico (gi\`a{} ``Conferenze di
              Fisica e di Matematica''). Universit\`a{} e Politecnico di
              Torino},
    VOLUME = {43},
      YEAR = {1985},
    NUMBER = {3},
     PAGES = {427--450},
      ISSN = {0373-1243},
   MRCLASS = {35B25 (35J05 35J85)},
  MRNUMBER = {884870},
MRREVIEWER = {Dietrich\ G\"ohde},
}

@article {P,
    AUTHOR = {Picard, C.},
     TITLE = {Analyse limite d'\'equations variationnelles dans un domaine
              contenant une grille},
   JOURNAL = {RAIRO Mod\'el. Math. Anal. Num\'er.},
  FJOURNAL = {RAIRO Mod\'elisation Math\'ematique et Analyse Num\'erique},
    VOLUME = {21},
      YEAR = {1987},
    NUMBER = {2},
     PAGES = {293--326},
      ISSN = {0764-583X,1290-3841},
   MRCLASS = {35B40 (35J05 35J85 49A29 65N15)},
  MRNUMBER = {896245},
MRREVIEWER = {Shuzi\ Zhou},
       DOI = {10.1051/m2an/1987210202931},
       URL = {https://doi.org/10.1051/m2an/1987210202931},
}

@article {CD,
    AUTHOR = {Casado-Díaz, J.},
     TITLE = {Existence of a sequence satisfying {C}ioranescu-{M}urat
              conditions in homogenization of {D}irichlet problems in
              perforated domains},
   JOURNAL = {Rend. Mat. Appl. (7)},
  FJOURNAL = {Rendiconti di Matematica e delle sue Applicazioni. Serie VII},
    VOLUME = {16},
      YEAR = {1996},
    NUMBER = {3},
     PAGES = {387--413},
      ISSN = {1120-7183,2532-3350},
   MRCLASS = {35B27 (35J25)},
  MRNUMBER = {1422390},
MRREVIEWER = {Patrizia\ Donato},
}

@incollection {M,
    AUTHOR = {Murat, F.},
     TITLE = {The {N}eumann sieve},
 BOOKTITLE = {Nonlinear variational problems ({I}sola d'{E}lba, 1983)},
    SERIES = {Res. Notes in Math.},
    VOLUME = {127},
     PAGES = {24--32},
 PUBLISHER = {Pitman, Boston, MA},
      YEAR = {1985},
      ISBN = {0-273-08670-7},
   MRCLASS = {35J25 (35B25 73K10)},
  MRNUMBER = {807534},
}

@incollection {CM1,
    AUTHOR = {Cioranescu, D. and Murat, F.},
     TITLE = {Un terme \'etrange venu d'ailleurs},
 BOOKTITLE = {Nonlinear partial differential equations and their
              applications. {C}oll\`ege de {F}rance {S}eminar, {V}ol. {II}
              ({P}aris, 1979/1980)},
    SERIES = {Res. Notes in Math.},
    VOLUME = {60},
     PAGES = {98--138, 389--390},
 PUBLISHER = {Pitman, Boston, Mass.-London},
      YEAR = {1982},
      ISBN = {0-273-08541-7},
   MRCLASS = {35J05 (35B25)},
  MRNUMBER = {652509},
MRREVIEWER = {M.\ H.\ Protter},
}

@incollection {CM2,
    AUTHOR = {Cioranescu, D. and Murat, F.},
     TITLE = {Un terme \'etrange venu d'ailleurs. {II}},
 BOOKTITLE = {Nonlinear partial differential equations and their
              applications. {C}oll\`ege de {F}rance {S}eminar, {V}ol. {III}
              ({P}aris, 1980/1981)},
    SERIES = {Res. Notes in Math.},
    VOLUME = {70},
     PAGES = {154--178, 425--426},
 PUBLISHER = {Pitman, Boston, Mass.-London},
      YEAR = {1982},
      ISBN = {0-273-08568-9},
   MRCLASS = {35J05 (35B25)},
  MRNUMBER = {670272},
MRREVIEWER = {M.\ H.\ Protter},
}

@article {O,
    AUTHOR = {Onofrei, D.},
     TITLE = {The unfolding operator near a hyperplane and its applications
              to the {N}eumann sieve model},
   JOURNAL = {Adv. Math. Sci. Appl.},
  FJOURNAL = {Advances in Mathematical Sciences and Applications},
    VOLUME = {16},
      YEAR = {2006},
    NUMBER = {1},
     PAGES = {239--258},
      ISSN = {1343-4373},
   MRCLASS = {35J65 (80M50)},
  MRNUMBER = {2253234},
}

@article {CDGO,
    AUTHOR = {Cioranescu, D. and Damlamian, A. and Griso, G. and Onofrei,
              D.},
     TITLE = {The periodic unfolding method for perforated domains and
              {N}eumann sieve models},
   JOURNAL = {J. Math. Pures Appl. (9)},
  FJOURNAL = {Journal de Math\'ematiques Pures et Appliqu\'ees. Neuvi\`eme
              S\'erie},
    VOLUME = {89},
      YEAR = {2008},
    NUMBER = {3},
     PAGES = {248--277},
      ISSN = {0021-7824},
   MRCLASS = {35B27 (35J25 74Q05)},
  MRNUMBER = {2401689},
MRREVIEWER = {Karsten\ Matthies},
       DOI = {10.1016/j.matpur.2007.12.008},
       URL = {https://doi.org/10.1016/j.matpur.2007.12.008},
}

@article {C1,
    AUTHOR = {Conca, C.},
     TITLE = {\'Etude d'un fluide traversant une paroi perfor\'ee. {I}.
              {C}omportement limite pr\`es de la paroi},
   JOURNAL = {J. Math. Pures Appl. (9)},
  FJOURNAL = {Journal de Math\'ematiques Pures et Appliqu\'ees. Neuvi\`eme
              S\'erie},
    VOLUME = {66},
      YEAR = {1987},
    NUMBER = {1},
     PAGES = {1--43},
      ISSN = {0021-7824},
   MRCLASS = {76D07 (35Q10)},
  MRNUMBER = {884812},
MRREVIEWER = {Monique\ Dauge},
}

@article {C2,
    AUTHOR = {Conca, C.},
     TITLE = {\'Etude d'un fluide traversant une paroi perfor\'ee. {II}.
              {C}omportement limite loin de la paroi},
   JOURNAL = {J. Math. Pures Appl. (9)},
  FJOURNAL = {Journal de Math\'ematiques Pures et Appliqu\'ees. Neuvi\`eme
              S\'erie},
    VOLUME = {66},
      YEAR = {1987},
    NUMBER = {1},
     PAGES = {45--70},
      ISSN = {0021-7824},
   MRCLASS = {76D07 (35Q10)},
  MRNUMBER = {884813},
MRREVIEWER = {Monique\ Dauge},
}

@article {K,
    AUTHOR = {Khrabustovskyi, A.},
     TITLE = {Operator estimates for the {N}eumann sieve problem},
   JOURNAL = {Ann. Mat. Pura Appl. (4)},
  FJOURNAL = {Annali di Matematica Pura ed Applicata. Series IV},
    VOLUME = {202},
      YEAR = {2023},
    NUMBER = {4},
     PAGES = {1955--1990},
      ISSN = {0373-3114,1618-1891},
   MRCLASS = {35B27 (35B40 35P05 47A55)},
  MRNUMBER = {4597609},
       DOI = {10.1007/s10231-023-01308-z},
       URL = {https://doi.org/10.1007/s10231-023-01308-z},
}

@article {AP,
    AUTHOR = {Attouch, H. and Picard, C.},
     TITLE = {Comportement limite de probl\`emes de transmission unilateraux
              \`a{} travers des grilles de forme quelconque},
   JOURNAL = {Rend. Sem. Mat. Univ. Politec. Torino},
  FJOURNAL = {Rendiconti del Seminario Matematico (gi\`a{} ``Conferenze di
              Fisica e di Matematica''). Universit\`a{} e Politecnico di
              Torino},
    VOLUME = {45},
      YEAR = {1987},
    NUMBER = {1},
     PAGES = {71--85},
      ISSN = {0373-1243},
   MRCLASS = {49A50 (35B40 49A29)},
  MRNUMBER = {981155},
}

@article {DV,
    AUTHOR = {Del Vecchio, T.},
     TITLE = {The thick {N}eumann's sieve},
   JOURNAL = {Ann. Mat. Pura Appl. (4)},
  FJOURNAL = {Annali di Matematica Pura ed Applicata. Serie Quarta},
    VOLUME = {147},
      YEAR = {1987},
     PAGES = {363--402},
      ISSN = {0003-4622},
   MRCLASS = {35J25 (35B20 35B25 35B40 35J67)},
  MRNUMBER = {916715},
MRREVIEWER = {Dietrich\ G\"ohde},
       DOI = {10.1007/BF01762424},
       URL = {https://doi.org/10.1007/BF01762424},
}

@article {GHV,
    AUTHOR = {Giunti, A. and H\"ofer, R. and Vel\'azquez, J. J.
              L.},
     TITLE = {Homogenization for the {P}oisson equation in randomly
              perforated domains under minimal assumptions on the size of
              the holes},
   JOURNAL = {Comm. Partial Differential Equations},
  FJOURNAL = {Communications in Partial Differential Equations},
    VOLUME = {43},
      YEAR = {2018},
    NUMBER = {9},
     PAGES = {1377--1412},
      ISSN = {0360-5302,1532-4133},
   MRCLASS = {35B27 (35J05 35J25 35R60 60H15)},
  MRNUMBER = {3915491},
MRREVIEWER = {Dan\ Polisevski},
       DOI = {10.1080/03605302.2018.1531425},
       URL = {https://doi.org/10.1080/03605302.2018.1531425},
}

@article {C,
    AUTHOR = {Cortesani, G.},
     TITLE = {Asymptotic behaviour of a sequence of {N}eumann problems},
   JOURNAL = {Comm. Partial Differential Equations},
  FJOURNAL = {Communications in Partial Differential Equations},
    VOLUME = {22},
      YEAR = {1997},
    NUMBER = {9-10},
     PAGES = {1691--1729},
      ISSN = {0360-5302,1532-4133},
   MRCLASS = {35J25 (35B40)},
  MRNUMBER = {1469587},
MRREVIEWER = {I.\ N.\ Tavkhelidze},
       DOI = {10.1080/03605309708821316},
       URL = {https://doi.org/10.1080/03605309708821316},
}

@article {DMFZ,
    AUTHOR = {Dal Maso, G. and Franzina, G. and Zucco, D.},
     TITLE = {Transmission conditions obtained by homogenisation},
   JOURNAL = {Nonlinear Anal.},
  FJOURNAL = {Nonlinear Analysis. Theory, Methods \& Applications. An
              International Multidisciplinary Journal},
    VOLUME = {177},
      YEAR = {2018},
     PAGES = {361--386},
      ISSN = {0362-546X,1873-5215},
   MRCLASS = {49J45},
  MRNUMBER = {3865203},
MRREVIEWER = {Michele\ Carriero},
       DOI = {10.1016/j.na.2018.04.015},
       URL = {https://doi.org/10.1016/j.na.2018.04.015},
}

@book {MK,
    AUTHOR = {Marchenko, V. A. and Khruslov, E. Y.},
     TITLE = {Homogenization of partial differential equations},
    SERIES = {Progress in Mathematical Physics},
    VOLUME = {46},
      NOTE = {Translated from the 2005 Russian original by M. Goncharenko
              and D. Shepelsky},
 PUBLISHER = {Birkh\"auser Boston, Inc., Boston, MA},
      YEAR = {2006},
     PAGES = {xiv+398},
      ISBN = {978-0-8176-4351-5; 0-8176-4351-6},
   MRCLASS = {35B27 (35-02 35J25 35J40 35K20 74Q05 76M50)},
  MRNUMBER = {2182441},
MRREVIEWER = {I.\ Aganovi\'c},
}

@incollection {PV,
    AUTHOR = {Papanicolaou, G. C. and Varadhan, S. R. S.},
     TITLE = {Diffusion in regions with many small holes},
 BOOKTITLE = {Stochastic differential systems ({P}roc. {IFIP}-{WG} 7/1
              {W}orking {C}onf., {V}ilnius, 1978)},
    SERIES = {Lect. Notes Control Inf. Sci.},
    VOLUME = {25},
     PAGES = {190--206},
 PUBLISHER = {Springer, Berlin-New York},
      YEAR = {1980},
      ISBN = {3-540-10498-4},
   MRCLASS = {60J60},
  MRNUMBER = {609184},
MRREVIEWER = {Constantin\ Tudor},
}

@article {GH,
    AUTHOR = {Giunti, A. and Höfer, R. M.},
     TITLE = {Homogenisation for the {S}tokes equations in randomly
              perforated domains under almost minimal assumptions on the
              size of the holes},
   JOURNAL = {Ann. Inst. H. Poincar\'e{} C Anal. Non Lin\'eaire},
  FJOURNAL = {Annales de l'Institut Henri Poincar\'e{} C. Analyse Non
              Lin\'eaire},
    VOLUME = {36},
      YEAR = {2019},
    NUMBER = {7},
     PAGES = {1829--1868},
      ISSN = {0294-1449,1873-1430},
   MRCLASS = {35B27 (35Q35 35R60)},
  MRNUMBER = {4020526},
       DOI = {10.1016/j.anihpc.2019.06.002},
       URL = {https://doi.org/10.1016/j.anihpc.2019.06.002},
}

@article {SZZ,
    AUTHOR = {Scardia, L. and Zemas, K. and Zeppieri, C. I.},
     TITLE = {Homogenisation of nonlinear {D}irichlet problems in randomly
              perforated domains under minimal assumptions on the size of
              perforations},
   JOURNAL = {Probab. Theory Related Fields},
  FJOURNAL = {Probability Theory and Related Fields},
    VOLUME = {192},
      YEAR = {2025},
    NUMBER = {1-2},
     PAGES = {471--544},
      ISSN = {0178-8051,1432-2064},
   MRCLASS = {35B27 (35J15 35J57 49J45 60F15)},
  MRNUMBER = {4914066},
       DOI = {10.1007/s00440-024-01320-1},
       URL = {https://doi.org/10.1007/s00440-024-01320-1},
}

@article {DMD,
    AUTHOR = {Dal Maso, G. and Defranceschi, A.},
     TITLE = {Limits of nonlinear {D}irichlet problems in varying domains},
   JOURNAL = {Manuscripta Math.},
  FJOURNAL = {Manuscripta Mathematica},
    VOLUME = {61},
      YEAR = {1988},
    NUMBER = {3},
     PAGES = {251--278},
      ISSN = {0025-2611,1432-1785},
   MRCLASS = {49A50 (31B20 31B25 35J65)},
  MRNUMBER = {949817},
MRREVIEWER = {Luciano\ Modica},
       DOI = {10.1007/BF01258438},
       URL = {https://doi.org/10.1007/BF01258438},
}

@article {DMG,
    AUTHOR = {Dal Maso, G. and Garroni, A.},
     TITLE = {New results on the asymptotic behavior of {D}irichlet problems
              in perforated domains},
   JOURNAL = {Math. Models Methods Appl. Sci.},
  FJOURNAL = {Mathematical Models and Methods in Applied Sciences},
    VOLUME = {4},
      YEAR = {1994},
    NUMBER = {3},
     PAGES = {373--407},
      ISSN = {0218-2025,1793-6314},
   MRCLASS = {35B27 (35J25 49J20)},
  MRNUMBER = {1282241},
MRREVIEWER = {Alexander\ Belyaev},
       DOI = {10.1142/S0218202594000224},
       URL = {https://doi.org/10.1142/S0218202594000224},
}

@article {CDG,
    AUTHOR = {Casado-Díaz, J. and Garroni, A.},
     TITLE = {Asymptotic behavior of nonlinear elliptic systems on varying
              domains},
   JOURNAL = {SIAM J. Math. Anal.},
  FJOURNAL = {SIAM Journal on Mathematical Analysis},
    VOLUME = {31},
      YEAR = {2000},
    NUMBER = {3},
     PAGES = {581--624},
      ISSN = {0036-1410,1095-7154},
   MRCLASS = {35B27 (35J25 74Q05)},
  MRNUMBER = {1741039},
MRREVIEWER = {Giovanni\ Alberti},
       DOI = {10.1137/S0036141097329627},
       URL = {https://doi.org/10.1137/S0036141097329627},
}

@book {DVJ,
    AUTHOR = {Daley, D. J. and Vere-Jones, D.},
     TITLE = {An introduction to the theory of point processes. {V}ol. {II}},
    SERIES = {Probability and its Applications (New York)},
   EDITION = {Second},
      NOTE = {General theory and structure},
 PUBLISHER = {Springer, New York},
      YEAR = {2008},
     PAGES = {xviii+573},
      ISBN = {978-0-387-21337-8},
   MRCLASS = {60G55 (60-02 60G57)},
  MRNUMBER = {2371524},
MRREVIEWER = {Gail\ Ivanoff},
       DOI = {10.1007/978-0-387-49835-5},
       URL = {https://doi.org/10.1007/978-0-387-49835-5},
}

@book {CSKM,
    AUTHOR = {Chiu, S. N. and Stoyan, D. and Kendall, W. S.
              and Mecke, J.},
     TITLE = {Stochastic geometry and its applications},
    SERIES = {Wiley Series in Probability and Statistics},
   EDITION = {Third},
 PUBLISHER = {John Wiley \& Sons, Ltd., Chichester},
      YEAR = {2013},
     PAGES = {xxvi+544},
      ISBN = {978-0-470-66481-0},
   MRCLASS = {60D05 (52A22 60G55 60G57)},
  MRNUMBER = {3236788},
MRREVIEWER = {Elena\ Villa},
       DOI = {10.1002/9781118658222},
       URL = {https://doi.org/10.1002/9781118658222},
}

@misc{B,
      title={Homogenization of the random {N}eumann sieve problem under minimal assumptions on the size of the perforations}, 
      author={Baştuğ, M.},
      year={2025},
      eprint={2512.14384},
      archivePrefix={arXiv},
      primaryClass={math.AP},
      note={https://arxiv.org/abs/2512.14384}, 
}

@article {ABZ,
    AUTHOR = {Ansini, N. and Babadjian, J.-F. and Zeppieri,
              C. I.},
     TITLE = {The {N}eumann sieve problem and dimensional reduction: a
              multiscale approach},
   JOURNAL = {Math. Models Methods Appl. Sci.},
  FJOURNAL = {Mathematical Models and Methods in Applied Sciences},
    VOLUME = {17},
      YEAR = {2007},
    NUMBER = {5},
     PAGES = {681--735},
      ISSN = {0218-2025,1793-6314},
   MRCLASS = {35B27 (35Q72 49J45 82C24)},
  MRNUMBER = {2325836},
MRREVIEWER = {Marco\ Veneroni},
       DOI = {10.1142/S0218202507002078},
       URL = {https://doi.org/10.1142/S0218202507002078},
}

@article {BFF,
    AUTHOR = {Bhattacharya, K. and Fonseca, I. and Francfort,
              G.},
     TITLE = {An asymptotic study of the debonding of thin films},
   JOURNAL = {Arch. Ration. Mech. Anal.},
  FJOURNAL = {Archive for Rational Mechanics and Analysis},
    VOLUME = {161},
      YEAR = {2002},
    NUMBER = {3},
     PAGES = {205--229},
      ISSN = {0003-9527,1432-0673},
   MRCLASS = {49J45 (74K35)},
  MRNUMBER = {1894591},
MRREVIEWER = {Yi-Chung\ Shu},
       DOI = {10.1007/s002050100177},
       URL = {https://doi.org/10.1007/s002050100177},
}

@article {N,
    AUTHOR = {Ansini, N.},
     TITLE = {The nonlinear sieve problem and applications to thin films},
   JOURNAL = {Asymptot. Anal.},
  FJOURNAL = {Asymptotic Analysis},
    VOLUME = {39},
      YEAR = {2004},
    NUMBER = {2},
     PAGES = {113--145},
      ISSN = {0921-7134,1875-8576},
   MRCLASS = {35B27 (74K35 74Q05)},
  MRNUMBER = {2093896},
MRREVIEWER = {Isabelle\ Gruais},
       DOI = {10.3233/asy-2004-636},
       URL = {https://doi.org/10.3233/asy-2004-636},
}
\bibliographystyle{plain}

\end{document}